\documentclass{amsart}
\usepackage{amssymb}
\usepackage{amsfonts}
\usepackage{amstext}
\usepackage{algorithmic}
\usepackage{algorithm}
\usepackage{graphicx}
\usepackage{epstopdf}
\usepackage[all]{xy}

\numberwithin{equation}{section}
\newtheorem{theorem}{Theorem}[section]
\newtheorem{lemma}[theorem]{Lemma}
\newtheorem{proposition}[theorem]{Proposition}
\newtheorem{corollary}[theorem]{Corollary}

\newtheorem{fact}[theorem]{Fact}
\newtheorem{facts}[theorem]{Facts}
\theoremstyle{definition}
\newtheorem{definition}[theorem]{Definition}
\newtheorem{example}[theorem]{Example}

\theoremstyle{remark}
\newtheorem{remark}[theorem]{\bf{Remark}}

\newtheorem{convention}[theorem]{\bf{Convention}}

\newtheorem{defconv}[theorem]{\bf{Definition-Convention}}

\newtheorem{notation}[theorem]{\bf{Notation}}

\newtheorem*{prooftheorem}{Proof of Theorem}

\newcommand{\Gcal}{{\mathcal{G}}}
\newcommand{\Kcal}{{\mathcal{K}}}
\newcommand{\Lcal}{{\mathcal{L}}}

\newcommand{\Rcal}{{\mathcal{R}}}

\newcommand{\la}{{\triangleright}}

\renewcommand{\ker}{\mathop{\mathrm{ker}}}

\newcommand{\tens}{\otimes}

\renewcommand{\la}{{\triangleright}}

\newcommand{\soc}{\mathop{\mathrm{Soc}}}
\newcommand{\id}{\mathord{\mathrm{id}}}
\newcommand{\mpl}{\mathop{\mathrm{mpl}}}
\newcommand{\sol}{\mathop{\mathrm{sl}}} 
\newcommand{\Sym}{\mathop{\mathrm{Sym}}}

\newcommand{\Ret}{\mathop{\mathrm{Ret}}}
\newcommand{\stu}{\mathbin{\natural}}

\begin{document}

\title[YBE, Braces and Symmetric
groups]{{Set-theoretic solutions of the Yang-Baxter equation, Braces and
Symmetric groups}}


\thanks{The author was partially supported by Grant I 02/18
"Computational and Combinatorial Methods in Algebra and Applications"
 of the Bulgarian National Science Fund, by IH\'{E}S, by The Abdus Salam
International Centre for
Theoretical
Physics (ICTP), Trieste,  and by Max-Planck Institute for Mathematics, Bonn}

\author{Tatiana Gateva-Ivanova}
\address{Institute of Mathematics and Informatics\\
Bulgarian Academy of Sciences\\
Sofia 1113, Bulgaria\\
and American University in Bulgaria\\
Blagoevgrad 2700, Bulgaria}
\email{ tatyana@aubg.bg, tatianagateva@gmail.com
}
\date{\today}
\begin{abstract}
We involve simultaneously the theory of matched pairs of groups
and the theory of braces to study set-theoretic solutions of the
Yang-Baxter equation (YBE). We show the intimate relation between
the notions of a symmetric group (a braided involutive group) and
a left brace, and find new results on  symmetric groups of finite
multipermutation level and the corresponding braces. We introduce
a new invariant of a symmetric group $(G,r)$, \emph{the derived
chain of ideals of} $G$, which gives a precise information about
the recursive process of retraction of $G$. We prove that every
symmetric group $(G,r)$ of finite multipermutation level $m$ is a
solvable group of solvable length $\leq m$. To each set-theoretic
solution $(X,r)$  of YBE we associate two invariant sequences of
symmetric groups: (i) the sequence of its \emph{derived symmetric
groups}; (ii) the sequence of its \emph{derived permutation
groups} and explore these for explicit descriptions of the
recursive process of retraction. We 
find new criteria
necessary and sufficient to claim that $(X, r)$ is
a multipermutation solution.
\end{abstract}
\maketitle

\setcounter{tocdepth}{1}
\tableofcontents

\section{Introduction }

Let $V$ be a vector space over a field $k$. It is well-known that the
``Yang--Baxter equation'' on a linear map $R:V\tens V\to V\tens V$,
the equation
\[ R_{12}R_{23}R_{12}=R_{23}R_{12}R_{23}\]
(where $R_{i,j}$ denotes $R$ acting in the $i,j$ place in $V\tens V\tens V$),
gives rise to a linear representation of the braid group on tensor powers of
$V$. When $R^2=\id$ one says that the solution is involutive, and in this
case one has a representation of the symmetric group on tensor powers.
A particularly nice class of solutions is provided by set-theoretic solutions,
where $X$ is a set and $r:X\times X\to X\times X$ obeys similar relations on
$X\times X\times X$, \cite{Dri}. Of course, each such solution extends linearly to $V=kX$
with matrices in this natural basis having only entries from 0,1 and many
other nice properties.

Associated to each set-theoretic solution are several algebraic
constructions:
the monoid $S(X,r),$ the group $G(X,r)$,  the semigroup algebra
$k S(X,r)= k_R[V]$ generated by $X$ with relations $xy=\cdot r(x,y)$
(where $\cdot$ denotes product in the free semigroup, resp free group)
and the permutation group $\Gcal(X,r) \subset \Sym (X)$ defined by the
corresponding left translations $y\mapsto {}^xy$ for $x\in X$, where
$r(x,y)=({}^xy,y^x)$ (under assumptions which will be given later).

 Lu, Yan and Zhu, \cite{LYZ} proposed a general way of constructing set-theoretical solutions of the
 Yang-Baxter equation using braiding operators on groups, or essentially,  the matched pairs of groups.
In his  survey  Takeuchi \cite{Takeuchi}  gave an introduction to
 the ESS-LYZ theory, reviewing the main results in \cite{LYZ,ESS},  from a matched pair of groups point of view, among them a good way
 to
 think about the
 properties of the group $G(X,r)$ universally generated from $(X,r)$. In particular it is known that the group $G(X,r)$ is itself a
 braided set
 in an induced manner and this is the starting point of these works. Following Takeuchi, \cite{Takeuchi},  in this paper "\emph{a
 symmetric group}"
 means "\emph{a braided group} $(G, \sigma)$ \emph{with an involutive braiding operator} $\sigma$".
In 2006 Rump introduced new algebraic structures called "braces" as a generalization of radical rings.
In a series of  works, see for example \cite{Ru06, Ru07, Ru14},   he suggested systematic study of braces and developed,  "a brace
theory
approach" to cycle sets and set-theoretic solutions of YBE.
Recently the study of braces   intensified and the "braces approach" to YBE produced new interesting results,
 \cite{CJO14, BC, BCJ15, BCJO, Matsumoto, Ru14, Smok, CGIS16, LV16, AGV}, et all.

We propose a new integral approach: to involve simultaneously the theory of matched pairs of groups and the theory of braces in the
study of set-theoretic solutions of  YBE. In the paper we follow  two main directions.
 1. We study general symmetric groups and general braces.
 More specifically, we are interested in the properties of symmetric groups $(G,r)$ which have finite multipermutation level,
 or satisfy various special conditions.  2. We study symmetric sets $(X,r)$
 and the associated symmetric groups $G=G(X,r)$ and $\Gcal= \Gcal(X,r)$, each endowed canonically with a left brace structure.
 In this case we find a strong correlation between certain properties of the generating symmetric set and the properties of $G$ and
 $\Gcal$, respectively, each of them
 considered as a symmetric group and as a brace.
 We find new criteria, each  necessary and sufficient to claim that $(X,r)$ is a multipermutation solution (we write  $m= \mpl
 X< \infty$ in this case).
The paper is organized as follows.
 Section 2 contains preliminary material and collects some known results and  facts which are used throughout the paper.
 In Section 3 we discuss  matched pairs of groups, symmetric groups in the sense of \cite{Takeuchi} and braces.
  Theorem \ref{Thm_BracesBraidedGp} verifies the equivalence between the notion of  a symmetric group  $(G, \sigma)$,  and the notion of
  a left brace $(G, +, \cdot)$.
It is known that each of these two structures on $G$ implies that
$(G, \sigma)$ forms a (non-degenerate) symmetric set in the sense
of \cite{ESS}, see \cite{LYZ, Ru06}. Sections
\ref{sec_derived_chain of ideals} and \ref{sec_symgrG(x,r)} are
central for the paper.  In Section \ref{sec_derived_chain of
ideals} we study the recursive process of retraction of general
symmetric groups and braces. Subsect. 4.1 gives some useful
technical results,
 Subsect. 4.2, discusses briefly the retraction of a general symmetric group $(G, r)$.
 In Subsect. 4.3 we introduce  \emph{the derived chain of ideals of}  $G$ (or abbreviated: \emph{DCI}), study its properties, and show
 that it is an invariant of the group which encodes a precise information about the recursive process of retraction on $(G, r)$,
 see Proposition   \ref{derived_chainProp}. Theorem  \ref{derived_chainThm1} gives some new conditions equivalent to $\mpl(G,r) = m
 <\infty.$ An important application of DCI of a group $(G,r)$ is Theorem
 \ref{slG_theorem} which proves
 that every symmetric group $(G ,r)$ of finite multipermutation level $m$ is a solvable group of solvable length $\leq m$.
In Section \ref{sec_symgrG(x,r)} we describe the retractions of
the symmetric group $(G, r_G)$ associated with a general solution
$(X,r)$ and apply these to study the process of retraction of a
solution in the most general case. Our starting point is the
isomorphism of symmetric groups $\Ret (G, r_G) \simeq (\Gcal,
r_{\Gcal})$, where $\Gcal= \Gcal(X,r)$ is the corresponding
permutation group. To iterate this and describe higher retractions
we introduce the notions of \emph{derived symmetric groups} $(G_j,
r_{G_j})= G(\Ret^j(X, r))$, $j \geq 0,$ and \emph{derived
permutation groups}, $\Gcal_j = \Gcal(\Ret^j(X, r)), j \geq 0,$
related to a solution $(X,r)$.
Theorem \ref{longsequenceprop} verifies our new formula: $\Ret^j G(X, r)\simeq (\Gcal (\Ret^{j-1}(X, r)), j \geq 1$.
 Theorem \ref{Th_mplG_mplX} shows that the symmetric group $G(X,r)$
has finite multipermutation level \emph{iff} $(X, r)$ is a
multipermutation solution. In this case $\mpl \Gcal = m-1\leq \mpl
X \leq \mpl G=m$. Moreover, if $(X,r)$ is  a square-free solution
of arbitrary cardinality, then $G$ and $X$ have the same
multipermutation level, $\mpl G= \mpl X$. In Section
\ref{sec_two-sided_brace}  we study the brace $(G, +, \cdot)$
canonically associated to a symmetric set $(X,r)$, where
$G=G(X,r)$. Theorem \ref{VIPthm} shows that conditions like "$G$
is a two-sided brace" or "$(G, r_G)$ is square-free" are too
rigid,
 each of them implies that $(X,r)$
is a trivial solution. This motivates our study of some, in
general,  milder conditions on symmetric groups and their
associated left braces. In Sec \ref{sec_actions} we introduce
special conditions on the actions of a symmetric group $(G, r)$
upon itself, such as \textbf{lri}, or \textbf{Raut} and study the
effect of these on the properties of $(G, r)$ and the associated
brace $(G, +,\cdot)$. In Section \ref{sect_TheoremGcal} we study
square-free solutions $(X,r)$ with condition \textbf{lri} on a
derived symmetric group, or on a derived permutation group.
Theorem \ref{Thm_main} shows that the symmetric group $G=G(X,r)$
of a nontrivial square-free solution $(X,r)$ satisfies
\textbf{lri} \emph{iff} $(X,r)$ is a multipermutation solution of
level $2$. More generally, $\mpl X = m \geq 2$ \emph{iff} its
derived group $(G_{m-2}, r_{G_{m-2}})$ satisfies \textbf{lri}.
Theorem \ref{Thm_gcal(X,r)} verifies that $\mpl (X,r)< \infty$,
whenever $(X,r)$ is a finite square-free solution whose
permutation group
  $(\Gcal, r_{\Gcal})$ satisfies \textbf{lri}. The result is generalized by Corollary \ref{Cor_main}.

\section{Preliminaries on set-theoretic solutions of YBE}
\label{Preliminaries}

During the last three decade the study of
set-theoretic solutions and related
structures notably intensified, a relevant selection of works for the
interested reader
is \cite{W, TM, ESS, LYZ, GI04, Carter, Ru05, Takeuchi, GIM08, GI11, GI12, GIC,  Dehornoy15, Matsumoto, CJO14, Ru14, V15, MS}, et all.
In this section we recall basic notions and
results which will be used in the paper.
 We shall use the terminology,  notation and some results from
\cite{GI04,  GIM07, GIM08, GIM11, GIC}.

 \begin{definition}
Let $X$ be a nonempty set (not necessarily finite) and let
$r: X\times X \longrightarrow X\times X$ be a bijective map.
In
this case we use notation $(X, r )$ and refer to it as \emph{a quadratic set}.
The image
of $(x,y)$ under $r$ is presented as
\[
r(x,y)=({}^xy,x^{y}).
\]
This formula defines a ``left action'' $\Lcal: X\times X
\longrightarrow X,$ and a ``right action'' $\Rcal: X\times X
\longrightarrow X,$ on $X$ as:
$\Lcal_x(y)={}^xy$, $\Rcal_y(x)= x^{y}$, for all $x, y \in X$.
(i) $r$ is \emph{non-degenerate}, if
the maps $\Lcal_x$ and $\Rcal_x$ are bijective for each $x\in X$.
(ii) $r$ is involutive if $r^2 = id_{X\times X}$.
(iii) $(X,r)$ is said to be \emph{square-free} if $r(x,x)=(x,x)$ for all $x\in
X.$
(iv) $r$ is \emph{a set-theoretic solution of the
Yang--Baxter equation} (YBE) if  the braid relation
\[r^{12}r^{23}r^{12} = r^{23}r^{12}r^{23}\]
holds in $X\times X\times X,$  where  $r^{12} = r\times\id_X$, and
$r^{23}=\id_X\times r$. In this
case we  refer to  $(X,r)$ also as   \emph{a braided set}.
A braided set $(X,r)$ with $r$ involutive is called \emph{a
symmetric set}.
\end{definition}

\begin{remark}
\label{rem_Invol-Sqfree}
Let $(X,r)$  be a  quadratic set.
The map  $r$ is involutive   \emph{iff} the actions satisfy:
\begin{equation}
\label{involeq}
{}^{{}^uv}{(u^v)}= u, \; ({}^uv)^{u^v} = v, \; \forall u, v \in X.
\end{equation}
Clearly, an involutive quadratic set $(X,r)$ is square-free
\emph{iff} ${}^aa= a, \;  \; \forall  a \in X$. For  a finite
symmetric set $(X,r)$ one-sided non-degeneracy implies
non-degeneracy.
\end{remark}
\begin{convention}
In this paper we shall consider only  the case when $r$ is non-degenerate. By
"\emph{solution}" we mean a non-degenerate symmetric set $(X,r)$,
where $X$  is a set of arbitrary cardinality. We shall also refer
to it as "\emph{a symmetric set}", keeping the convention that we
consider only non-degenerate symmetric sets. It is known that in this case $X$ is embedded in $G(X,r)$.
We do not assume
(unless stated explicitly)
 that  $(X,r)$ satisfies additional special conditions on the actions.
\end{convention}
To each quadratic set $(X,r)$  we associate canonical algebraic
objects  generated by $X$ and with quadratic defining relations
$\Re=\Re(r)$, defined by  $xy=zt \in \Re(r)$, \emph{iff} $r(x,y) =
(z,t)$, and $(x,y)\neq (z,t)$.
(i) The monoid  $S =S(X, r) =
\langle X; \Re \rangle$, with a set of generators $X$ and a set of
defining relations $ \Re(r),$ is called \emph{the monoid
associated with $(X, r)$}. (ii) The \emph{group $G=G(X, r)$
associated with} $(X, r)$ is defined as $G=G(X, r)={}_{gr} \langle
X; \Re \rangle$. (iii) Furthermore, to each non-degenerate braided
set $(X,r)$ we also associate \emph{a permutation group}, denoted
$\Gcal= \Gcal(X,r)$, see Definition \ref{Gcaldef}.
 If $(X,r)$ is a solution, then $S(X,r)$, resp. $G(X,r)$,
 $\Gcal(X,r)$ is called the
{\em{Yang--Baxter monoid}}, resp. the {\em{Yang--Baxter group}},
resp. the {\em{Yang--Baxter permutation group}},  or shortly
\emph{the YB permutation group}, associated to $(X,r)$. The YB
groups $G(X,r)$ and  $\Gcal(X,r)$ will be of particular importance
in this paper. In \cite{ESS} $G(X,r)$ is called \emph{the
structure group of} $(X,r)$, whenever $(X,r)$ is a solution.

Each element $a \in G$ can be presented as a monomial
\begin{equation}
\label{keyeq2} a = \zeta_1\zeta_2 \cdots \zeta_n,\quad \zeta_i \in X
\bigcup X^{-1}.
\end{equation}
We shall consider  \emph{a reduced form of} $a$, that is a
presentation (\ref{keyeq2}) with minimal length $n$. By convention, \emph{length
of} $a$,
denoted by $|a|$ means the length of a reduced form of $a$.

\begin{example}
\label{trivialsolex} For arbitrary set $X$, $|X|\geq 2,$ denote by
$\tau_X = \tau$ the flip map $\tau(x,y)= (y,x)$ for all $x,y \in
X.$ Then  $(X, \tau)$ is a solution called \emph{the trivial
solution} on $X$. Clearly, an involutive quadratic set $(X,r)$ is
the trivial solution if and only if ${}^xy =y$,  for all $x,y \in
X,$ or equivalently $\Lcal_x= \id_X$ for all $x\in X$. In this
case $S(X,r)$ is the free abelian monoid generated by $X$,
$G(X,r)$ is the free abelian group, and $\Gcal(X,r) = \{\id_X\}$
is the trivial group.
\end{example}

It is well
known that if $(X,r)$   is a braided
set, and $G=G(X,r)$, then the equalities
\[
\begin{array}{lclc}
 {\bf l1:}\quad& {}^x{({}^yz)}={}^{{}^xy}{({}^{x^y}{z})},
 \quad\quad\quad
 & {\bf r1:}\quad&
{(x^y)}^z=(x^{{}^yz})^{y^z},
\end{array}\]
 hold for all $x,y,z \in X$. So
the assignment $x \longrightarrow \Lcal_x$ for
$x\in X$ extends canonically to a group homomorphism
$\Lcal: G \longrightarrow \Sym(X)$,
which defines  the \emph{canonical left action} of $G$  on the set  $X$.
Analogously,  there is a \emph{canonical right action} of
$G$  on~$X$.

\begin{definition}
\label{Gcaldef} Let $(X,r)$ be a non-degenerate braided set, let
$\Lcal:G(X,r) \longrightarrow \Sym(X)$ be the canonical group
homomorphism defined via the left action. We set $K := \ker
\Lcal$. The image $\Lcal(G(X,r))$ is denoted by $\Gcal=\Gcal(X,r)$
and called \emph{the YB-permutation group of} $(X,r),$ \cite{GI04,
GIC}). We shall often refer to it simply as "$\Gcal(X,r)$". The
group $\Gcal$ is generated by the set $\{\Lcal_x \mid x \in X\}$.
\end{definition}

In a series of works
we have shown that the combinatorial properties of a solution $(X,r)$ are closely related
to the algebraic properties of its YB structures.  Especially interesting are  solutions satisfying some of the following
conditions.
\begin{definition} \cite{GIM08, GI04}
\label{lri&cl}
Let $(X,r)$ be a quadratic set.
\begin{enumerate}
\item  The following are called \emph{cyclic conditions on}  $X$.
\[\begin{array}{lclc}
 {\rm\bf cl1:}\quad&{}^{(y^x)}x= {}^yx, \quad\text{for all}\; x,y \in
 X;
 \quad\quad&{\rm\bf cr1:}\quad &x^{({}^xy)}= x^y, \quad\text{for all}\; x,y
 \in
X;\\
 {\rm\bf cl2:}\quad&{}^{({}^xy)}x= {}^yx,
\quad\text{for all}\; x,y \in X; \quad\quad &{\rm\bf cr2:}\quad
&x^{(y^x)}= x^y, \quad\text{for all}\; x,y \in X.
\end{array}\]
\item Condition \textbf{lri} on $(X,r)$ is defined as
 \[ \textbf{lri:}\quad
\quad ({}^xy)^x= y={}^x{(y^x)} \;\text{for all} \quad
x,y \in X.\]
In other words \textbf{lri} holds if and only if
$(X,r)$ is non-degenerate, $\Rcal_x=\Lcal_x^{-1}$, and $\Lcal_x =
\Rcal_x^{-1}$
\end{enumerate}
\end{definition}
It is known that  \emph{every  square-free solution  $(X,r)$
satisfies the cyclic conditions \textbf{cc} and condition
\textbf{lri}, so it is  uniquely determined by the left action:
$r(x,y) = (\Lcal _x(y), \Lcal^{-1}_y(x))$}, see \cite{GI04,
GIM08}. Solutions with \textbf{lri} are studied in sections 7 and
8.

The cyclic conditions was introduced by the author in \cite{GI94, GI96}, and was crucial for the proof that every binomial skew
polynomial algebra defines canonically (via its relations)
a set-theoretic solution of YBE, see \cite{TM}.  Algebraic objects with relations
satisfying the cyclic condition were studied later, \cite{GIJO, JO04, CJO06}.


\begin{remark} \cite{GIM08}, Proposition 2.25.
Suppose $(X,r)$ is a quadratic set.  Then any two of the following conditions
 imply the remaining third
 condition: (i) $(X,r)$ is involutive;  (ii) $(X,r)$ is non-degenerate and
 cyclic; (iii) $(X,r)$ satisfies \textbf{lri}.
 \end{remark}
 \begin{corollary}
 \label{cor_lri_cc}
 Let $(X,r)$ be a solution of arbitrary cardinality. Then $(X,r)$ satisfies condition \textbf{lri} if and only if it
satisfies the cyclic
 conditions \textbf{cc}.
 \end{corollary}
The notions of \emph{retraction of a symmetric set} and \emph{multipermutation solutions}
were introduced in the general case
in \cite{ESS}. The active study of multipermutation solutions started in 2004, when
the author
conjectured that every finite square-free solution $(X,r)$ has a finite multipermutation level, \cite{GI04}.
Multipermutation
solutions were studied  in \cite{GI04,
GIM07,
Ru07,  GIM08, GIM11, GIC, CJO10, CJO14}, et all.

Let
$(X,r)$ be a non-degenerate symmetric set. An equivalence relation
$\sim$ is defined on $X$ as
 $x \sim y$ if and only if
$\Lcal_x = \Lcal_y$.  In this case we also have $\Rcal_x = \Rcal_y,$ see
\cite{ESS}.
Denote by  $[x]$ the equivalence class of $x\in X$ and by  $[X]=
X/_{\sim}$ the set of equivalence classes.
The following  can be extracted
from \cite{GIM08}.
The left and the right actions of $X$ onto itself naturally induce
left and right actions on the retraction $[X],$ via
\[
{}^{[\alpha]}{[x]}:= [{}^{\alpha}{x}],\quad [\alpha]^{[x]}:=
[\alpha^x], \; \forall\; \alpha, x \in X.
\]
The new actions define a canonical map $r_{[X]}: [X]\times[X]
\longrightarrow [X]\times[X]$, where $r_{[X]}([x], [y])=
({}^{[x]}{[y]}, [x]^{[y]}).$ Then $([X], r_{[X]})$ is a
non-degenerate symmetric set. Note that the canonical map $\mu : X
\longrightarrow [X], x \mapsto [x]$ is \emph{a braiding-preserving
map}, (epimorphism of solutions) that is, $(\mu\times \mu) \circ r
= r_{[X]}\circ (\mu\times \mu)$. Furthermore, various special
properties of $(X,r)$ are inherited by its retraction: (i)
$(X,r)\; \text{is}\; {\bf lri} \Longrightarrow([X],
r_{[X]})\;\text{is}\; {\bf lri}.$ (ii) $(X,r)\; \text{satisfies
the cyclic conditions} \Longrightarrow   ([X], r_{[X]})$ does so.
(iii) $ (X,r)\; \text{is square-free} \Longrightarrow ([X],
r_{[X]}) \; \text{is square-free}.$
The solution $\Ret(X,r)=([X],r_{[X]})$ is
called the \emph{retraction of $(X,r)$}. For all integers $m \geq
0$, $\Ret^m(X,r)$ is defined recursively as $\Ret^0(X,r)=(X,r), \; \Ret^1(X,r)=\Ret(X,r), \;  \Ret^m(X,r)=
\Ret(\Ret^{m-1}(X,r)).$
$(X,r)$ is  \emph{a multipermutation solution of level} $m$,
 if $m$ is
the minimal number (if any), such that $\Ret^m(X,r)$ is the trivial
solution on a set of one element. In this case we write $\mpl(X,r)=m$.
By definition $(X,r)$ is \emph{a multipermutation solution of level}
$0$ if and only if $X$ is a one element set.

The trivial solution $(X, \tau)$ is \emph{a multipermutation solution of
level} $1$. More generally, a symmetric set $(X,r)$ has
 $\mpl X =1$ if and only if it is an involutive permutation solution of
 Lyubashenko, see Example \ref{exlri}.

 We made the following conjecture in \cite{GI04}:
  \emph{Every finite  square-free solution  $(X,r)$ is retractable, and therefore
 it is a multipermutation solution
  with $\mpl(X,r)< |X|$.}

Several independent results verify the Conjecture for the case when the
YB-permutation group $\Gcal(X,r)$ is abelian, see \cite{CJO10, CJO14},
and \cite{GIC}, where we give also an upper bound for the multipermutation level of $X$:
$\mpl X$ is at most the number of $\Gcal$-orbits of
$X$. The Conjecture is also true for all square-free solutions $(X,r)$ of cardinality $\leq 7$.
Leandro Vendramin was the first who constructed an infinite family of
square free-solutions which are not multipermutation solutions, see \cite{V15}, hence the Conjecure is not true for general finite
square-free solutions.  Moreover Vendramin's counter example
of minimal order ($|X|= 8$) satisfies  both: (i) $(X,r)$ is an irretractable square-free solution of finite order $8$ and (ii) $(X,r)$
is a strong twisted union of two solutions of multipermutation level $2$, see for details Remark \ref{Rem_ex_Ve}.
A natural question arises:

\emph{Where is the borderline between the classes  of retractable and the nonretractable solutions? More generally,
what can be said about symmetric sets which are not multipermutation solutions?}

 Some answers are given by  Proposition   \ref{derived_chainProp}, which presents a precise information about the recursive process of
 retraction on $(G, r)$ in terms of its derived chain of ideals.

It is also natural  \emph{to search new conditions which imply
$\mpl(X,r)< \infty$}  and this is one of the main themes discussed
in the paper.

\section{Matched pairs of groups, symmetric groups, and braces}
\label{MP_brace}

\subsection{Matched pairs of groups, braided groups, and symmetric groups}
The notion of matched pairs of groups in relation to group
factorisation has a classical origin. For finite groups it was
used in the 1960's in the construction of certain Hopf algebras
\cite{KP66}. Later such 'bismash product' or 'bicrossed product'
Hopf algebras were rediscovered by Takeuchi \cite{Tak1}  and Majid
\cite{Ma90, Ma90a, Ma90b}. By now there have been many works on
matched pairs in different contexts, see \cite{Ma:book}  and
references therein. This notion was used by Lu, Yan, and Zhu  to
study the set-theoretic solutions of YBE and the associated
'braided group', see \cite{LYZ} and the  survey \cite{Takeuchi}.
Matched pairs of monoids were studied in \cite{GIM08}, and used
to characterise general solutions of YBE and extensions  of
solutions in terms of matched pairs of their associated monoids.
\begin{definition}
\label{matchedpairsdef1} \cite{Takeuchi, Ma:book} \emph{A matched
pair of groups} is a triple $(S, T, \sigma)$, where $S$ and $T$
are groups and $\sigma: T\times S \longrightarrow S\times T, \quad
\sigma(a, u) = ({}^au, a^u)$ is a bijective map satisfying the
following conditions
\begin{equation}
\label{matchedpairseq}
\begin{array}{llll}
\textbf{ML0}:\quad&{}^a1 =1, \; {}^1u =u, \quad &\textbf{MR0}:\quad & 1^u =1,
\; a^1 =a,
\\
\textbf{ML1}:&{}^{ab}u ={}^a{({}^bu)},\quad &\textbf{MR1}:& a^{uv} =(a^u)^v,
\\
\textbf{ML2}:&{}^a{(u.v)} =({}^au)({}^{a^u}v),\quad &\textbf{MR2}:& (a.b)^u
=(a^{{}^bu})(b^u), \\
\end{array}
\end{equation}
for all $ a,b \in T, u,v\in S.$
\end{definition}
\begin{definition}
\label{def_SymmetricGr}
(1) \cite{LYZ}
\emph{A braided group} is a pair  $(G, \sigma)$, where $G$ is a group and
$\sigma: G \times G \longrightarrow G \times G$  is a map
such that the triple $(G, G, \sigma)$  forms a matched pair of groups, and the
left and the right  actions induced by $\sigma$  satisfy
\emph{the compatibility condition}:
\begin{equation}
\label{compatibilityeq} \textbf{M3}: \quad uv =({}^uv)(u^v)
\;\text{is an equality in}\; G, \quad \forall u, v \in G.
\end{equation}

(2) \cite{Takeuchi} A braided group $(G, \sigma)$, with an
involutive braiding operator $\sigma$ is called a \emph{symmetric
group}   (\emph{in the sense of Takeuchi}).
\end{definition}
The following results can be extracted from \cite{LYZ}, Theorems 1,  2, and
4.
\begin{facts}
\label{factLYZ}
(1)  The pair $(G, \sigma)$ is a braided group \emph{iff} conditions
\textbf{ML1}, \textbf{MR1} and the compatibility condition \textbf{M3}
(for $S=T=G$)  are satisfied.
(2)  If $(G, \sigma)$ is a braided group, then $ \sigma$ satisfies the braid
relations and is non-degenerate. So  $(G, \sigma)$ forms a
non-degenerate
braided set in the sense of \cite{ESS}.
(3)  Let $(X,r)$ be a non-degenerate braided set, let $G = G(X,r)$ be the
associated YB-group, and let $i : X\longrightarrow G(X,r)$ be the canonical map. 
 Then there is unique braiding operator
$r_G: G \times G \longrightarrow G \times G $, such that $r_G\circ (i\times i)= (i\times i)\circ r_G$.
In particular, if $X$ is embedded in $G$
then the restriction of $r_G$ on $X\times X$ is exactly the map $r$.
Furthermore, $(r_G)^2 = id_{G\times G}$ \emph{iff} $r^2 = id_{X\times X}$, so $(G, r_G)$ is a
symmetric group \emph{iff} $(X,r)$ is a symmetric set.
\end{facts}
We shall refer to the group $(G,r_G)$ as \emph{the symmetric group
associated to} $(X,r)$.

\subsection{Braces, definition and basic properties}
\label{brace}
In 2006 Rump introduced  and initiated a systematic study of a new algebraic structure called \emph{a brace} as a
generalization
of radical rings.
In a series of papers he developed
  "a brace theory approach" to cycle sets and set-theoretic solutions of YBE,
\cite{Ru06, Ru07, Ru14}, et all. He proved that given a brace $(G, +, \cdot)$ one can define explicitly a
left and a right actions on $G$  (considered as a set),
which define a non-degenerate involutive set-theoretic solution of YBE denoted by
$(G, r)$, see   \cite{Ru07} and also   \cite{CJO14}.
 Rump has shown that for every non-degenerate symmetric set $(X,r)$ one can define canonically an operation $+$ on the YB group
 $G=G(X,r)$, which makes $(G, +, \cdot)$ a
left brace, and therefore a set-theoretic solution of YBE.
Recently the number of works on braces increased rapidly and the "braces approach" to YBE produced new interesting results,
 \cite{CJO14, BC, LV16, Matsumoto, Ru14, BCJ15,  Smok, CGIS16,  LV16,  AGV}, et all.
In this paper we work with a definition of a brace, given explicitly  in
\cite{CJO14}. It is equivalent
to Rump's original definition, \cite{Ru06, Ru07}.
\begin{definition}
\label{defbrace}
Let $G$ be a set  with two operations "$\cdot$" and "$+$"
such that  $(G,\cdot)$ is a group, and $(G,+)$ is an abelian  group.
(i)  $(G, +, \cdot)$ is called \emph{a left brace} if
\begin{equation}
\label{eqleftbrace}
a(b+c)+a = ab+ac, \quad\text{for all}\; a,b,c \in G.
\end{equation}
(ii)  $  (G, +, \cdot)$ is \emph{a right brace} if
$(a+b)c+c = ac+bc, \quad\text{for all}\; a,b,c \in G$.
(iii) $  (G, +, \cdot )$ is \emph{a two-sided brace} if
(i) and (ii) are in force.
The groups $(G,\cdot)$ and $(G,+)$ are called, respectively, \emph{the
multiplicative }and \emph{the additive groups} \emph{of the
brace.}
\end{definition}
Denote by $0$ and $e$, respectively,  the neutral elements
with respect to the two operations
"$+$" and "$\cdot$" in $G$. Setting $b=0$ in  (\ref{eqleftbrace}) one obtains $0 = e$
in $G$.
In any left brace $(G,+,\cdot)$ one defines another operation $*$ by the
rule
\begin{equation}
\label{eqleftbrace*}
a*b =a\cdot b-a-b,\quad \forall \;   a,b\in G.
\end{equation}
In general, the operation $*$ in a left brace $G$ is not associative,
but it is left distributive with respect to the sum, that is
\[a*(b+c)=a*b+a*c, \quad \forall\; a,b,c\in B.\]
It was proven in \cite{Ru07}  that
if $(G, +, \cdot)$ is a two-sided brace, then the operation $*$
makes $(G, +, *)$ a Jacobson radical ring.
The following facts can be extracted from Rump's works  \cite{Ru06, Ru07}.
\emph{(1)  If $(G, +, \cdot)$ is a left brace, then there is a canonically associated solution  $(G,r)$  uniquely determined by the
operations in
$G$. (2) If G is a finite non-trivial two-sided brace, then
the (non-degenerate) symmetric set (G, r ) associated to G is a multipermutation solution. }

We shall prove  that,  more generally, every symmetric group $(G, r)$ (in the sense of Takeuchi) has a left brace
structure, and conversely, on every left brace $G$, one can define a braiding operator $r$, so that $(G,r)$ is a symmetric group, and
therefore it is a solution to YBE, see Theorem \ref{Thm_BracesBraidedGp} (this is independent of Rump's result).

Suppose $(G,  +  , \cdot)$  is a set  with two operations,
such that  $(G,\cdot)$ is a group, and $(G,+)$ is an abelian  group,
but no more restrictions on $G$ are imposed.
We define
a ``left action'' $G\times G \longrightarrow G, $ on $G$ as:
\begin{equation}
\label{LcalRcal} \Lcal_a(b)={}^ab := ab -a, \quad  a, b \in G.
\end{equation}
It is straightforward that the map $\Lcal_a$ is bijective for
each $a \in G$, so we have a map (of sets)
\[\Lcal: G \longrightarrow \Sym(G), \quad a\mapsto \Lcal_a, \]
where $\Sym (G)$ denotes the symmetric group on the set $G$.

We  shall verify that $G$ is a left brace \emph{iff } the map $\Lcal$  defines
a left action of the multiplicative group  $(G, \cdot)$ upon itself.
In this case  the equality (\ref{eqlbr3}) given below  presents  the
 operation "$+$" in terms of the left
 action and the operation
 "$\cdot$".
 The following proposition contains
 formulae which will be used throughout the paper. The statement extends results of Rump, see also
\cite{CJO14}, Lemma 1. More precisely,
the implications (1)$\Longrightarrow$ (2),  (1)$\Longrightarrow$
 (3) and (1) $\Longrightarrow$
 (4) can be extracted from  \cite{Ru07},  see also
 \cite{CJO14}.
We suggest a direct and independent verification of the equivalence of the four
 conditions.

\begin{proposition}
 \label{Thm_braces1}
Let $(G, +, \cdot )$ be a set with two operations,
"$\cdot$" and "$+$"  such that
$(G, \cdot)$ is a group
and $(G,+)$ is an abelian  group.
Assume that the map
$\Lcal: G \longrightarrow \Sym(G), \quad a\mapsto \Lcal_a, $
is defined via
(\ref{LcalRcal}).
The following conditions are equivalent.
(1) $(G, +, \cdot)$ is a left brace, i.e. condition
(\ref{eqleftbrace}) holds.
(2) $G$ satisfies:
\begin{equation}
\label{eqleftbrace1}
a(b-c)-a = ab-ac, \quad \forall  a,b,c \in G.
\end{equation}
(3) The map $\Lcal$ determines a left action of the multiplicative group
$(G,\cdot)$ on the set $G$, that is
\begin{equation}
\label{eq1.iii}
{}^{(ab)}c = {}^a({}^bc),\quad \forall\; a,b,c \in G.\end{equation}
(Equivalently, $\Lcal:
(G, \cdot) \longrightarrow \Sym(G),\; a \mapsto \Lcal_a$,
is a homomorphism of groups).
 In this case
 the  operation "$+$" can be presented as
\begin{equation}
\label{eqlbr3}
a+b = a({}^{a^{-1}}b), \; \text{or equivalently}, \;  a+ {}^ab = ab, \quad \forall a,b \in G.
\end{equation}
(4) For every $a\in G$, the map  $\Lcal_a$ is an automorphism of the additive
group $(G, +)$ that is  condition \textbf{Laut} defined below is in force:
\begin{equation}
\label{Laut}
\textbf{Laut}:\quad \quad \quad {}^a(b+c) = {}^ab+ {}^ac, \quad  \forall \; a,b,c \in G.
\end{equation}
 In this case the map $\Lcal: (G, \cdot) \longrightarrow Aut
 (G, +)$
is a homomorphism of groups.
\end{proposition}
\begin{proof}
(1) $\Longrightarrow$ (2). Assume (1) holds. Let $a,b,c \in G$. We set $d= b-c$,  so $b=c+d$,
and the left brace condition   $ac + ad   =a(c+d) + a$, see (\ref{eqleftbrace}),  can be written as $ac+
a(b-c)-a= ab$
or, equivalently, $a(b-c)-a = ab-ac$, as desired.
(2) $\Longrightarrow$ (1).
Suppose, $a,c, d \in G,$  set $b = d+c$, so $d = b-c.$
The equality  (\ref{eqleftbrace1}) is in force for $a,b,c$, hence replacing $b-c$ with $d$
we obtain  $ad-a = a(c+d)-ac$, or equivalently
$a(c+d) +a  = ac+ad$, as desired.
(3) $\Longrightarrow$ (2). The equalities
${}^{ab}c = (ab)c-ab$ and $ {}^{a}{({}^bc)}=  a(bc-b)-a$
show that  $\Lcal$ is a left action (i.e.
 (\ref{eq1.iii}) holds)  if and only if
$(ab)c-ab= a(bc-b)-a$, for all $a,b,c \in G$.
Now choose arbitrary $a,b,c\in G$  and set $d = bc$ in the above equality to deduce that it is equivalent to
$ad-ab = a(d-b)-a, \quad\forall a,b,d \in G,$ which  is exactly condition (\ref{eqleftbrace1}).
(\ref{eq1.iii}) $\Longrightarrow$ (\ref{eqlbr3}). Assume  (\ref{eq1.iii}) is in force. We have to show that the operation $+$
and the left action defined via (\ref{LcalRcal}) are related
by the identity (\ref{eqlbr3}).
We set
$c = {}^{a^{-1}}b$.
Then
$ ac-a = {}^ac
= {}^a{({}^{a^{-1}}b)}  = {}^{(a.a^{-1})}b = {}^1b=b,$ and therefore
$a+b=ac = a({}^{a^{-1}}b)$, as desired.
(4) $\Longleftrightarrow$ (1).
Suppose $a,b,c \in G.$ It follows from the definition of the left action that the equality
${}^a(b+c) = {}^ab+{}^ac$
is equivalent to
$a(b+c)-a = (ab-a)+ (ac-a)$,
which, after canceling $-a$ from both sides implies (and is equivalent to)  the equality  (\ref{eqleftbrace}) defining a left brace.
\end{proof}

Note that  the left brace $(G, +, \cdot)$ is a two-sided brace \emph{iff}
\begin{equation}
\label{eqrightbracea}
(a+{}^ab)c+c = ac+({}^ab)c, \quad \forall\; a,b,c \in G.
\end{equation}
Lemma \ref{prop_twosidedbrace}, (\ref{prop_twosidedbrace_eq})
provides an identity in the multiplicative group  $(G, \cdot)$,
which is equivalent to  (\ref{eqrightbracea}).

\subsection{The close relation between symmetric groups and braces}
\begin{theorem}
  \label{Thm_BracesBraidedGp}  The following two structures on a group $(G, \cdot)$ are
  equivalent.
 \begin{enumerate}
\item
The pair $(G, \sigma)$ is a symmetric group, i.e. a braided group with
an involutive braiding operator $\sigma$.
\item
$(G,+, \cdot )$ is a left brace.
\end{enumerate}
Furthermore, each of these conditions implies that $(G, \sigma)$
forms a non-degenerate symmetric set, so $(G, \sigma)$ is a
solution of YBE.
    \end{theorem}
\begin{proof}
(1) $\Longrightarrow$ (2). Define on $G$ a second operation "$+$" via the
formula
\begin{equation}
\label{braceq1}
a+b := a({}^{a^{-1}}b),  \; \forall\; a,b \in G.
\end{equation}
This is equivalent to
$ab  = a+{}^ab,   \;\forall\;  a,b \in G$, which is exactly (\ref{eqlbr3}).
We shall prove that the operation $+$ is commutative and associative.
By hypothesis $\sigma$ is involutive, it follows then from  (\ref{involeq})  that
the actions satisfy:
\begin{equation}
\label{braceq4}
a^b = {}^{({}^ab)^{-1}}a.
\end{equation}
The operation $+$ defined above is commutative. This follows from the equalities
\[
\begin{array}{llll}
a+b &= &a({}^{a^{-1}}b)
= ({}^a{({}^{a^{-1}}b)}).(a^{({}^{a^{-1}}b)})\quad\quad  & \text{by (\ref{braceq1}) and
\textbf{M3}}\\
&= &b.( {}^{({}^a{({}^{a^{-1}}b)})^{-1}} a)
= b({}^{b^{-1}}a) = b+a \quad &\text{by (\ref{braceq4}) and (\ref{braceq1})}.
\end{array}
\]
Clearly,  $a + e = e +a =a$, so the identity element \emph{e} of the group $G$
is the zero-element of $(G, +)$.
Moreover,  $a +({}^a{(a^{-1})}) = a(a^{-1})=e$,  for each $a \in G$,
so ${}^a{(a^{-1})} = -a$ is the inverse of $a$ in
$(G, +)$.
We shall verify that the operation $+$ is associative, but before proving this
we shall show
 that the group $G$ acts from the left on $(G, +)$ as automorphisms, that is
 condition \textbf{Laut}, see (\ref{Laut}),  is in force.
Due to the non-degeneracy condition \textbf{Laut} is equivalent to
\begin{equation}
\label{braceeq6}
{}^a{(b+{}^bc)}= {}^ab+{}^a {({}^bc)}, \quad \forall a,b,c \in
G.
\end{equation}
We apply (\ref{braceq1}) and  \textbf{ML2}  to the LHS and obtain:
\begin{equation}
\label{braceeq7}
{}^a{(b+{}^bc)}= {}^a{(bc)}=({}^ab )({}^{a^b}c).
\end{equation}
Compute the RHS applying first \textbf{ML1} and \textbf{M3}, and then  (\ref{eqlbr3}):
\begin{equation}
\label{braceeq8}
{}^ab +{}^a{({}^bc)}={}^ab +{}^{{}^ab}{({}^{a^b}c)}=
                    ({}^ab)({}^{a^b}c).
\end{equation}
Now  (\ref{braceeq7})  and (\ref{braceeq8}) imply (\ref{braceeq6}),  therefore
condition \textbf{Laut} is in force.
The associativity of the operation $+$ follows from (\ref{braceeq9}) and
(\ref{braceeq10}):
\begin{equation}
\label{braceeq9}
\begin{array}{llll}
a+(b+c) &= &a({}^{a^{-1}}(b+c)) = a({}^{a^{-1}}b+ {}^{a^{-1}}c) & \text{by (\ref{braceq1})
and \textbf{Laut}}\\
&= &a({}^{a^{-1}}b). ({}^{({}^{a^{-1}}b)^{-1}}{({}^{a^{-1}}c)})= a({}^{a^{-1}}b)({}^{(a({}^{a^{-1}}b))^{-1}}c)\quad
&\text{by (\ref{braceq1})}.
\end{array}
\end{equation}
\begin{equation}
\label{braceeq10}
(a+b)+c = (a+b)({}^{(a+b)^{-1}}c)
        =a({}^{a^{-1}}b)({}^{(a({}^{a^{-1}}b))^{-1}}c).
\end{equation}
We have verified that $(G, +)$ is an abelian group, and the multiplicative
group $(G,\cdot)$ acts from the left on  $(G, +)$
as automorphisms.
By (\ref{eqlbr3}) the left action can be
presented as
${}^ab  = ab -a, \quad  \forall a,b \in G$. In other words
one has ${}^ab= \Lcal_a(b)$, for all $a,b
\in G$,  where  the maps  $\Lcal_a$ are defined in
(\ref{LcalRcal}).
This implies  that $(G, +, \cdot)$ satisfies the hypothesis of Proposition
\ref{Thm_braces1}, and condition \textbf{Laut}
is in force, hence $G$ is a left brace, see Prop.
\ref{Thm_braces1} (4). The implication (1)
$\Longrightarrow$ (2) has been proved.

(2) $\Longrightarrow$ (1).
Some steps in this implication can be extracted from \cite{CJO14}.
For  completeness (and simplicity) we give a different proof which is compatible with
our settings and involves   relations and conditions that will be used
throughout  the paper. Assume $(G, +, \cdot)$ is  a left brace.
We first define a left and a right action of $(G, \cdot)$ upon itself.
As in the previous subsection
 define the bijective maps
$\Lcal_a: G\times G
\longrightarrow G$ as $\Lcal_a(b)={}^ab := ab -a$,  see (\ref{LcalRcal}).
 It follows from Proposition \ref{Thm_braces1} that
the map
\[\Lcal: G \longrightarrow \Sym(G) \quad a\mapsto \Lcal_a \quad
\]
defines a left action  of $(G, \cdot)$ upon itself, that is condition \textbf{ML1}
and \textbf{ML0} hold.
Furthermore,  condition \textbf{Laut} and the equalities (\ref{braceq1}),  (\ref{eqlbr3}) are in force.
We next define a right "action" of $(G, \cdot)$ upon itself,
$\Rcal: G
\longrightarrow Sym(X),  \; b\mapsto \Rcal_b,$
by the equality
\begin{equation}
\label{braceq4a}
\Rcal_b(a) = a^b: = {}^{({}^ab)^{-1}}a,  \quad \forall \; a, b \in G.
\end{equation}
We shall verify that this is, indeed, an action. Clearly, \textbf{MR0} is in force, we have to verify \textbf{MR1}.
The argument will involve the compatibility condition \textbf{M3} , see
(\ref{compatibilityeq}) and
condition \textbf{ML2}, so we shall prove these first.
Condition \textbf{M3} is verified by the following equalities:
\begin{equation}
\label{braceeq12}
\begin{array}{llll}
ab= a + {}^ab &=& {}^ab + a\quad& \text{by  (\ref{eqlbr3}), and "$+$ commutative"} \\
        &=& ({}^ab)({}^{({}^ab)^{-1}}a) = ({}^ab)(a^b)\quad&  \text{by  (\ref{braceq1}) and (\ref{braceq4a})}.
\end{array}
\end{equation}
Next we show that \textbf{ML2} is also in force. Indeed, one has
\begin{equation}
\label{braceeq13}
\begin{array}{llll}
{}^a{(bc)}&=& {}^a{(b+ {}^bc)}= {}^ab + {}^a {({}^bc)} \quad& \text{by (\ref{eqlbr3}) and \textbf{Laut}}\\
 &=& {}^ab + {}^{({}^ab.a^b)}c = ({}^ab) ({}^{(a^b)}c) \quad& \text{by \textbf{ML1},
 \textbf{M3}, and (\ref{eqlbr3})}.
\end{array}
\end{equation}
Finally we verify \textbf{MR1}.
Let $a,b,c \in G.$ One has
\begin{equation}
\label{braceeq14}
\begin{array}{llll}
a^{(bc)}&=& {}^{({}^a{(bc)})^{-1}}a  = {}^{(({}^ab)({}^{a^b}c))^{-1}}a \quad&\text{by (\ref{braceq4a}) and by \textbf{ML2}}\\
        &=& {}^{({}^{a^b}c)^{-1}}{({}^{({}^ab)^{-1}}a)} = {}^{({}^{a^b}c)^{-1}}{(a^b)} = (a^b)^c \quad&\text{by
        \textbf{ML1}  and (\ref{braceq4a})}.
\end{array}
\end{equation}
Hence condition \textbf{MR1} is in force, and therefore
the formula (\ref{braceq4a}) defines a right action of
$(G,\cdot)$ upon itself.

We have shown that the (multiplicative) group $(G, \cdot)$  acts upon itself from
the left and from the right via the actions defined above. These actions satisfy \textbf{ML1},
\textbf{MR1},  \textbf{ML2} and the compatibility condition
\textbf{M3}.
Define a map
\[\sigma: G\times G \longrightarrow G\times G, \quad \sigma(a, b): = ({}^ab,
a^b).\] It follows from the results of \cite{LYZ}, see  Fact
\ref{factLYZ}, that the pair $(G, \sigma)$ is a braided group, so
condition \textbf{MR2} is also in force. To show that the map
$\sigma$ is involutive we verify conditions (\ref{involeq}). The
definition (\ref{braceq4a}) of the right action implies
straightforwardly that ${}^{{}^ab}{(a^b)}=
{}^{{}^ab}{({}^{({}^ab)^{-1}}a)}= a$, for all $a,b  \in G$. We
leave the reader to verify the remaining equality $({}^ab)^{a^b} =
b,$  for all $a,b  \in G$. We have shown that $(G, \sigma)$ is a
symmetric group (in the sense of Takeuchi), so the implication (2)
$\Longrightarrow$ (1) is in force.

It follows from Fact \ref{factLYZ} again that each of the
equivalent conditions (1) and (2) implies that $(G, \sigma)$ forms
a non-degenerate symmetric set in the sense of \cite{ESS}.
\end{proof}
Till the end of the paper we shall assume that the following is in force.
\begin{defconv}
\label{convention_actions}
\begin{enumerate}
\item
Let $(G, \sigma)$ be a symmetric group. We shall always consider the left
and the right actions of $G$ upon itself
defined via the formula $\sigma(u,v)= ({}^uv, u^v)$.
We call the left brace $(G, +, \cdot)$ with operation "$+$" defined via
(\ref{braceq1})
\emph{the left brace associated to} $(G, \sigma)$.
\item
Conversely, given a left brace $(G, +, \cdot)$, \emph{the associated symmetric
group} $(G, \sigma)$ is "built" on the multiplicative group
of the brace $(G, \cdot)$ with a left action $\Lcal$ and a right action $\Rcal$ defined as in the proof of
Theorem
\ref{Thm_BracesBraidedGp}, see (\ref{LcalRcal}) and (\ref{braceq4a}),
respectively.
\item Let $(X,r)$ be a (non-degenerate) symmetric set with YB-group $G = G(X,r)$, let
$(G,r_G)$ be  \emph{the symmetric group associated to}
$(X,r)$, see Facts \ref{factLYZ}. The left brace
$(G,+,\cdot)$  associated to $(G,r_G)$ is called
\emph{the left brace associated to the solution} $(X,r)$.
\end{enumerate}
\end{defconv}
 Theorem \ref{Thm_BracesBraidedGp} gives a valuable information on left braces:
 every left brace $(G,+,\cdot)$ is equipped (canonically)
 with \emph{an involutive braiding operator} $r_G$, so that $G$ has an additional structure
  - the structure of a symmetric group. In particular, the powerful matched pairs conditions \textbf{ML2} and
\textbf{MR2} in $G$ can be and will be used to study various properties of the given brace.
We shall study the close relation between the properties of a symmetric set $(X,r)$ and the
corresponding two structures (i) the symmetric group $(G, r_G)$, and (ii) the
left brace  $(G, +, \cdot)$ each of them built on the YB group $G=G(X,r)$.
Theorem \ref{Thm_BracesBraidedGp}  can be used also to directly "translate" results from braces to braided groups or vice versa. The
theorem was already applied in this way by Agata Smoktunowicz,  who  obtained interesting results on two-sided braces and Engel groups,
see \cite{Smok}.

\subsection{More details about the YB monoid $S=S(X,r)$}
Let $(X, r)$ be a solution,
let $S = S(X,r)$ be its associated monoid,  and let $(G, r_G), (G, +, \cdot)$ be its associated symmetric group and left brace,
respectively.
We shall prove, see Proposition \ref{details_Prop1}, that
if the monoid $S= S(X,r)$ is embedded in $G=G(X,r)$ then $S$ is invariant under
the actions of $G$ upon itself.

\begin{remark}
\label{remark_invariantsubset}
Recall that a subset $Y\subset X$ is  \emph{$r$-invariant } if $r(Y
\times Y)\subseteq Y \times Y.$  The restricted map $r_Y= (r_{\mid Y\times
Y}): Y \times Y \longrightarrow Y \times Y$ defines canonically \emph{the induced solution} $(Y, r_Y)$ on $Y$.
A nonempty subset $Y\subset X$ is
\emph{a (left)}
$G$-\emph{invariant subset of} $X$,
if $Y$ is invariant under the
left action of $G$.
Since $r$ is involutive,
 $Y$ is left $G$-invariant if and only if it is right
$G$-invariant, so we shall refer to it simply as \emph{a
$G$-invariant subset}. Each $G$-invariant subset $Y$ of $X$ is also
$r$-invariant, the converse, in general, is not true. Clearly, $Y$ is (left) $G$-invariant
\emph{iff} it is a $\Gcal$-invariant subset.
 Each $G$-orbit $X_0$ under the left action of $G$ on $X$ is
$G$-invariant and therefore it is an $r$-invariant subset.
\end{remark}
\begin{remark}
  \label{remark_embeddingoffinitesolutions}
  It follows from \cite{ESS}, and  \cite{LYZ}, p.18, that the additive group
  $(G, +)$ is isomorphic to the free abelian group,
 ${}_{gr}[ X ]$ generated by $X$, so the set $X$ is embedded in $G$.
It is well-known that if $(X,r)$ is a finite symmetric set then the monoid
  $S=S(X,r)$ is with cancellation law and satisfies the \"{O}re
  conditions. The YB group $G$ is  the group of quotients of $S$, so  $S$
  is embedded in $G$  (clearly $X$ is embedded in $S$).
  \end{remark}
\begin{convention}
   \label{convention_embedding}
 In the cases when $(X, r)$ is an infinite symmetric set we shall assume that the  associated YB monoid $S= S(X,r)$
 is embedded in
 $G=G(X,r)$.
 \end{convention}
Note that for any pair  $x,y \in X\times X$ the only nontrivial equality  of degree 2 in the monoid $S$ involving
the word $xy$ (if any) has the
 shape
 $xy={}^xy.x^y$.

\begin{proposition}
  \label{details_Prop1}
Let  $(X,r)$ be a non-degenerate symmetric set, $S=S(X,r), G=G(X,r), (G, r_G)$ in the usual notation. Then the following three
conditions
hold.
\begin{enumerate}
  \item
The monoid  $S$ is invariant under the left and the right actions of
$G$ upon itself, so $S$ is an $r_G$-invariant subset of $(G, r_G)$.
\item
Let $r_S: S \times S \longrightarrow S \times S$ be the restriction of
$r_G$
on $S \times S$. Then $(S, r_S)$ is a solution. Moreover, it is a braided \textbf{M3}-monoid in the
sense
of \cite{GIM08}.
\item
$S$ is closed under the operation $+$ in $G$. So $(S, +, \cdot)$ is an
    algebraic structure
such that (i) $(S, \cdot)$ is a monoid with unit $1$.
(ii) $(S, +)$ is an abelian monoid generated by $X$ and with the same
neutral element ($0=1$). More precisely,
$(S, +)$ is isomorphic to the free abelian monoid ${}_{mon}[X]$ generated by $X$.
(iii) the left brace equality
$a(b+c)+a=ab +ac$ holds for all $a,b,c \in S$;
 \end{enumerate}
\end{proposition}
\begin{proof}
 \textbf{(1)}. One has  ${}^a1 = 1, \; \forall a \in G$ (by \textbf{ML0}). We use induction on the length $|u|$ of $u\in S$ to show that
 ${}^au\in
 S$ for all $a \in G, u \in S$.
The group  $G$  acts on $X$ from the left, so this gives the base for the
induction. Assume  ${}^au\in S$ for all $a \in G$
and all, $u \in S$ with $|u| \leq n$. Suppose $a \in G,$
$u \in S$ with $|u| = n+1.$ Then $u =vx$, where $x \in X,  v \in S, |v| = n$.
We use \textbf{ML2} to yield
$
 {}^au= {}^a(vx)=({}^av)({}^{(a^v)}x).
$
One has ${}^{(a^v)}x \in X$, and by the inductive assumption
${}^av\in S$, hence ${}^au\in S, \; \forall a\in G, u \in S$.
It follows that $S$ is invariant under the left action of $G$ upon itself.
 The braiding operator $r_G$ is involutive, hence  $S$ is also invariant under the right action of $G$
upon itself. Therefore $S$ is $r_G$-invariant. \textbf{(2)}. The left and the right actions of $G$ on $S$ induce a left and a
right actions
 ${}^{(\ )}\bullet$,$\quad$   $\bullet^{(\ )}$ of $S$ upon itself.
It follows  straightforwardly that  $\;S, S\;$ is a matched pair of monoids in the sense of \cite{GIM08}.
Let $r_S: S \times S \longrightarrow S \times S$ be the restriction of $r_G$
on $S \times S$.
Then the map $r_S$  is bijective, non-degenerate, involutive and satisfies the braid relation.
Therefore $(S, r_S)$ is a
non-degenerate symmetric set.
It follows that $(S, r_S)$ is a braided \textbf{M3}-monoid in the sense
of \cite{GIM08}.
Part (3) is straightforward.
\end{proof}

\section{The retractions of a symmetric group and its derived chain of ideals. Symmetric groups and symmetric sets of finite
multipermutation level}
\label{sec_derived_chain of ideals}
This section is central for the paper.  We study the recursive process of retraction of general symmetric groups and braces. We are
particularly interested in symmetric groups, braces, and symmetric sets of finite multipermutation level. The first
question
to be asked is "what is
 the retraction of a symmetric group?"
 The answer is given in Subsect. 4.2 for a general symmetric group $(G, r)$.
In Subsect. 4.3 we introduce  \emph{the derived chain of ideals
of} $G$, or abbreviated \emph{DCI} and study its properties. Proposition \ref{derived_chainProp} shows
that DCI is an invariant of the group which gives a precise
information about the recursive process of retraction on $(G, r)$.
Theorem  \ref{derived_chainThm1} gives new conditions equivalent to $\mpl(G,r) = m
 <\infty.$
An important application of \emph{DCI} is Theorem
 \ref{slG_theorem} which proves that every symmetric group $(G ,r)$ of finite multipermutation level $m$ is a solvable group of solvable length $\leq m$.

\subsection{Some technical results on multipermutation solution, general case}
 In this subsection we collect some notions and technical results on multipermutation solutions.
  We recall first an important class of solutions.
\begin{example} \cite{Dri} (\emph{Involutive permutation solution of  Lyubashenko}).
\label{exlri}
Let $X$ be a set of cardinality $\geq 2$, let  $f \in Sym (X)$.  Then the map
$r:
X\times X\longrightarrow X\times X$
defined as
\[r(x,y):= (f(y), f^{-1}(x)), \; \forall x,y \in X\] is a non-degenerate
involutive solution of YBE, we shall refer to it as \emph{a permutation solution}.
Clearly, $\Lcal_x = f,$ and $\Rcal_x=f^{-1},$  for all $x \in X$,
so condition \textbf{lri} holds.
This is a solution of multipermutation level 1.
Moreover,   $(X,r)$ is a
(square-free) trivial solution \emph{iff}  $f= id_X$.
\end{example}
The following is straightforward.
\begin{remark}
\label{mpl1} A non-degenerate symmetric set $(X,r)$ has
 $\mpl X =1$ if and  only if it is a permutation solution.  The notion  "\emph{a multipermutation solution}" is a generalization of
 \emph{a permutation solution}.
\end{remark}

Let $(X,r)$,  be a symmetric set of order $|X|\geq 2$. The
recursive definition of retractions, $\Ret^m(X,r)=
\Ret(\Ret^{m-1}(X,r))$ implies straightforwardly
that the process of taking retractions of higher order "halts" at
the $m$-th step \emph{iff}
 $m\geq1$ is the minimal integer (if any), such that $\Ret^{m+1}(X,r)=
\Ret^m(X,r)$. Moreover,
 there are obvious equivalences
\[
\begin{array}{llll}
[\mpl X =m]    &\Longleftrightarrow&  [\mpl(\Ret^{(k)}(X,r)) = m-k, \; \forall \; k,  0 \leq k \leq  m]&\\
              &\Longleftrightarrow & [\Ret^{(m-1)}(X,r)\;\text{is a
              permutation solution}].&
\end{array}
 \]
 In particular, if  $(X,r)$ is a square-free solution, then
$\mpl X =m $ \emph{iff}    ${\Ret}^{(m-1)}(X,r)$ is a trivial solution
of order $\geq 2$.

\begin{definition}
\label{defcondstar}
Let  $(X, r)$ be a symmetric set, $|X|\geq 2$. We define \emph{Condition} (*) on $X$ as follows.
\[\text{\emph{Condition}}\;  (*): \quad  \text{For every}\; x \in X\;\text{there exist some}\;a \in X,\; \text{such that}\;  {}^ax =
x.\]
\end{definition}
Condition (*) is satisfied by every square-free solution $(X,r)$ and by every symmetric group  $(G, r)$. Indeed in the first
case
${}^xx=x,$ holds for all $x \in X,$ while in the second case one has    $\Lcal_1 = id_G$  (condition \textbf{ML0} in $G$).
\begin{lemma}
\label{Lemma_symgr_mpl1}
Let  $(X, r)$ be a symmetric set, $|X|\geq 2$. Suppose that $X$ satisfies condition (*).
Then the following hold.
(1) $\mpl X = 1$ \emph{iff} $(X,r)$ is a trivial solution.
(2)  One has $2 \leq \mpl X = m< \infty$ \emph{iff} $\Ret^{m-1}(X,r)$ is a trivial solution
with $\geq 2$ elements.
(3) $\mpl X = 2$ if and only if
$(X,r)$  satisfies
  \begin{equation}
 \label{Lemma1_Thm_maineq1a}
 \Lcal_{{}^yx}= \Lcal_x, \; \forall x, y \in X, \quad \text{and}  \;\Lcal_x \neq id_X \text{    for some}\; x \in X.
 \end{equation}
 Moreover, in this case $(X,r)$ satisfies \textbf{lri}.
 \end{lemma}
 In particular,  Lemma \ref{Lemma_symgr_mpl1} is in force for symmetric groups $(G, r)$.

In cases when we have to write a sequence of successive actions in $X$ (or in $G$) we
shall use one also well known notation
\begin{equation}
\label{laa} \alpha\la x = {}^{\alpha}{x}.
\end{equation}
\begin{definition}
\label{def_tower}\cite{GIC}
Let
$\zeta_1, \zeta_2,\cdots, \zeta_m \in X$. The expression
\[
\omega = (\cdots((\zeta_m \la \zeta_{m-1})\la\zeta_{m-2})\la
\cdots\la\zeta_2)\la\zeta_1
\]
is called \emph{a tower of actions}  (or shortly \emph{a tower})
\emph{of length} $m$.
\end{definition}
Denote by $u$ the tower $(\cdots((\zeta_m \la
\zeta_{m-1})\la\zeta_{m-2})\la \cdots\la\zeta_3)\la \zeta_2$,
then, obviously, the element
 $\; \omega=
{}^{u}{\zeta_1}$, belongs to the $G$-orbit of $\zeta_1$.

\begin{remark}
 \label{mplremark1}
 Suppose $(X, r)$ is a non-degenerate symmetric set, $|X|\geq 2$.
 Then  the conditions (\ref{mplremark11}) and (\ref{mplmequality_generalcase3})
  given below are equivalent:
 \begin{equation}
\label{mplremark11}
\begin{array}{rl}
(\cdots((a\la y_{m})\la
y_{m-1})\la \cdots \la y_2)\la
y_1& = (\cdots((b\la y_{m})\la
y_{m-1})\la \cdots \la y_2)\la
y_1,\\
\forall a, b,\; y_1, \cdots y_m \in X.&
\end{array}
\end{equation}
\begin{equation}
\label{mplmequality_generalcase3}
\begin{array}{rl}
(\cdots(([a]\la [y_{m}])\la
[y_{m-1}])\la \cdots \la [y_3])\la [y_2]&
=(\cdots(([b]\la [y_{m}])\la
[y_{m-1}])\la \cdots \la [y_3])\la [y_2], \\
\forall a, b,\; y_2, \cdots y_m \in X.&
\end{array}
\end{equation}
\end{remark}
The following statement is a generalization of Theorem 5.15, \cite{GIC} which initially was proven for square-free solutions.
 \begin{proposition}
 \label{mplmtheorem1}
Suppose $(X, r)$ is a solution, $|X|\geq 2$.
\begin{enumerate}
\item $(X, r)$ is a multipermutation solution with $\mpl X \leq m <\infty$, if and only if
(\ref{mplremark11})
is in force. Moreover, $\mpl X= m$ \emph{iff} $m$ is the minimal integer for
which (\ref{mplremark11}) holds.
\item Suppose furthermore that $X$ satisfies condition (*).
Then  $\mpl X \leq m< \infty$ \emph{iff} the following  is in force.
\begin{equation}
\label{mplmequality}
\begin{array}{rl}
((\cdots((y_m\la y_{m-1})\la
y_{m-2})\la \cdots \la y_2)&\la
y_1)\la x \\
&=((\cdots( y_{m-1}\la y_{m-2})\la \cdots \la y_2)\la y_1)\la x, \\
\forall x,\; y_1, \cdots y_m \in X.&
\end{array}
\end{equation}
Moreover,
$\mpl (X,r)= m$ \emph{iff} $m$ is the minimal integer for
which (\ref{mplmequality}) holds.
\end{enumerate}
\end{proposition}
The proof is technical and we skip it here, but it can be found in
 \cite{GI15} Variant 2, see Proposition 5.6.

\begin{corollary}
 \label{cor_mplG}
Let $(G, r)$ be a symmetric group. Then $\mpl G \leq m< \infty$ \emph{iff} (\ref{mplmequality}) is in force. Furthermore,
$\mpl (G,r)= m$ \emph{iff} $m$ is the minimal integer for
which (\ref{mplmequality}) holds.
\end{corollary}

\subsection{The retraction of a symmetric group}
Let $(G, r)$ be a symmetric group. Let $\Gamma= \Gamma_l$ be
\emph{the kernel of the left action of $G$ upon itself}. In a
symmetric set $(X,r)$ one has $\Lcal_x = \Lcal_y$ \emph{iff}
$\Rcal_x = \Rcal_y$, see \cite{ESS}, section 3.2. Therefore
$\Gamma$ coincides with the kernel $\Gamma_r$  of the right action
of $G$ upon itself:
\begin{equation}
\label{gamma}
\Gamma = \Gamma_l=\{a \in G\mid {}^au = u, \; \forall u \in G\} =
\{a \in G\mid u^a =
u, \; \forall u \in G\}= \Gamma_r.
\end{equation}
Moreover,   $\Gamma$ is invariant with respect to the left and the right actions of $G$ upon
itself. Indeed, let $a \in \Gamma, u, v \in G$.
Then  the equalities
\[
({}^au)({}^av) = uv = {}^a{(uv)}=({}^au)({}^{a^u}v)
 \]
imply ${}^{a^u}v ={}^av = v$ for every $a \in \Gamma, u, v \in G$.
Therefore $a^u \in \Gamma, \forall \; u \in G,$ so $\Gamma$ is
invariant with respect to the right action of $G$ upon itself.
Then by Remark \ref{remark_invariantsubset} $\Gamma$ is
$G$-invariant. It follows that $\Gamma$ is also $r$-invariant. It
is clear that $\Gamma$ is an abelian normal subgroup of $G$. Let
$\widetilde{G}= G/\Gamma$, be the
 quotient group of $G$, modulo $\Gamma$. The following result can be extracted from  \cite{Takeuchi}, Proposition 2.9.
\begin{fact}
\label{Gamma_Takeuchi}
 \cite{Takeuchi}
The matched pair structure
$r: G\times G \longrightarrow  G\times G$ induces a map $\; r_{\widetilde{G}}:
\widetilde{G} \times \widetilde{G} \longrightarrow
\widetilde{G} \times \widetilde{G}\;$, which makes $\widetilde{G}$ a braided
group.
The pair $(\widetilde{G}, r_{\widetilde{G}})$ is called \emph{a  quotient braided group
of} $(G, r)$.
\end{fact}
Note that the map $r_{\widetilde{G}}$
is in fact involutive, since $r$ is involutive.  Therefore
 $(\widetilde{G}, r_{\widetilde{G}})$ is  a symmetric
group, we call it \emph{the  quotient symmetric group
of} $(G, r)$. We shall use notation $\widetilde{a}$ for the image of $a$ under
the canonical group epimorphism $G \longrightarrow
\widetilde{G}.$

Consider now $(G, r)$ as a symmetric set, with retraction $\Ret(G, r) = ([G],
r_{[G]})$.  Then the obvious implications \[\Lcal_a = \Lcal_b \;
\Longleftrightarrow\;  \Lcal_{(b^{-1}a)}  = id_G \; \Longleftrightarrow\;
b^{-1}a \in \Gamma \; \Longleftrightarrow\;  a\Gamma
= b\Gamma\]
and an easy argument implies the following.
\begin{lemma}
\label{Lemma_retG}
Notation and assumption as above. Let $(G, r)$ be a symmetric group, let $(\widetilde{G}, r_{\widetilde{G}})$ be
its quotient symmetric group,
and let $\Ret(G,r)=([G],
r_{[G]})$ be its retraction, where $(G, r)$ is considered as a symmetric set.
Then the map
\[ \varphi: (\widetilde{G}, r_{\widetilde{G}}) \longrightarrow  ([G], r_{[G]}), \quad
\widetilde{a}\mapsto [a],
\]
is an isomorphism of symmetric sets. Moreover, the retraction $([G],
r_{[G]})$  is also a symmetric group and $\varphi$ is an isomorphism of symmetric groups.
\end{lemma}
We shall often identify the retraction $\Ret(G,r)= ([G], r_{[G]})$ of $(G, r)$ and its
quotient symmetric group   $(\widetilde{G},
r_{\widetilde{G}})$, where $\widetilde{G}=
G/\Gamma$.

\begin{remark}
Analogous results but in terms of braces, (and \emph{the socle of a brace})
were proven in works of  Rump some years after Takeuchi, see  \cite{Ru07}, Proposition 7,
for details. One can find  helpful interpretation of some Rump's results  and other results on braces in \cite{CJO14}.
\end{remark}

\subsection{The derived chain of ideals of a symmetric group}
Let $(G, +, \cdot)$ be a left brace.  Recall that a subset  $I\subset G$ is called \emph{an ideal of G} if
it is (i) a normal subgroup of the multiplicative group $(G, \cdot)$, (ii)  a subgroup of the additive group of $(G, +),$
and (iii) $I$
is invariant with respect to the left actions
$\mathcal{L}_a$,  $a\in G$, \cite{CJO14}.

Suppose now that $(G, r)$ is a symmetric group,  $(G, +, \cdot)$ is the associated left brace, and
$H$ is a subgroup of $(G, r)$. By Remark \ref{remark_invariantsubset} $H$ is left $G$-invariant
 \emph{iff} $H$ is $G$-invariant, hence
 $H$ is also $r$-invariant.
Moreover, in this case $H$ is (automatically)  a subgroup of the additive group $(G, +)$ of the brace. Indeed, the equalities $a -b = a
+{}^b{(b^{-1})}=a({}^{(a^{-1}b)}{(b^{-1})})$ imply
that $a-b \in H,$ whenever $a,b \in H.$
So a subgroup $H$ of $(G, r)$ is an ideal of the brace $(G, +, \cdot)$ if and only if it is a $G$-invariant normal subgroup of $(G, r)$.
\begin{definition}
\label{ideal_def}  We say that $H$ is \emph{an ideal} \emph{of the
symmetric group}  $(G, r)$ if $H$ is a normal subgroup of $G$
which is $G$-invariant.
\end{definition}

Suppose $H$ is an ideal of the symmetric group $(G,r)$. Then the
(multiplicative) quotient $\overline{G}=G/H$ has also a canonical
symmetric group structure $(\overline{G}, r_{\overline{G}})$
induced from $(G, r)$. Analogously, (and independently) there is a
canonical brace structure $(\overline{G_{br}} , +, \cdot)$ on the
quotient brace $\overline{G_{br}} = (G,+, \cdot) /H$ modulo the
(brace) ideal $H$. Note that the left brace $(\overline{G}, +,
\cdot)$ associated with the quotient symmetric group
$(\overline{G}, r_{\overline{G}})$ (by Definition-Convention
\ref{convention_actions}) and the quotient brace
$\overline{G_{br}}$ coincide, so we shall simply use the notation
$(\overline{G}, +, \cdot)$.

 The following  Isomorphism Theorems for Symmetric Groups are in force. The proofs are routine and we leave them  for the reader.
 \begin{remark}
  \label{IsomThmforSymGroups}

   \begin{enumerate}
\item
 \label{1stIsomThm}
 [\emph{First Isomorphism Theorem for Symmetric Groups}]
 $\;$ Let $f:  (G,r) \longrightarrow (\widetilde{G}, r_{\widetilde{G}})$ be an epimorphism of symmetric groups.
 The kernel $K = \ker f$ is an ideal of  $(G,r),$ and there is a natural isomorphism of symmetric groups
 $G/ K \simeq \widetilde{G}.$
 \item
 \label{3rdIsomThm}
 [\emph{Third Isomorphism Theorem for Symmetric groups}] $\;$
 Let $(G,r)$ be a symmetric group, let $K$ be an ideal of $G$, $\widetilde{G} = G/ K$,  and
 let $f: G \longrightarrow \widetilde{G}$ be the canonical epimorphism of symmetric groups (one has $\ker f = K$).

 (a) There is a bijective correspondence
 \[\{\text{ideals $H$ of $G$ containing} \;\; K\} \longleftrightarrow \{ \text{ideals $\widetilde{H}$ \;of}\;\; \widetilde{G}\},\]
 given by
 \[  H \rightarrowtail f(H)\simeq H/K, \quad\quad    f^{-1}(H)  \leftarrowtail \widetilde{H}.\]
(b) For every ideal $H \supset K$ of $G$ one has \[(G/ K)/ (H/K)
\simeq G/H,  \;\; gK.(H/K) \mapsto gH.\]
\end{enumerate}
Analogous statements are in force for braces.
\end{remark}

Suppose $(G,r)$ is a  symmetric group, then the kernel $\Gamma$ of
the left action of $G$ upon itself, see (\ref{gamma}), is a
$G$-invariant normal subgroup of $G$, so it is an ideal, called
\emph{the socle } of $G$, and denoted $\soc (G)$. Clearly, $\soc
(G)$ is an abelian subgroup of $G$, ($\soc (G)= \{1\}$  is also
possible). Moreover, $ \soc (G)= G$ \emph{iff}  $(G,r)$ is the
trivial solution and $G\neq \soc (G)= \{1\}$ \emph{iff} $(G,r)$ is not
retractable.

We call $(G,r)$ \emph{a prime symmetric group}  if $G$ does not have proper, nontrivial ideals.
Prime braces are defined analogously.

\begin{remark} Suppose $(G,r)$ is a prime symmetric group. Then either (i) $\soc G = G$,
thus $(G,r)$ is the trivial solution, or (ii) $\soc G = \{1\}$, hence $G$ is not retractable.
More generally, if $K$ is \emph{a maximal ideal of a symmetric group} $(G,r)$, then the quotient symmetric group
$\widetilde{G} = G/K$ is either a trivial solution, or $\widetilde{G}$  can not be retracted.
\end{remark}
\begin{remark}
Let $G=(G,r)$ be a  symmetric group, denote by $(G^j, r^j), j \geq 0,$ its $j$-th retraction. So for each $j \geq 1$ one has
 $(G^j, r^j) = \Ret^{j} (G, r) \simeq G^{j-1}/ \Gamma_{j},$ where $\Gamma_{j}= \soc (G^{j-1})$. We have $G^0 = \Ret^0(G,r)= G$,
 $\Gamma_{1}= \Gamma$.
 For $j \geq 0$  denote by
 $\varphi_{j+1}$ the canonical epimorphism of symmetric groups
 \[
   G^j \stackrel{\varphi_{j+1}}{\longrightarrow} G^j/\Gamma_{j+1} \simeq G^{j+1}, \quad  \ker
  \varphi_{j+1}
  = \Gamma_{j+1}.
 \]
 Observe the canonical sequence of epimorphisms of symmetric groups (some of these possibly coincide):
\[G \stackrel{\varphi_1}{\longrightarrow} G^1=G/\Gamma \stackrel{\varphi_2}{\longrightarrow}
G^2=G^1/\Gamma_2\stackrel{\varphi_3}{\longrightarrow} G^3=G^2/\Gamma_3\stackrel{\varphi_4}{\longrightarrow} \cdots .\]
 Set $K_0 = \{1\}, \; K_1 = \Gamma=\Gamma_1$,  and for all $j > 1$ denote by $K_j$ the pull-back of $\Gamma_{j}$ in $G$, that is $K_j
 = (\varphi_{j-1}\circ\cdots \varphi_2 \circ \varphi_1)^{-1}(\Gamma_{j})$.
Each set $K_j$, $j \geq 0$, is an ideal of $G$, thus we obtain a (non-decreasing) chain of ideals in $G$:
\begin{equation}
  \label{derived_chaineq}
 \{1\}= K_0\subset K_1 \subset K_2\subset \cdots \subset K_j\subset\cdots .
\end{equation}
\end{remark}
 \begin{definition}
  \label{derived_chainDef}
  We shall call  the chain (\ref{derived_chaineq})  \emph{the derived chain of ideals of the symmetric group $(G,r)$}
  or shortly, \emph{the derived chain of} $G$.
  \end{definition}
The derived chain is an invariant of the symmetric group $G$ which
reflects the recursive process of retraction. In particular, it
gives an explicit information whether $(G,r)$ has a finite
multipermutation level.
\begin{proposition}
  \label{derived_chainProp}
  Let $(G,r)$ be a symmetric group with derived chain of ideals (\ref{derived_chaineq}). Notation as above.
 \begin{enumerate}
 \item
For all $j \geq 1$ there are isomorphisms
\[ K_j/ K_{j-1} \simeq
\Gamma_j, \quad G/K_j \simeq  \Ret\,^{j}(G,r),\] and canonical epimorphisms
of symmetric groups \[\mu_j : G/ K_{j-1} \longrightarrow G/K_j, \;\;
\ker\mu_j \simeq K_j/ K_{j-1}.\] In particular, $K_j/ K_{j-1} =
\soc (G/K_{j-1})$, $j \geq 1,$ are abelian symmetric groups ($K_j/
K_{j-1} =1$ is possible). The following diagram is commutative:
\[
\xymatrix{
 G\ar[d]_{\|}
\ar[r]^{\mu_1}
&G/K_1\ar[d]_{\wr}
 \ar[r]^{\mu_2}
&G/K_2\ar[d]_{\wr}\ar[r]^{\mu_3}&\dots\ar[r]^{}&G/K_{m-1}
\ar[d]_{\wr}\ar[r]^{\mu_{m}}&G/K_{m}\ar[d]_{\wr}\ar[r]^{}&\cdots\\
G_0\ar[r]_{\varphi_1}  &\Ret (G,r) \ar[r]_{\varphi_2}
&\Ret^2 (G,r)\ar[r]_{\varphi_{3}}&\dots\ar[r]_{}&\Ret^{m-1} (G,r)\ar[r]_{\varphi_{m}}
&\Ret^{m} (G,r)\ar[r] &\cdots }
\]
(The vertical arrows denote isomorphisms of symmetric groups).
\item The derived chain of $G$  stabilizes if and only if
$\;\Ret^{j+1} (G, r)= \Ret^j (G, r)$ (or equivalently, $K_{j+1}=
K_j$) for
    some  $j \geq 0.$
\item Let $m$ be the minimal integer (if any) such that $K_{m+1}=
K_m$. Then the derived chain \emph{has length}  $m+1$, $\Ret^{m}
(G, r)= \Ret^{m+1}  (G, r)$, so the process of recursive
retraction halts in $m$ steps. Moreover, exactly one of the
following two conditions is satisfied:
 \begin{enumerate}
 \item $K_m = G$. Then $m \geq 1$,  and  $(G,r)$ is a multipermutation solution with $\mpl (G,r) = m$;
 \item $K_m \subsetneqq G$  ($m = 0$ is possible). Then $\Ret^m (G, r)$ is a symmetric group of order $\geq 2$ which can not be
     retracted.
      \end{enumerate}
   \end{enumerate}
\end{proposition}
 \begin{proof}
 (1) We "iterate" the Third Isomorphism Theorem for symmetric groups to yield the following:
 \begin{equation}
  \label{isom_eq1}
 \begin{array}{llll}
 \Gamma_1 =  &K_1,             &\;&\Ret^1 (G, r)= (G^1, r^1) \simeq G/K_1, \\
  \Gamma_2\simeq & K_2/K_1,    &\;&\Ret^2(G, r)= (G^2, r^2) \simeq G/K_2, \\
 \Gamma_j\simeq  &K_j/K_{j-1}, &\;&\Ret^j (G, r) = (G^j, r^j) \simeq G/K_j, \\
                               &&&\forall \; j \geq 1.
\end{array}
\end{equation}
The isomorphisms $\Gamma_j\simeq K_j/K_{j-1}$ imply that each of the quotient groups $K_j/K_{j-1}$, $j \geq 1$, is abelian. The
commutativity of the diagram is straightforward.

(2) The above argument implies that there is an equality $K_{j+1}= K_j$ \emph{iff}  $\Gamma_{j+1} = 1$, which is equivalent to
$\Ret^{j+1} (G, r)= \Ret^j (G, r)$. In this case one has $\Ret^{j+p} (G, r)= \Ret^j (G, r)$, therefore
 $K_{j+p}= K_j$, for every $p \geq 1$, and the derived chain stabilizes.
Suppose that this is the case and let $m$ be the minimal integer with $K_{m+1}= K_m$. Without loss of generality we may assume $m \geq
1$.
 Then the derived chain of $G$ is exactly $\{1\}= K_0\subsetneqq K_1\subsetneqq\cdots \subsetneqq K_m$.
 It follows from our discussion above that
all retractions $\Ret^{j}(G,r), \; 0 \leq j\leq m$ are distinct, but $\Ret^m (G, r)= \Ret^{m+p}(G,r), \forall  p \geq 1$, therefore the
process of recursive "retraction" halts exactly at the $m$-th step.

(3)  (a). Assume  $K_m = G$ is in force. Then  $m >1$, the
symmetric group $G^{m-1} =\Ret^{m-1} (G, r) \simeq  G/K_{m-1}$ has
order $> 1$ and $\Gamma_m=\soc (G^{m-1}) \simeq  K_m/K_{m-1},$
$G/K_{m-1}= G^{m-1}$. Clearly, then $\Ret(G^{m-1}) = \Ret^{m} (G,
r)$ is a one element solution, so $\mpl (G,r) = m.$

(b). Suppose now  $K^m \subsetneqq G$ holds. In this case $\Ret^{m} (G, r) \simeq  G/K_{m}$ is a symmetric group of order $>1$ whose
socle is trivial, so $\Ret^{m} (G, r) = \Ret^{m+1} (G, r)$, and the process of "retracting" $(G,r)$ halts in $m$ steps, but without
reaching a one element solution. Clearly, in this case $(G,r)$ is not a multipermutation solution.
 \end{proof}

\begin{corollary}
\label{Cor_finite_derchain_eq}
The derived chain of ideals of $G$ has the shape
 \begin{equation}
  \label{finite_derchain_eq}
\{1\}=  K_0\subsetneqq K_1 \subsetneqq K_2 \subsetneqq \cdots \subsetneqq K_{m-1}\subsetneqq K_{m} = G
    \end{equation}
    if and only if $K_{m-1}\subsetneqq K_{m} = G$.
 \end{corollary}

 Let $B= (B, +, \cdot)$ be a left brace, the operation $*$ on $B$ is defined via (\ref{eqleftbrace*}). By $B^{(s)},\; s = 1, 2,
 \cdots,\;$
 we denote the chain of ideals
introduced by Rump in \cite{Ru07}, one has $B^{(1)}=B$ and
$B^{(s+1)}=B^{(s)}*B$. $B$ is called \emph{a right nilpotent brace} if there
exists a positive integer $n$ such that $B^{(n)}=0$, \cite{CGIS16}.
Using (\ref{LcalRcal}) and (\ref{eqleftbrace*}) one can present
the usual left action of $B$ upon itself as
\begin{equation}
  \label{br_radical_action1}
  {}^ab = a*b+b, \quad a*b= {}^ab-b.
  \end{equation}
   So the socle of $B$ satisfies
\begin{equation}
  \label{soq_eq1}
\soc(B)* B = 0, \quad \soc (B) = \{a \in B\mid a*b = 0, \; \forall b \in B \}.
\end{equation}
\begin{remark}
\label{mplB=1impliesBtwo-sided_rem} Suppose $(B, +, \cdot)$ is a
left brace, $|B|> 1.$  Then $B*B = 0$ if and only if $\mpl B = 1$.
In this case $a\cdot b = a+b, \forall a,b \in B$ and $B$ is a two-sided
brace.
\end{remark}

\begin{remark}
The following theorem provides four different conditions, each of
which is equivalent to $\mpl (G, r) = m$. The equivalence of
conditions (\ref{a}) and (\ref{c}) was proved first by the author
for the case when $(G, +, \cdot)$ is a two-sided brace, see
Proposition 5.16, \cite{GI15}. Later this equivalence was proven
for the general case when $(G, +, \cdot)$ is an arbitrary left
brace, \cite{CGIS16},  Proposition 6. Using our technique with
derived chains we provide an independent (and different) proof of
this particular equivalence and find new equivalent
conditions.
\end{remark}
\begin{theorem}
 \label{derived_chainThm1}
Let $(G,r)$ be a nontrivial symmetric group, $(G, +, \cdot)$ its associated left brace, and let $ \;\{1\}= K_0 \subseteq K_1 \subseteq
K_2 \subseteq \cdots\; $ be its derived chain of ideals.
\begin{enumerate}
\item The derived chain of ideals satisfy
 \begin{equation}
  \label{dchain_eq1}
   ((\cdots((K_j*  \overbrace{G)* G)* \cdots )*G)}^{s\; \text{times}}  \subseteq K_{j-s},  \; \forall   j,s, 1 \leq s \leq j.
\end{equation}
\item
The following conditions are equivalent.
\begin{enumerate}
 \item
  \label{a}
 $(G,r)$ has a finite multipermutation level $\mpl G =m\geq 1.$
 \item
 \label{b}
 The derived chain of ideals of $G$ has the shape
 (\ref{finite_derchain_eq}).
\item
 \label{d}
 The ideals $G^{(j)}$ satisfy:
 \begin{equation}
 \label{dchain_eq2}
 G^{(j+1)}\subseteq K_{m-j},\; 0 \leq j \leq m; \quad  G^{(j+1)}\varsubsetneq
 K_{m-j-1},\;
                   0 \leq j \leq m-1.
                                   \end{equation}
\item
 \label{c}
The brace $G$ is right nilpotent of nilpotency class $m+1$, i.e. $G^{(m+1)} = 0$, and $G^{(m)}\neq 0.$
\end{enumerate}
\end{enumerate}
\end{theorem}
\begin{proof}
We shall use the notation and results of  Proposition \ref{derived_chainProp}. Recall first that $1= 0$ in $(G, +, \cdot)$, so  we also
have $\{0\}=  K_0$.

\textbf{(1).}
  We have shown that for each $j \geq 0$ there are isomorphisms of symmetric groups
 \begin{equation}
\label{vip_isoeq}
\Ret\,^j(G,r) = G^j \simeq G/K_j,  \;\;\soc (G/K_j) = K_{j+1}/K_j,
\end{equation}
  where for completeness  $\Ret^0(G,r) = G^0= G$.
 In each of the braces $(G^j, +, \cdot)$ we also have the induced operation $*$.
 Then the first equality in (\ref{soq_eq1}) implies $(\soc (G/K_j))* (G/K_j) = 0,$ or equivalently,
 $(K_{j+1}/K_j)*(G/K_j) = 0$ holds in $G/K_j$. This implies $K_{j+1}* G \subseteq K_j,$ for each $j \geq 0$, which verifies
 (\ref{dchain_eq1}), where $s=1.$
 The general equalities (\ref{dchain_eq1}), where $s$  is an integer, $1 \leq s \leq j$, follows by induction on $s$.

 \textbf{(2).}
  (\ref{a}) $\Longrightarrow$  (\ref{b}). Assume $\mpl G = m$. Then (\ref{vip_isoeq}) implies that (\ref{finite_derchain_eq}) is exactly
  the derived chain of ideals for $G$.

 (\ref{b}) $\Longrightarrow$ (\ref{a}).
  Suppose the derived chain of $G$ is given by (\ref{finite_derchain_eq}),
then $\Ret^{m-1}  (G, r) \simeq G/K_{m-1}$ is a solution of order $\geq 2$, and
$\Ret^{m}  (G, r)\simeq G/K_{m} = \{1\}$ is a one element solution,  therefore $\mpl G = m.$

(\ref{d})  $\Longrightarrow $ (\ref{b}). Setting $j=0$ in  (\ref{dchain_eq2}) we get
$G= G^{(1)}= K_{m}\varsubsetneq K_{m-1}$, but by definition $K_{m}\supseteq K_{m-1}$, so $K_{m-1}\varsubsetneq K_{m}$ and by Corollary
\ref{Cor_finite_derchain_eq}
the derived chain of ideals has the shape
(\ref{finite_derchain_eq}).

(\ref{b})  $\Longrightarrow $ (\ref{d}).
Assume now the derived chain has the shape (\ref{finite_derchain_eq}). Then $G^{(1)}= K_m$,  and the equalities (\ref{dchain_eq1})
imply
\[G^{(j+1)}=
((\cdots((K_m*  \overbrace{G)* G)* \cdots )*G)}^{j\; \text{times}}  \subseteq K_{m-j},  \; \forall   j,  1 \leq j \leq m.
\]

We shall use induction on $m$ to prove $G^{(j+1)}\varsubsetneq K_{m-j-1}, \; \forall   j,  \;1 \leq j \leq m$.
If  $m =1$, the derived chain is trivial: $\{1\}=  K_0\subsetneqq K_1  = G$, so clearly $G^{(1)}= K_1,$ and $G^{(1)} \subsetneq K_0$,
which gives the base for induction. Assume for all $m, 1 \leq m\leq n$,  (\ref{finite_derchain_eq}) implies
\begin{equation}
 \label{dchain_eq2a}
   G^{(j+1)}\varsubsetneq K_{m-j-1},
                  \; \forall j \; 0 \leq j \leq m.
                                   \end{equation}
  Suppose the derived chain of ideals has length $m+1$ and shape (\ref{finite_derchain_eq}) with $m=n+1$.
  We consider the quotient group $\widetilde{G}= G/K_1$. Then the ideals of the derived chain of $\widetilde{G}$ are exactly
  $\widetilde{K_s}= K_s/K_1$, $1 \leq s \leq m$, so the chain has length $m=n+1$ and shape
\[ \{1\}=  \widetilde{K_1}= K_1/K_1\subsetneqq \widetilde{K_2}= K_2/K_1
 \subsetneqq \cdots \subsetneqq \widetilde{K_{m-1}} \subsetneqq \widetilde{K_{m}}= \widetilde{G}.
 \]
 By the inductive assumption $\widetilde{G}^{(j+1)}\varsubsetneq   \widetilde{K_{n-j-1}}$ for all $j,  1 \leq j \leq n.$
This implies $G^{(j+1)}\varsubsetneq   K_{n-j-1}$, hence (\ref{dchain_eq2a}) is in force.

Next we shall show
(\ref{d}) $\Longrightarrow$  (\ref{c}) $\Longrightarrow$  (\ref{b}) using again "the derived cains technique".
The implication (\ref{d}) $\Longrightarrow$  (\ref{c}) is straightforward.

(\ref{c}) $\Longrightarrow$  (\ref{b}). Assume $G^{(m+1)} = 0$, and $G^{(m)}\neq 0.$ We shall use induction on $m$ to show  $K_{m-1}
\subsetneq K_{m}=G$.  The following implications are clear:
\begin{equation}
 \label{dchain_eq2b}
G^{(m+1)} = 0 \Longleftrightarrow G^{(m)} \subseteq K_1, \quad G^{(m)}\neq 0 \Longleftrightarrow G^{(m-1)}\varsubsetneq K_1.
\end{equation}
If $m=1$  then the equalities $G*G = G^{(2)} = 0$, and $G=G^{(1)}\neq 0$ imply $G = \soc G = K_1 \varsupsetneq K_0$ which gives the base
for induction.
Assume our statement is true for all $m, 1\leq m \leq n-1$. Suppose $G^{(m+1)} = 0$, and $G^{(m)}\neq 0,$ where $m=n$.
Consider the quotient group $\widetilde{G}= G/K_1$ and its derived chain of ideals  $\{1\}=  \widetilde{K_1}= K_1/K_1 \subseteq
\widetilde{K_2}= K_2/K_1
 \subseteq \cdots $.
 Our assumption and  (\ref{dchain_eq2b})  imply $\widetilde{G}^{(m)}= 0$, and $\widetilde{G}^{(m-1)}\neq 0$.  Thus by the inductive
 hypothesis the derived chain of ideals of $\widetilde{G}$   has length $m$ and $\widetilde{K_{m-1}}= K_{m-1}/K_1  \subsetneq
 \widetilde{K_{m}}= K_{m}/K_1 = \widetilde{G}$. By the Third Isomorphism Theorem for symmetric groups this implies $K_{m-1} \subsetneq
 K_{m}=G$, hence by Corollary \ref{Cor_finite_derchain_eq} the derived chain of $G$ satisfies (\ref{finite_derchain_eq}).
\end{proof}

\begin{theorem}
\label{slG_theorem}
Every symmetric group  $(G,r)$ of finite multipermutation level $\mpl G = m $ is a solvable group of solvable length
$\sol G \leq m.$
\end{theorem}
\begin{proof}
By Theorem \ref{derived_chainThm1} $\mpl G = m$ implies that $G$ has a chain of normal subgroups given by (\ref{finite_derchain_eq}).
 Moreover, by Proposition  \ref{derived_chainProp} each quotient $K_j / K_{j-1}, 1 \leq j \leq m,$ is abelian. It follows then that
 $G$ is a solvable group with solvable length $\sol G \leq m.$
\end{proof}

\begin{remark} Note that solvability of a symmetric group $(G,r)$ does not imply $\mpl(G,r)<\infty$. Recall that
the symmetric group $G = G(X,r)$  is solvable, whenever $(X,r)$ is
a finite solution, \cite{ESS}. Suppose that the solution
 $(X,r)$ is finite but $(X,r)$ is not a mulipermutation
solution, e.g. Vendramin's example of order 8, see Remark
\ref{Rem_ex_Ve}. Then the associated group $G=G(X,r)$ is a
solvable group which is not a multipermutation solution (this
follows from Theorem \ref{Th_mplG_mplX}).
\end{remark}

\section{The symmetric group $G(X,r)$ of a symmetric set $(X,r)$. Derived symmetric groups and derived permutation groups of a solution}
\label{sec_symgrG(x,r)}
In this section as usual $(X,r)$ is a non-degenerate symmetric set of arbitrary cardinality,
$G= G(X,r)= (G, r_G)$, $(G, +, \cdot)$, $\Gcal= \Gcal(X,r)$, $\Gamma = \soc G$, $K = \ker \Lcal$, see Definition \ref{Gcaldef}.
\subsection{The symmetric group $G(X,r)$ of a symmetric set $(X,r)$ and its retraction.}
\begin{remark}
\label{factGamma=K}
\cite{Takeuchi}, p. 15.
Let $(X,r)$ be a (non-degenerate) symmetric set, notation as above.
 Then the following conditions hold.
 (i)
 The kernels $\Gamma$ and $K$ coincide, so:
\[
\begin{array}{rl}
\Gamma = &\Gamma_l=\{a \in G\mid {}^au = u, \; \forall u \in G\}
       = \Gamma_r=\{a \in G\mid u^a =u, \; \forall u \in G\} \\
       =  &K =\{a \in G\mid {}^ax = x, \; \forall x \in X\}.
       \end{array}
\]
(ii) $\Gamma = K$ is an abelian normal subgroup of  $G$.
(iii)  $\widetilde{G}=
G/\Gamma$
is
a quotient symmetric group of $(G, r_G)$, in the sense of Takeuchi, so $\Ret (G, r_G) \simeq \widetilde{G}=
G/\Gamma$ is an isomorphism of symmetric groups.
\end{remark}

Lemma \ref{Lemma_retG}  and the equality of the two kernels $\Gamma= K$ imply the following.
\begin{corollary}
\label{prop_RetG_Gcal}
Let $(X,r)$ be a non-degenerate symmetric set, notation as above.
\begin{enumerate}
\item
\label{prop_RetG_Gcal1}
There is an equality $\soc G = K$ and an isomorphism of symmetric groups
\begin{equation}
\label{RetG=Gcal}
\Ret(G,r_G) \simeq \Gcal(X,r).
\end{equation}
\item
\label{prop_RetG_Gcal2}
If $\mpl (G,r_G) = m < \infty$ then
\begin{equation}
\label{prop_RetG_Gcal_eq1} m = \mpl (G,r_G) \geq \mpl G(\Ret
(X,r)) \geq \mpl  (\Gcal, r_{\Gcal})= m -1.
\end{equation}
\item
\label{prop_RetG_Gcal3}
If the solution $(X,r)$ is finite then the retraction $\Ret(G,r_G)\simeq \Gcal(X,r)$ is a finite symmetric group.
\end{enumerate}
\end{corollary}
The equality $\Gamma= K$ for the symmetric groups $G= G(X,r)$ was published first by Takeuchi in  2003, (\cite{Takeuchi}, see p. 15).
A natural questions arises.

\emph{How can we express the higher retractions $\Ret^j(G,r_G)$, $   j \geq 2$? In particular, what is the retraction  $\Ret(\Gcal,
r_{\Gcal})\simeq \Ret^2(G,r_G)$?}

 We shall answer this questions in terms of derived permutation groups of $(X,r)$, see Definition
 \ref{derivedSymgr_def}, Lemma \ref{socGcal_Lemma} and Theorem \ref{longsequencepropThm}.

\subsection{Derived symmetric groups and derived permutation groups of a solution}
In assumption and conventions as above we collect a list of
notation which will be used throughout the paper. The existence of
all objects and maps given below is proven in Proposition
\ref{longsequenceprop}. Here $j \geq 0$ is an integer.

\begin{notation}
\label{notation1}
\item
We set $\Ret^0(X,r)=(X,r)$, $\Ret^j(X,r)= Ret(\Ret^{j-1}(X,r))$ is the $j$-th retraction of $(X,r)$, $j \geq 1$,  but
 when $j=1$ we
 use both notations
$\Ret(X,r)=\Ret^1(X,r)$ and $([X], r_{[X]})$.
\item
$x^{(j)}$ denotes the image of $x$ in $\Ret^j(X,r)$ ($x^{(0)}= x$).
\item
 $G_j:= G(\Ret^j(X,r))$,
  $G_0=G(X,r) =G$;
 $\; \Gcal_j:= \Gcal(\Ret^j(X,r))$, $\;\Gcal_0=\Gcal(X,r)= \Gcal.$
\item $\Lcal^j: G_j \longrightarrow \Gcal_j$ is the epimorphism
extending the assignment $x^{(j)} \mapsto \Lcal_{x^{(j)}} \in \Sym
(\Ret^j(X,r)),\; x \in X$, $\Lcal^{0}= \Lcal: G
\longrightarrow \Gcal$, extends $x \mapsto \Lcal_x, x \in X.$
\item
 $\Kcal_{j}$
is the pull-back of $\ker \Lcal^j$ in $G,$ in particular $\Kcal_0=
\ker \Lcal$. \item $\nu_{j}: G_{j}  \longrightarrow G_{j+1}$ is
the epimorphism extending the assignment $x^{(j)} \mapsto
x^{(j+1)}$. $N_j$ is the pull-back of $\ker \nu_j$ in $G$,
$N_0=\ker \nu_0$. \item $\varphi_j: \Gcal_{j}   \longrightarrow
\Gcal_{j+1}$ is the epimorphism extending  the assignments
$\Lcal_{x^{(j)}} \mapsto \Lcal_{x^{(j+1)}},\; x\in X$, see
Proposition \ref{longsequenceprop}. $H_j$ is the pull-back of
$\ker \varphi_j$ in $G$.
\end{notation}

\begin{lemma}
\label{retlemma1}
 In assumption and notation as above
 the following conditions hold.
\begin{enumerate}
\item
\label{rethom1}
The canonical epimorphism of solutions
$\nu_0: (X,r) \longrightarrow ([X], r_{[X]}), x \mapsto [x]$,
extends to a group epimorphism $\nu_0: G_0 \longrightarrow G_1.$
Analogously, for each $j \geq 1$ there exists an epimorphism of symmetric groups $$\nu_j: (G_j, r_{G_j}) \longrightarrow (G_{j+1},
r_{G_{j+1}}).$$
\item
\label{rethom2} There is a canonical epimorphism of symmetric groups \[\varphi_0: \Gcal_0
\longrightarrow \Gcal_1,\quad\Lcal_x \mapsto \Lcal_{[x]}, \; \forall
x\in X. \]
\item
\label{rethom4} The subgroups $N_0, \Kcal_0= K,  \Kcal_1, H_0$ of $G$ are ideals of the symmetric group $(G,r_G)$.
One has
\begin{equation}
\label{kernel2}
\begin{array}{c}
N_0 \subsetneqq \Kcal_0 \subset \Kcal_1= H_0 \\
 \ker \nu_1 \simeq N_1/N_0; \quad\quad  \ker \Lcal^1\simeq \Kcal_1/N_0; \quad\quad \ker \varphi_0 \simeq  \Kcal_1/\Kcal_0.
\end{array}
\end{equation}
In particular, $N_0$ and $\Kcal_0$ are  abelian. \item
\label{rethom5} There is a canonical epimorphism of symmetric
groups
\[f_0: G_1 \longrightarrow \Gcal_0, \; [x] \mapsto \Lcal_{x}, \; x\in X, \; \text{where}\; \ker f_0 \simeq \Kcal_0/ N_0, \;
\text{and}\;f_0\circ \nu_0 = \Lcal^0 .\]
\item
\label{rethom6} There are short exact sequences:
\begin{equation}
\label{vipexactsequences}
\begin{array}{rl}
1 \longrightarrow N_0 \longrightarrow G
\stackrel{\nu_0}{\longrightarrow} G_1 \longrightarrow 1\quad \quad&
1 \longrightarrow \Kcal_0 \longrightarrow G
\stackrel{\Lcal^{0}}{\longrightarrow} \Gcal_0 \longrightarrow 1
\\
1 \longrightarrow N_1/N_0 \longrightarrow G_1
\stackrel{\nu_1}{\longrightarrow}   G_2 \longrightarrow 1\quad
\quad& 1 \longrightarrow \Kcal_0/N_0 \longrightarrow G_1
\stackrel{f_0}{\longrightarrow}  \Gcal_0 \longrightarrow 1
\\
1 \longrightarrow K_1/N_0 \longrightarrow G_1
\stackrel{\Lcal^1}{\longrightarrow} \Gcal_1 \longrightarrow 1 \quad
\quad& 1 \longrightarrow \Kcal_1/\Kcal_0 \longrightarrow \Gcal
\stackrel{\varphi_0}{\longrightarrow} \Gcal_1 \longrightarrow 1.
\end{array}
\end{equation}
\item
\label{vipcomdiagram0}
The following diagram is commutative:
\[\]
\begin{equation}
\label{vipcomdiagram} \setlength{\unitlength}{0.7cm}
\begin{picture}(10,8)
\put(0,2){$1$} \put(0,4){$1$} \put(0.8,2){$\rightarrow$}
\put(0.8,4){$\longrightarrow$} \put(2,0){$1$} \put(1.5,2){$\Kcal_1/\Kcal_0$}
\put(2,4){$N_0$} \put(3,1){$\swarrow$} \put(3.2,2){$\rightarrow$}
\put(3,4){$\longrightarrow$} \put(3,5){$\nearrow$} \put(4,0){$1$}
\put(4,1){$\downarrow$} \put(4,2){$\mathcal{G}_0$}
\put(4,3){$\downarrow$}
\put(3.3,2.9){$\scriptstyle{\mathcal{L}^{0}}$} \put(4,4){$G_0$}
\put(4,5){$\downarrow$} \put(4,6){$\Kcal_0$} \put(4,7){$\downarrow$}
\put(4,8){$1$} \put(5,2){$\longrightarrow$}
\put(5.2,2.3){$\scriptstyle{\phi_0}$} \put(5,3){$\swarrow$}
\put(4.8,3.2){$\scriptstyle{f_0}$} \put(5,4){$\longrightarrow$}
\put(5.2,4.3){$\scriptstyle{\nu_0}$} \put(5,7){$\Kcal_1$}
\put(4.5,6.5){$\nearrow$} \put(5.5,6.5){$\searrow$} \put(6,0){$1$}
\put(6,1){$\downarrow$} \put(6,2){$\mathcal{G}_1$}
\put(6,3){$\downarrow$}
\put(6.3,2.9){$\scriptstyle{\mathcal{L}^{1}}$} \put(6,4){$G_1$}
\put(6,5){$\downarrow$} \put(5.8,6){$\Kcal_1/N_0$}
\put(6,7){$\downarrow$} \put(6,8){$1$} \put(7,2){$\longrightarrow$}
\put(7,4){$\longrightarrow$} \put(7,5){$\swarrow$} \put(8,2){$1$}
\put(8,4){$1$} \put(7.8,6){$\Kcal_0/N_0$} \put(9,7){$\swarrow$}
\put(10,8){$1$}
\end{picture}
\end{equation}
\end{enumerate}
\end{lemma}

\begin{proof}
(\ref{rethom1}). The canonical map $\nu_0 : X \longrightarrow
[X],\; x \mapsto [x],$  from $X$ to its retract is \emph{a
braiding-preserving map}, (epimorphism of solutions) therefore, by
\cite{LYZ}, Proposition 6,   its canonical extension $\nu_0: G
\longrightarrow G_1$  is an epimorphism of symmetric groups (that
is \emph{a braiding preserving group epimorphism}) $\nu_0: (G,
r_G) \longrightarrow (G_1, r_{G_1})$. Condition (\ref{rethom2}) is
clear. We shall verify  (\ref{rethom4}).
 We know that $K=\Kcal_0 = \Gamma = \soc G$, so it is an ideal of $G.$
The kernel $N= N_0=\ker \nu_0$ is a normal subgroup of $G$ which
consists of all $a \in G,$ such that $[a]= 1_{G_1}.$  Suppose $a
\in N.$ One has $[{}^ua] = {}^{[u]}{[a]}= {}^{[u]}{1_{G_1}}
=1_{G_1}$  hence $N$ is left $G$-invariant, and therefore $N$ is
an ideal of $G$. To verify $N_0\subset K$ it will be enough to
show
 the  implication
\[
[a]= 1_{G_1} \Longrightarrow \Lcal_a = \id_X.
\]
 Note first that the retraction $[X]$ is embedded in $G_1$, and $[x]\neq 1_{G_1}$, for all $x \in X.$
Suppose $a \in N$, $a \neq 1$.  Then $a$  has a reduced length
$|a|\geq 2$, and can be written as $a = uv$, where $u,v \in G,
|u|, |v|\geq 1$. One has $[a]= [uv]=[u][v] = 1_{G_1}$, hence $[u]
= [v]^{-1}= [v^{-1}]$. This gives $[a] = [v^{-1}][v].$ Let $x \in
X$, one has ${}^ax= {}^{v^{-1}}{({}^vx)} =x$. Therefore ${}^ax =
x,$ for every $x\in X,$ so $\Lcal _a= \id_X.$ This verifies
$N_0\subseteq \Kcal_0$. The socle  $ \Gamma = \Kcal_0$ is an
abelian group and so is $N_0$.

The kernel $\widetilde{\Kcal_1}$ of the map $\Lcal^1: G_1
\longrightarrow \Gcal_1$  is the socle
 of $G_1,$ hence, an ideal of $G_1$. But, by definition $\Kcal_1= \nu_0^{-1}(\widetilde{\Kcal_1})$
 ($\widetilde{\Kcal_1}= \nu_0 (\Kcal_1)$) so by the Third Isomorphism theorem,  $\Kcal_1$  is also an ideal   of
 $(G, r)$, which contains $N_0$.
The equality $H_0= \Kcal_1 $ follows from the  implications:
\[
\begin{array}{lll}
u \in H_0 &\Longleftrightarrow \Lcal_{[u]} =
\id_{[X]}
&\Longleftrightarrow {}^{[u]}{[x]} = [{}^ux]=[x],\quad  \forall x \in
X\\
&\Longleftrightarrow {}^{({}^ux)}z = {}^xz, \quad \forall x,z \in X
&\Longleftrightarrow \Lcal_{({}^ux)} = \Lcal_{x},\quad \forall x \in X\\
&\Longleftrightarrow  u \in \Kcal_1.&
\end{array}
\]
Now the inclusions (\ref{kernel2}) for the three kernels  are clear.
This implies the second line in   (\ref{kernel2}). The existence of
the short exact sequences
 (\ref{vipexactsequences}) is straightforward from  (\ref{kernel2}).
It is easy to see that (\ref{rethom4}), (\ref{rethom5}), and
 (\ref{vipcomdiagram0}) are in force.
 \end{proof}
 \begin{remark}
\label{retrem1}
 Suppose $(X, r)$ is a solution of
finite order, which is not a permutation solution
($\mpl X \geq 2$ is not necessarily finite).
 Then  $K_0$ is a normal subgroup of $G$
of finite index $p=[G:K_0 ]$.  In contrast, the index $[G:N_0]$ of
$N_0$ is not finite.
Indeed, the retraction
 $([X], r_{[X]})$ is a solution of order $> 1,$ which generates the
group $G_{1}=G([X], r_{[X]}).$ Note that $[x] \neq 1_{G_{[X]}}$, for
all $x \in X$. The group $G_1$ is torsion free as a YB group of
a  solution of order $> 1$. In particular,
$[x^p]=[x]^p \neq 1_{G_{[X]}},$  so $ x^p$ is not in
$N_0$,  $\forall x \in X$. On the other hand
$x^p \in \Kcal_0, \;\forall x \in X.$
\end{remark}

\begin{lemma}
\label{socGcal_Lemma}
There is an isomorphism of symmetric groups:
\begin{equation}
\label{soc_eq_Lemma}
    \Ret(\Gcal, r_{\Gcal}) \simeq \Gcal(\Ret(X,r)),\quad \text{where} \; \soc(\Gcal) \simeq \Kcal_1/K_0.
\end{equation}
\end{lemma}
\begin{proof}
Recall that $\Gcal(\Ret(X,r))$ was denoted as $(\Gcal_1,
r_{\Gcal_1})$. We shall prove that $\soc(\Gcal) = \ker \varphi_0$,
where $\varphi_0: \Gcal_0 \longrightarrow \Gcal_1,\quad\Lcal_x
\mapsto \Lcal_{[x]}$, is the canonical epimorphism from Lemma
\ref{retlemma1}.  This follows from the implications
\[
\begin{array}{cll}
\Lcal_u \in  \soc(\Gcal) &\Longleftrightarrow  {}^{\Lcal_u}{(\Lcal_a)} = \Lcal_a, \forall a \in G &
\Longleftrightarrow \Lcal_{({}^ua)} = \Lcal_a, \forall a \in G \\
&\Longleftrightarrow {}^{[u]}{[a]}= [{}^ua]= [a], \forall a \in G
&\Longleftrightarrow \Lcal_{[u]} = id_{G_1}\\
&\Longleftrightarrow \Lcal_{[u]} = id_{[X]}
& \Longleftrightarrow \Lcal_{[u]} \in \ker \varphi_0\simeq H_0/\Kcal_0.
\end{array}
\]
 For the first implication on the last line we use (\ref{eq_ker_Soc1}).  By Lemma \ref{retlemma1} one has $H_0 = \Kcal_1$ and $\ker
 \varphi_0 \simeq  \Kcal_1/K_0$, hence $\soc(\Gcal)\simeq  \Kcal_1/\Kcal_0$. Moreover,
 there are isomorphisms of symmetric groups
$\Gcal_1 \simeq \Gcal/(\ker \varphi_0)=\Gcal/\soc(\Gcal) \simeq \Ret(\Gcal, r_{\Gcal})$.
\end{proof}
\begin{proposition}
\label{longsequenceprop} Let $(X,r)$ be a
solution.
\begin{enumerate}
\item
For all $j \geq 0$ there are canonical group
epimorphisms
\[
\begin{array}{llll}
\nu_j: G_{j} \longrightarrow G_{j+1}, \quad& x^{(j)} \mapsto
x^{(j+1)};\quad& \Lcal^j: G_j \longrightarrow \Gcal_j, \quad&
x^{(j)} \mapsto
\Lcal_{x^{(j)}};\\
f_j: G_{j+1} \longrightarrow \Gcal_{j}, \quad& x^{(j+1)} \mapsto
\Lcal_{x^{(j)}}; \quad& \varphi_j: \Gcal_{j} \longrightarrow
\Gcal_{j+1}, \quad& \Lcal_{x^{(j)}} \mapsto \Lcal_{x^{(j+1)}}.
\end{array}
\]
\item
\label{H0} For $j \geq 0$ let $N_j$, (respectively, $\Kcal_j,
H_j$) be the pull-back in $G$ of the kernel $\ker\nu_j,$
(respectively, the pull-back of $\ker\Lcal^j$, $\ker \varphi_j$).
These are ideals of $(G, r)$ and  there are inclusions
\[
\begin{array}{ccccccccccccc}
N_0 & \subset & N_1 & \subset & N_2 & \subset & \cdots & \subset & N_j & \subset & N_{j+1} & \subset & \cdots \\[5pt]
\bigcap & & \bigcap & & \bigcap & & & & \bigcap & & \bigcap & & \\[5pt]
\Kcal_0 & \subset & \Kcal_1 & \subset & \Kcal_2 & \subset & \cdots & \subset & \Kcal_j & \subset & \Kcal_{j+1} & \subset & \cdots
\\[5pt]
& & \| & & \| & & & & \| & & \| & & \\[5pt]
& & H_0 & \subset & H_1 & \subset & \cdots & \subset & H_{j-1} &
\subset & H_j & \subset & \cdots
\end{array}
\]
The kernels satisfy
\[
\begin{array}{ll}
\ker\nu_j \simeq N_j/ N_{j-1}\quad \quad & \ker \Lcal^j \simeq \Kcal_j/N_{j-1} \\
\ker f_{j}\simeq \Kcal_{j}/N_j,                 &\ker \varphi_j \simeq \Kcal_{j+1}/ \Kcal_{j}  \simeq H_{j}/ H_{j-1},\\
&\text{where} \; N_{-1}: = \{1\}=:H_{-1}.
\end{array}
\]
\item  One has
\[\begin{array}{ll}
  G_j \simeq G/N_{j-1},\quad &
  \soc G_j  \simeq \Kcal_j/ N_{j-1} \\
  \Gcal_j \simeq G/\Kcal_j, & \soc(\Gcal_j) \simeq  \Kcal_{j+1}/ \Kcal_{j}\simeq H_{j}/ H_{j-1}, \forall j \geq 0.
  \end{array}
  \]
 So there are  isomorphisms of symmetric groups:
\[
\begin{array}{ll}
 \Ret (G_j, r_{G_j}) \simeq &(\Gcal_j, r_{\Gcal_j}), \quad\Ret (\Gcal_j, r_{\Gcal_j}) \simeq (\Gcal_{j+1},
 r_{\Gcal_{j+1}}),
   \;  \forall j \geq 0.
   \end{array}
     \]
      \end{enumerate}
\end{proposition}
\begin{proof}
 Applying Remark  \ref{factGamma=K} to the solution $([X], r_{[X]})$ and its symmetric group $G_1= G([X], r_{[X]})$, and more generally,
 to
$\Ret^j(X,r)$ and its symmetric group $G_j = G(\Ret^j(X,r))$, $j
\geq 1$, one yields
\begin{equation}
\label{eq_ker_Soc1}
\begin{array}{ll}
\Kcal_1 =&\{u \in G\mid \Lcal_{[u]} = \id _{[X]}\} = \{u \in G\mid \Lcal_{[u]} = \id _{G_1}\}\\
\Kcal_j =& \{u \in
G\mid \Lcal_{(u^{(j)})} = \id _{\Ret^j(X,r)}\} =\{u \in
G\mid \Lcal_{(u^{(j)})} = \id _{G_j}\}, \; j \geq 1.
\end{array}
\end{equation}
A routine iteration of Lemmas \ref{retlemma1}, and
\ref{socGcal_Lemma} verifies the proposition.
\end{proof}
\begin{theorem}
\label{longsequencepropThm} Let $(X,r)$ be a solution, notation as
above.
\begin{enumerate}
\item The retractions of the symmetric group $(G, r_G)$ satisfy
     \begin{equation}
     \Ret\,^{j+1}(G, r_G) \simeq \Gcal(\Ret\,^j(X,r)), \; \text{for all}\;  j \geq
     0.
\end{equation}
\item The  following diagram is commutative:
\begin{equation}
\label{Diagram1} \xymatrix{ G_0\ar[d]_{\Lcal^0}\ar[r]^{\nu_0}
&G_1\ar[dl]_{f_0}\ar[d]_{\Lcal^1}\ar[r]^{\nu_1}
&G_2\ar[dl]_{f_1}\ar[d]_{\Lcal^2}\ar[r]^{\nu_2}&\dots\ar[r]&G_{m-1}\ar[dl]_{f_{m-2}}
\ar[d]_{\Lcal^{m-1}}\ar[r]^{\nu_{m-1}}&G_{m}\ar[dl]_{f_{m-1}}\ar[d]_{\Lcal^{m}}\ar[r]^{\nu_{m}}&\cdots\\
\Gcal_0\ar[r]_{\varphi_0}  &\Gcal_1\ar[r]_{\varphi_1}
&\Gcal_2\ar[r]&\dots\ar[r]_{\varphi_{m-2}}&\Gcal_{m-1}\ar[r]_{\varphi_{m-1}}
&\Gcal_{m}\ar[r]_{\varphi_{m}} &\cdots }
\end{equation}
\item The derived chain of ideals of $G$ satisfies
\begin{equation}
\label{derchain_(G(X,r)eq}
\{1\}=  K_0\subsetneqq K_1= \Kcal_0 \subseteq K_2 = \Kcal_1\subseteq \cdots \subseteq K_{m} = \Kcal_{m-1}\subseteq \cdots.
\end{equation}
The following implications are in force
\[ (\mpl G = m+1)\Longleftrightarrow (\Kcal_{m-1}\subsetneqq \Kcal_m = G)\Longleftrightarrow (\mpl \Gcal = m).  \]
      \end{enumerate}
\end{theorem}
\begin{proof}
The theorem follows from Proposition \ref{longsequenceprop}.
\end{proof}
We introduce derived symmetric groups and derived
permutation groups of a solution, these are in the spirit of the
notion derived chain of ideals of a symmetric group.
\begin{definition}
\label{derivedSymgr_def}
Let $(X,r)$ be a solution. We shall call
$(G_j, r_{G_j})=G(\Ret^j(X,r))$,  $j \geq 0$, \emph{the derived symmetric groups of} $(X,r)$.
Similarly,  the symmetric groups
$(\Gcal_j, r_{\Gcal_j}) =\Gcal(\Ret^j(X,r)), j \geq 0,$  will be called \emph{the derived permutation  groups of} $(X,r).$
As usual $(G_j, +, \cdot)$ and $(\Gcal_j, + , \cdot)$ denote the corresponding derived braces.
\end{definition}

It follows from Proposition \ref{longsequenceprop}, and Theorem
\ref{longsequencepropThm} that every solution $(X,r)$ has two
sequences of derived symmetric groups,$\{(G_j, r_{G_j})| j \geq
0\}$ and $\{(\Gcal_j, r_{\Gcal_j})\mid j \geq 0\}$. Each of these
sequences is an invariant of the solution and reflects the process
of retraction. The sequences are closely related and satisfy
\[(\Gcal_j, r_{\Gcal_j}) \simeq  \Ret(G_j, r_{G_j}),  \; \Ret(\Gcal_j, r_{\Gcal_j})\simeq (\Gcal_{j+1},
r_{\Gcal_{j+1}}),\;
 \Ret\,^{j+1}(G, r_G) \simeq (\Gcal_j, r_{\Gcal_j}). \]
In general, each of the derived sequences may have repeating
members, see Example \ref{ex_Gcal=Gcal1}. In the special case when
$(X,r)$ is non-retractable one has $(G_j, r_{G_j})= (G, r_{G})$,
and $(\Gcal_j, r_{\Gcal_j}) = (\Gcal, r_{\Gcal}),$ for all $j \geq
0$. Each derived symmetric group and each derived permutation
group encodes various combinatorial properties of the solution
$(X,r)$ and may have strong impact on it. This is illustrated by
Corollary \ref{cor_Gcal_twosided}, (2), Theorem \ref{Thm_main},
and Corollary \ref{Cor_main}.

\begin{corollary}
\label{cor3}  If the groups $\Gcal (X,r)$ and $\Gcal_1 (X,r)$ are not isomorphic, then
 $(X,r) \neq \Ret(X,r)$ and therefore the solution $(X,r)$  is retractable.
\end{corollary}

\begin{remark}
There are numerous examples of solutions $(X,r)$ such that
  $\Ret(X,r) \neq (X,r)$, but $\Gcal = \Gcal (X,r)\simeq \Gcal(\Ret(X,r)) = \Gcal_1$. In fact any trivial extension $(X,r)= Y \bigcup Z$
  of an irretractable solution $(Y, r_Y)$, $|Y|\geq 2$,
with  a trivial solution $(Z, r_Z)$, where $|Z|\geq 2$, satisfies $(X,r) \neq \Ret(X,r)$ but  $\Gcal(X,r) \simeq \Gcal(\Ret(X,r))$.
Example \ref{ex_Gcal=Gcal1} presents a concrete solutions with these properties.
 \end{remark}

The unique non-retractable square-free solution $(Y, r_Y)$ of order $8$ was noticed first by Vendramin in the complete list of solutions
of order $8$ presented by Schedler
(who used a computer programme). Vendramin wrote: "It is remarkable that this is the only counterexample to the Gateva-Ivanova's
conjecture among the 2471 square-free solutions of size $|X| = 8$!". We recall its definition below.
\begin{remark}
\label{Rem_ex_Ve}
Let $(Y, r_Y)$ be the square-free solution of order $8$ given in \cite{V15} Example 3.9. It is defined as
\label{ex_Ve}
\begin{equation}
\label{ex_vendramin}
\begin{array}{llll}
 &Y = \{x_1,x_2,x_3,x_4, x_5, x_6, x_7, x_8\}, & &r(x,y)= (\Lcal_x (y), \Lcal_y^{-1} (x)),\\
 \Lcal_{x_1} =&  (x_5x_7), \Lcal_{x_3} = (x_5x_7)(x_2x_6)(x_4x_8),& \Lcal_{x_5} =  (x_1x_3), &\Lcal_{x_7} = (x_1x_3)(x_2x_6)(x_4x_8),\\
 \Lcal_{x_2} =&  (x_6x_8), \Lcal_{x_4} = (x_6x_8)(x_1x_5)(x_3x_7),& \Lcal_{x_6} =  (x_2x_4), &\Lcal_{x_8} = (x_2x_4)(x_1x_5)(x_3x_7).\\
\end{array}
\end{equation}
Clearly, the solution $(Y, r_Y)$ is irretractable.  One has $\Gcal =\Gcal (Y, r_Y\simeq D_4$, the dihedral group of $8$ elements, see
\cite{V15}. The set $Y$ splits  into two $\Gcal$-orbits: $Y_1 = \{x_1,x_3,x_5, x_7\}$ and $Y_2 = \{x_2,x_4,  x_6,  x_8\}$, denote the
induced solution by
$(Y_1, r_1)$, $(Y_2, r_2)$.  Then
$\mpl(Y_i, r_i)= 2, i = 1,2$ . It is easy to see that $(Y, r_Y)= Y_1 \stu Y_2$ is a strong twisted union
of solutions of multipermutation level $2$.

Therefore $(Y,r_Y)$ provided the first negative answer to a question posed by  the author and Cameron: "Is it true that
a strong twisted union of two multipermutation solutions is also a multipermutation solution?", \cite{GIC}, Open Questions 6.13.
\end{remark}
\begin{example}
\label{ex_Gcal=Gcal1} Let $X = Y\stu_0 Z$ be the trivial strong
twisted union of the solutions $(Y, r_Y)$ and $(Z,r_Z)$, where
$(Y, r_Y)$ is the irretractable solution (\ref{ex_vendramin}), and
$(Z, r_Z)$ is the trivial solution on the two-elements set $Z =
\{a, b\}$. By definition $X = Y\bigcup Z$ and the map $r$ is
defined as $r(u,v) = (\Lcal_u (v), \Lcal_v^{-1} (u)), u, v \in X,$
where $\Lcal_a = \Lcal_b=id_X$, and the maps $\Lcal_y, y \in Y$,
are given in  (\ref{ex_vendramin}). Thus $\Lcal_y \neq \Lcal_v$,
whenever $y\in Y, v \in X$, $v\neq y$. The permutation group
$\Gcal(X,r)$ is a subgroup of $\Sym (X)= S_{10}$ and is generated
by $\Lcal_u, u \in X$. In fact, each map $\Lcal_y$ considered as
permutation of $X$ moves only elements of $Y$, so $(\Lcal_y)_{|Y}=
\Lcal_y$, and the set $\Lcal_y, y \in Y$ generates the whole group
$\Gcal(X,r)$. Therefore $\Gcal(X,r) \simeq\Gcal (Y, r_Y)\simeq
D_4$. It is clear that  $[a]= [b]$ in the retraction $([X],
r_{[X]})$, and $[y]\neq [v]$, whenever $y \in Y, v \in X, y \neq
v$, hence the retraction $([X], r_{[X]})=\Ret (X,r)$ is a solution
with cardinality $9$. More precisely,  $\Ret (X,r)= Y\stu_0
\{[a]\}$, is a strong twisted union of $Y$ with a one element
solution. Obviously, $(X,r)$ is not isomorphic to $\Ret (X,r)$,
but $\Gcal_1 = \Gcal (\Ret (X,r))\simeq \Gcal \simeq D_4 $ and the
socle $\soc (\Gcal) = \{1\}$ is trivial.
 In contrast the groups $G= G(X, r)$ and $G_1 = G(\Ret(X,r))$ are not isomorphic. Moreover $\Ret^2(X,r) = \Ret^1(X,r)$, so $\Ret^1(X,r)$
 is non-retractable and the process of retraction of $(X,r)$ halts after the first step.
The derived chain of ideals satisfies $1 \neq K_0 = K_1 \subsetneqq G$.
\end{example}
It is known that  the minimal order of a non-retractable solution $(X,r)$ is $|X|= 4$.
\begin{lemma}
\label{retlemma2} Assumptions and notation as above. The diagram (\ref{Diagram1}) is finite \emph{iff} there exists an integer $m \geq
0,$ such that $\Ret ^m(X, r)= \Ret ^{m+1}(X, r).$  Suppose this is the case,  and let $m$ be the minimal integer with this property.
Exactly one of the following two conditions is satisfied:
\begin{enumerate}
 \item
 The retraction $\Ret ^m(X, r)$ is a one element solution, so $\mpl X = m\geq 1$. Then $m \leq \mpl G(X, r) \leq m + 1$, and the
 diagram has the shape
\begin{equation}
\label{Diagram2}
\xymatrix{ G_0\ar[d]_{L^0_0}\ar[r]^{\nu_0}
&G_1\ar[dl]_{f_0}\ar[d]_{L^1}\ar[r]^{\nu_1}
&G_2\ar[dl]_{f_1}\ar[d]_{L^2}\ar[r]^{\nu_2}&\dots\ar[r]&G_{m-2}\ar[dl]_{f_{m-3}}
\ar[d]_{\Lcal^{m-2}}\ar[r]^{\nu_{m-2}}&G_{m-1}\ar[dl]_{f_{m-2}}\ar[d]_{\Lcal^{m-1}}\ar[r]^{\nu_{m-1}}&G_m=\langle \xi\rangle \simeq
C_{\infty}\ar[dl]_{f_{m-1}}\ar[d]_{\Lcal^{m}}\\
\Gcal_0\ar[r]_{\varphi_0}  &\Gcal_1\ar[r]_{\varphi_1}
&\Gcal_2\ar[r]&\dots\ar[r]_{\varphi_{m-3}}&\Gcal_{m-2}\ar[r]_{\varphi_{m-2}}
&\Gcal_{m-1}\ar[r]_{\varphi_{m-1}} &\Gcal_{m} = 1},
\end{equation}
where $\xi = x^{(m)}$ is the $m$-th retraction of an arbitrary $x
\in X$, and $\Gcal_{m-1}$ is a finite cyclic group ($\Gcal_{m-1}=
1$ is possible). \item
 $\Ret ^m(X, r)$  is a  non-retractable solution of order $n\geq 4$, $G_m = G_{m+1}$ is an $n$-generated torsion-free group. Moreover
 $\Gcal_{m} = \Gcal_{m+1}$ is a  nonabelian finite group with $\soc \Gcal_{m} = \{1\}$.
\end{enumerate}
 \end{lemma}

\begin{theorem}
\label{Th_mplG_mplX}
Let $(X,r)$ be a non-degenerate symmetric set of order $|X|\geq 2$ (not necessarily finite), notation as above.
\begin{enumerate}
\item
\label{mplXmplG1a}
The symmetric group  $(G, r_G)$
has finite multipermutation level \emph{iff}
$(X, r)$
is a multipermutation solution. In this case one has
\begin{equation}
\label{mplXmplG3}
0 \leq \mpl (\Gcal,r_{\Gcal})= m -1 \leq \mpl (X, r) \leq \mpl  (G, r_G) = m < \infty.
\end{equation}
\item
Suppose furthermore that $(X,r)$ satisfies condition (*).  Then
\begin{equation}
\label{mplXmplG3a}
\mpl X = m < \infty\; \text{\emph{iff}}\;  \mpl G = m < \infty.
\end{equation}
In particular,  (\ref{mplXmplG3a}) is in force for square-free solutions $(X,r)$ of arbitrary cardinality.
 \end{enumerate}
\end{theorem}
\begin{proof}
(1). Consider $(G, r_G)$ as a symmetric set. Then $X$ is an $r_G$-invariant
subset of $G$.
Suppose $(G, r_G)$ has a finite multipermutation level  $\mpl G =
m<\infty$. Note that the left action of $G$ upon itself extends canonically the left action of $G$ upon
$X$. Moreover, the equality $K = \soc G= \Gamma$ implies that two elements $x, y\in X$ have the same action $\Lcal_x = \Lcal_y$ on $X$
if and only if they have the same action on $G$.
Our assumption $\mpl G =m$ implies
that (\ref{mplmequality}) is in force for every choice of  $x,\; y_1,
\cdots, y_m \in G$. But $X$ is embedded in $G$,
 hence,
(\ref{mplmequality}) is satisfied by any $x,\; y_1, \cdots, y_m \in X$, and
therefore by Proposition \ref{mplmtheorem1} $\mpl X \leq m.$
This verifies the  implication
\begin{equation}
\label{mplXmplG4}
\mpl G =
m<\infty\Longrightarrow 1\leq \mpl X \leq   \mpl G.
\end{equation}

Assume now that $(X,r)$ is a multipermutation solution of level $m = \mpl X$.  Then, by Lemma \ref{retlemma2}
 $\Gcal_{m} = \{1\}$ is a one element solution.
  But $\Gcal_{m} \simeq \Ret^{m+1} (G, r_G),$ hence $\mpl G \leq m+1$. More precisely, either (a) $\Gcal_{m-1}$ is a nontrivial finite
  cyclic group, then $\mpl G = m+1$, and $\mpl \Gcal = \mpl X = m$; or (b) $\Gcal_{m-1} = \{1\},$ then $\Ret^{m} (G, r_G)$ is a one
  element solution ($m$ is minimal with this property). In this case $\mpl X = m = \mpl G$, and $\mpl (\Gcal, r_{\Gcal}) = m-1.$
(2). Suppose $(X,r)$  satisfies condition (*), and $\mpl X = m.$
Then $\Ret^{m-1}(X,r)$ is a trivial solution of order $\geq 2$,
$G_{m-1}$ is a free abelian group of finite rank, hence
$\Gcal_{m-1} = \{1\}$. By Theorem \ref{longsequencepropThm}
$\Ret^{m}(G, r_G) \cong\Gcal_{m-1}$ so it is a one element
solution, and if $m \geq 2$, $\Ret^{m-1}(G, r_G) =\Gcal_{m-2}$ has
order $\geq 2$. Therefore $\mpl G = m = \mpl X $.
   \end{proof}
\begin{corollary}
\label{cor_Gcal_twosided}
Let $(X,r)$ be a finite solution, notation as usual.
\begin{enumerate}
\item
 Suppose the left brace  $(\Gcal, +,\cdot)$ is a nontrivial two-sided brace.
 \begin{enumerate}
 \item
 \label{cor_Gcal_twosided_a}
 The corresponding radical ring $\Gcal^{\star}$
is nilpotent, say of class $m >1$;
\item
 \label{cor_Gcal_twosided_b}
In this case $(X,r)$ is a multipermutation solution, with $m-1 \leq \mpl X \leq m$, and the inequalities (\ref{mplXmplG3}) are in
force.
Moreover, if $(X,r)$ is a square free-solution, then \[\mpl (X,r)= m= \text{the nilpotency degree of}\;\;  \Gcal^{\star}. \]
\end{enumerate}
\item
$(X,r)$ is a multipermutation solution with $\mpl X \geq 2$ \emph{iff} there exist some $j \geq 0$ such that $(\Gcal_j, +,\cdot)$ is a
two-sided brace of order $> 1$.
\end{enumerate}
\end{corollary}
\begin{proof}
(1). It is well-known that the Jacobson radical of a left Artinian ring is nilpotent, see for example \cite{Lam},
Theor. 4.12. This straightforwardly implies that a finite radical ring $R$ is nilpotent, so  (\ref{cor_Gcal_twosided_a}) holds.
By Theorem \ref{derived_chainThm1}
$\mpl (\Gcal, r_{\Gcal})= m-1$ so  $\mpl (G, r_{G})= m$. Now part
(\ref{cor_Gcal_twosided_b}) follows straightforwardly from Theorem \ref{mplXmplG3}.

(2). Assume  $(\Gcal_j, +,\cdot)$ is a two-sided brace for some $j
\geq 0$, $|\Gcal_j|>1$. Part (1) implies that $\Ret^j(X,r)$ is a
multipermutation solution, say $\mpl\Ret^j(X,r)= s$, hence
$\mpl(X,r)= s+j< \infty.$ Conversely, suppose $\mpl X =
m\geq 2$. Then $m -1 \leq s= \mpl \Gcal \leq m$, so
$\mpl\Gcal_{s-1}=1$, and therefore, by Remark
\ref{mplB=1impliesBtwo-sided_rem} $\Gcal_{s-1}$ is a two-sided
brace.
\end{proof}

\section{A symmetric set $(X, r)$  whose brace $G(X, r)$  is two-sided, or square-free must be a
trivial solution}
\label{sec_two-sided_brace}
In this section $(X,r)$ denotes (a non-degenerate)  symmetric set,
not necessarily finite. As usual, $S=S(X,r)$ is the
associated YB  monoid, $(G, r_G)$ and
$(G, +,\cdot)$, are the associated symmetric group and left brace,
respectively, where $G=G(X,r)$. We shall show that each of the conditions: "$G$ is a
two-sided brace"; or "$(G, r_G)$ is a square-free solution", is
extremely restrictive. The main result of the section is the
following theorem.
\begin{theorem}
\label{VIPthm} Suppose $(X,r)$ is a non-degenerate symmetric set
(not necessarily finite). The following conditions are equivalent.
\begin{enumerate}
\item \label{VIPthm1} The associated left brace $(G, +, \cdot)$ is a
two-sided brace. \item \label{VIPthm2}
 The associated symmetric group $(G,r_G)$ is a square-free solution.
 \item
\label{lemma_sqfree1}
$(X,r)$ is a square-free solution and 
${}^{xy}{(xy)}= xy, \; \forall x,y \in G$, holds in $G$.  
\item \label{VIPthm4} The multiplicative group $G= G(X,r)$ is a free
abelian group generated by $X$, so $(G, r_G)$ is the trivial
solution.
\item \label{VIPthm5} $(X,r)$ is a trivial (square-free)
solution.
\end{enumerate}
\end{theorem}
The implication  (2) $\Longrightarrow$ (1) was proven independently in
[CJO14, Theorem 5]. Our proof of Theorem \ref{VIPthm} does not use their result. We apply our "symmetric groups -braces" approach to 
show a stronger statement:
each of the conditions (1) and (2) is equivalent to "$(X,r)$ is a trivial solution".
 We prove first some more results.
 \begin{lemma}
 \label{prop_twosidedbrace}
Let $(G, +, \cdot)$ be a left brace, and let $(G,r)$ be the
corresponding symmetric group. (1) $G$ is a two-sided brace if and
only if
\begin{equation}
\label{prop_twosidedbrace_eq} c({}^{(abc)^{-1}}c) =
({}^{b^{-1}}c)({}^{((a^b)({}^{b^{-1}}c))^{-1}}c),  \quad \forall
\; a,b, c \in G.
\end{equation}
(2) Assume that $(G,r)$ satisfies \textbf{lri}. Then $G$ is a
two-sided brace if and only if
\begin{equation}
\label{prop_twosidedbrace_eq2} c(c^{(abc)}) = (c^b)(c^{(a^bc^b)}),
\quad \forall \; a,b, c \in G.
\end{equation}
\end{lemma}
The proof of Lemma \ref{prop_twosidedbrace}  involves technical computations in the
symmetric group $(G, r)$ and can be found in the first two
versions of our preprint, see arXiv:1507.02602 [math.QA]. The
original proof of Theorem \ref{VIP_Thm_2-sided_brace_implies
triv_solution} can be also found in the first two versions of the same preprint. Here we give a second proof
whose idea was kindly suggested to the author by Ferran Ced\'{o}.

\begin{theorem}
  \label{VIP_Thm_2-sided_brace_implies triv_solution}
  Let $(X,r)$ be a symmetric set (not necessarily finite), and let
  $(G, +,\cdot)$, be the associated left brace.
 If $(G, +,\cdot)$ is a two-sided brace then
   $(X,r)$ is the trivial solution on the set $X$.
  \end{theorem}
  \begin{proof}
  Assume $(G, +, \cdot)$ is a two-sided brace and consider the associated Jacobson radical ring $G^{*}= (G, +, *)$, where
 $a* b := ab -a -b$ or equivalently, $ab= a*b+ a + b$. In particular,
 a left, and a right distributive law hold in $G^{\star}$. Moreover,
 by (\ref{br_radical_action1}), one has
 ${}^ab =  a*b  +b,   \forall a, b \in G$. Let $x,y \in X,$ then one has:
 \[
 {}^{(x+x)}y =  (x+x)* y + y = x* y + x* y + y = 2(x* y+ y) -y = 2({}^xy)- y.
 \]
 So the following identity is in force in the additive group $(G, +)$:
 \begin{equation}
 \label{rel}
 y + {}^{(x+x)}y - 2({}^xy) = 0, \; \forall x,y \in X.
   \end{equation}
Note first that $y, {}^{(x+x)}y, {}^xy \in X$.
  Recall that the set  $X$ is embedded in $G$ and the additive group
  $(G, +)$ is isomorphic to the (additive) free abelian group,
 ${}_{gr}[ X ]$ generated by $X$, see Remark \ref{remark_embeddingoffinitesolutions}.
Then (\ref{rel}) must be a trivial relation, which is possible if
and only if  ${}^xy = y = {}^{(x+x)}y.$ This proves that ${}^xy =
y$, for all $x,y \in X,$ and therefore $(X,r)$ is a trivial
solution.
\end{proof}

\begin{prooftheorem} \ref{VIPthm}.
The implications  (\ref{VIPthm5})$ \Longleftrightarrow $ (\ref{VIPthm4}) $ \Longrightarrow $(\ref{VIPthm1}), (\ref{VIPthm5}) $ \Longrightarrow $(\ref{VIPthm2})
 $ \Longrightarrow $(\ref{lemma_sqfree1}) are obvious.  Theorem \ref{VIP_Thm_2-sided_brace_implies triv_solution} verifies the
implication (\ref{VIPthm1})  $\Longrightarrow$ (\ref{VIPthm5}).

(\ref{lemma_sqfree1}) $\Longrightarrow$ (\ref{VIPthm5}).
Let $x,y \in X, x \neq y.$  Then $xy= {}^{xy}{(xy)}
=({}^{xy}x)({}^{(xy)^x}{y})$,
 by \textbf{ML2}. Therefore the following equality
 of monomials of length two holds in the monoid $S=
S(X,r)$:
\begin{equation}
\label{lemma_sqfree1eq2} xy=({}^{xy}x)({}^{(xy)^x}{y}).
 \end{equation}
Two cases are possible. \textbf{1.} The  equality
(\ref{lemma_sqfree1eq2}) holds in the free monoid $\langle
X\rangle$. In this case  ${}^{(xy)}x =x$. Moreover, ${}^{(xy)}x  =
{}^{x}{({}^yx)}= {}^xx$, which by  the non-degeneracy of $(X,r)$
implies ${}^yx = x$. \textbf{2.}  (\ref{lemma_sqfree1eq2}) is not
an equality of words in $\langle X\rangle$. Then it is a
(non-trivial) relation for $S$, therefore $({}^xy, x^y) = r(x,y) =
({}^{xy}x, {}^{(xy)^x}{y})$ holds in $X \times X$. Hence ${}^xy =
{}^{xy}x = {}^{x}{({}^yx)} $. The non-degeneracy again implies $y
= {}^yx$. But $(X,r)$ is a square-free solution, so $y={}^yy$, and
$ {}^yx = y = {}^yy$ imply $x = y,$ which contradicts our
assumption $x,y\in X, x \neq y.$ Therefore case \textbf{2} is
impossible. It follows that the equality ${}^yx = x$ holds for all
$x,y \in X,$ that is $\Lcal_y = id_X$, for all $y \in X$, hence
$(X,r)$ is  a trivial solution. $ \quad\quad\quad   \quad\quad\quad    \quad\quad\quad \quad\quad\quad \quad\quad \Box$
\end{prooftheorem}

\section{Symmetric groups  with conditions
\textbf{lri} and \textbf{Raut}}
\label{sec_actions}
In this section we introduce some special conditions on the
actions of a symmetric group $(G, r)$ upon
itself
and study their effect
on the properties of $(G, r)$ and the associated brace $(G, +,\cdot)$.

\subsection{Symmetric sets $(X, r)$  with condition \textbf{lri}}
\label{section_lri}
We present first some useful technical results.
As usual, $(X,r)$ denotes a non-degenerate symmetric set, $(G,r_G)$, and $(G,+, \cdot)$ denote its
associated symmetric group and left brace, respectively.

Recall that
condition \textbf{lri} holds on $(X,r)$ \emph{iff}
$(X,r)$ satisfies the cyclic conditions \textbf{cc}, see Corollary \ref{cor_lri_cc}.
Moreover, every square-free solution satisfies \textbf{lri}.
The permutation solutions, see Example \ref{exlri}, are a well known class of solutions with condition \textbf{lri}.
Recently, an interesting class of non-retractable solutions each of which satisfies \textbf{lri} but is not necessarily square-free were
introduced and studied in \cite{BCJO}.
Proposition \ref{prop_ActsAsAut&lri} shows that every symmetric group  $(G,r)$ with $\mpl G = 2$
satisfies \textbf{lri} and conditions (\ref{Lemma1_Thm_maineq1}) given below.
Note that if $(X,r)$ is with \textbf{lri}, then its  symmetric group  $(G,r_G)$ is a
symmetric set which, in general,  may not satisfy
condition \textbf{lri}. However, a mild generalization of  \textbf{lri} denoted by
\textbf{lri$\star$} is always in force, see Proposition \ref{VIPLemma}.  The special cases when some of the symmetric group $(G,r_G)$, or $(\Gcal, r_{\Gcal})$
satisfies \textbf{lri} will be studied in the next section.

\begin{lemma}
 \label{lemma_Aut_lri_sec4}
 Let $(X,r)$ be a nontrivial symmetric set of arbitrary cardinality, $|X|\geq 2$.
The following two conditions are equivalent.
  \begin{equation}
 \label{Lemma1_Thm_maineq1}
(i)\quad \Lcal_{{}^xa}= \Lcal_a, \quad  \forall a, x \in X;
\quad\quad   (ii)\quad   \Lcal_{a^x} =  \Lcal_a, \quad  \forall a,
x \in X.
 \end{equation}
Moreover, each of these conditions implies  condition \textbf{lri}.
\end{lemma}
\begin{proof}
 (i)   $\Longrightarrow$ (ii) and \textbf{lri}. Assume  (i)  is in
force, and let $a, x \in X$.
Applying first (\ref{involeq}) and then (\ref{Lemma1_Thm_maineq1}) (i)  we obtain
$x = {}^{{}^xa}{(x^a)}
 =  {}^{a}{(x^a)}.$
This gives  the first \textbf{lri} identity ${}^{x}{(a^x)}=a, \; \forall a, x
\in X$  which together with (\ref{Lemma1_Thm_maineq1}) (i) imply
 $\Lcal_a  = \Lcal_{({}^x{(a^x)})} =  \Lcal_{a^x},  \quad \forall \; a, x\in X,$ hence (ii) is in force.
We use (\ref{involeq}) and (ii) to yield
\[x= ({}^ax)^{(a^x)}=  ({}^{(a^x)}x)^{(a^x)},\]
hence (by the non-degeneracy) the second \textbf{lri} identity
$({}^ax)^a =x,  \quad \forall a, x \in X,$ is also in force.
Analogous argument verifies the implication (ii) $\Longrightarrow$
(i) and \textbf{lri}.
\end{proof}

\begin{proposition}
  \label{Lemma_TrivialOrbits}
 Let $(X,r)$ be a non-degenerate symmetric set, with associated symmetric group
 $G=G(X,r)$.
 Let $X_i, i \in I,$ be the set of all
$G$-orbits in $X$ (possibly infinitely many).
 Suppose that for each $i \in I$,
  $(X_i, r_i)$ is a trivial solution,
 or one element solution. Then
 \begin{enumerate}
 \item
 \label{Lemma_TrivialOrbits1}
 The following equalities hold  for all $y \in X, \; a, b \in X_i,  i\in I$:
  \begin{equation}
\label{Lemma_TrivialOrbitseq1}
{}^{{}^by}a ={}^ya, \quad\quad
{}^{y^b}a = {}^ya,\quad a^{y^b}=a^y,\quad  a^{{}^by}= a^y.
 \end{equation}
 \item
 \label{Lemma_TrivialOrbits2}
 $X$ is a strong
twisted union
$
X= \stu_{i \in I} X_i,
$
 in the sense of \cite{GIC}.
 \item
 \label{Lemma_TrivialOrbits3}
 $(X,r)$ satisfies condition \textbf{lri}.
 \end{enumerate}
  \end{proposition}
\begin{proof}
\textbf{(1)}.
By hypothesis each orbit is a trivial or one element solution,
   therefore
   \begin{equation}
  \label{Lemma_TrivialOrbitseq2}
 {}^ba =a = a^b,  \quad \forall a, b \in X_i.
  \end{equation}
 Assume now
 $a,b \in X_i, i \in I$, and $y \in X$.
 Then  $a, b, {}^ya, b^y \in X_i$ so we apply  (\ref{Lemma_TrivialOrbitseq2})
 twice to deduce
$ {}^ya ={}^b{({}^ya)} = {}^{{}^by}{({}^{b^y}a)} ={}^{{}^by}a.$
This verifies the first equality in (\ref{Lemma_TrivialOrbitseq1}).
  The proof of the remaining three equalities is analogous.
\textbf{(2)}. Clearly, $X$ is a disjoint union $X= \bigcup_{i \in I} X_i$ of
   its $G$-orbits.
  It follows from  (\ref{Lemma_TrivialOrbitseq1}) that for
  each pair $i,j \in I, i \neq j$, $X_i\natural X_j$ is a
  strong twisted union,  but this means that $X$ is a strong
  twisted
  union $X= \stu_{i \in I} X_i,$  in the sense of  \cite{GIC}.
  \textbf{(3)}.  We shall prove that \textbf{lri} holds.
  Let $a, x \in X$, say $x \in X_i, i \in I$.
Then $x, x^a\in X_i$, so  (\ref{involeq}) and (\ref{Lemma_TrivialOrbitseq1}) imply
$
 x = {}^{{}^xa}{(x^a)} = {}^{a}{(x^a)},
$
which gives the first \textbf{lri} equality, $x = {}^{a}{(x^a)}$. The second equality, $({}^ax)^a = x$, is proven analogously.
  \end{proof}

 \begin{remark}
 \label{Lemma_lri}
 If $(X, r)$ is a symmetric set with \textbf{lri} then
  \begin{equation}
  \label{eq_VIP_Lemma5}
  c(b^c) = b(c^b),   \quad\text{for all} \;\; b,c \in X.
                    \end{equation}
 \end{remark}
Indeed,
  conditions \textbf{M3}, \textbf{lri} and the cyclic condition \textbf{cc} imply:
$
  c(b^c) =  ({}^c{(b^c)})c^{b^c} =
       b(c^b).$
\begin{notation}
\label{X*}
$\quad X^{\star}:=X\bigcup X^{-1}, \quad\text{where}\quad  X^{-1}=\{ x^{-1}\mid x
\in X \}.$
\end{notation}

\begin{lemma}
 \label{lemma_lriSymSet}
 Let $(X,r)$ be a symmetric set,  $G=G(X,r)$  in the usual notation.
 \begin{enumerate}
 \item
 \label{equivlri}
 The following implications  are in force for all $a,b \in G$.
 \begin{equation}
 \label{eq1lri}
  ({}^{a^{-1}}b= b^a) \Longleftrightarrow ({}^a{(b^a)} = b);  \quad
  (b^{a^{-1}}= {}^ab)  \Longleftrightarrow (({}^ab)^a = b).
    \end{equation}
     \item  The following equalities hold for all $a,b \in G$.
 \label{equivlri2}
      \begin{equation}
  \label{eq2lri}
  \begin{array}{lllll}
\quad &a^{({}^{a^{-1}}b)} = {}^{b^{-1}}a,\quad& ({a^{-1}})^b= (a^{({}^{a^{-1}}b)})^{-1}, &\quad {}^b{(a^{-1})}=
({}^{(b^{a^{-1}})}a)^{-1}, & \\
&&&&\\
({}^ba)^{-1}&=(a^{-1})^{(b^{-1})},& ({{}^b{(a^{-1})}})^{-1}=a^{(b^{-1})},
&({}^{b^{-1}}a)^{-1}=(a^{-1})^{b},&({}^{b^{-1}}{a^{-1}})^{-1}=
a^b.
  \end{array}
    \end{equation}
 \end{enumerate}
\end{lemma}
The proof of this lemma is given in  \cite{GI15}, p.18.

The condition \textbf{lri } on $(X,r)$ can be
extended for the actions of the whole group $G$ upon
$X^{\star}$, as shows the following result.

 \begin{proposition}
 \label{VIPLemma}
 In notation as above, let  $(X,r)$ be a symmetric set with  \textbf{lri},
 and let $G=G(X,r)$ be its YB group.
 \begin{enumerate}
 \item
 \label{1}
 The following generalization of \textbf{lri}
 holds:
  \begin{equation}
  \label{lri*}
 {\rm\bf lri\star:} \quad  {}^a{(x^a)}= x =({}^ax)^a,\quad \forall\;x \in
 X^{\star},  \; a \in G.
 \end{equation}
 \begin{equation}
  \label{lri5}
 {}^{a^{-1}}x= x^a, \quad x^{a^{-1}} ={}^ax, \quad \forall \; x \in
 X^{\star},\; a \in G.
 \end{equation}
 \item
  \label{3}
 The following equalities hold:
 \begin{equation}
 \label{eqlri01}
 (x^a)^{-1}=(x^{-1})^a, \quad ({}^ax)^{-1}= {}^a{(x^{-1})}, \quad \forall \; x
 \in X^{\star},  \; a \in G.
 \end{equation}
 \item
  \label{4}
   The following generalization of the cyclic conditions   is in
     force.
 \begin{equation}
  \label{cl1}
  \begin{array}{lcl}
 {\rm\bf cl1\star:}\quad {}^{a^x}x= {}^ax,\quad&
 {\rm\bf cr1\star:}\quad
x^{{}^xa}= x^a,
\quad &\forall\;
 x \in X^{\star},  \; a \in G.
\end{array}
\end{equation}
More generally,
\begin{equation}
  \label{cl1_k}
  \begin{array}{lcl}
{}^{(a^{(x^k)})}x= {}^ax,  \quad &
 x^{({}^{(x^k)}a)}= x^a, \quad &\forall \;k \in  \mathbb{N}, \; x \in
 X^{\star},
 \; a \in G.
\end{array}
\end{equation}
\item
 \label{5}
The following equalities hold.
 \begin{equation}
  \label{lri66}
 {}^{a}{(x^k)}= ({}^{a}x)^k, \quad (x^k)^a = (x^a)^k   \quad \forall \; k \in \mathbb{N}, x \in
 X^{\star},\; a \in G.
 \end{equation}
 \end{enumerate}
  \end{proposition}
A detailed proof of the proposition can be found in \cite{GI15}.
\begin{corollary}
 \label{VIPcor}  Let $(X,r)$ be a symmetric set with \textbf{lri}.
 (1) The set $X^{\star}$ is $r_G$-invariant.
(2) Let $r^{\star}$ be the restriction of $r_G$ on $X^{\star}\times
  X^{\star}$.
Then $(X^{\star}, r^{\star})$ is a symmetric set which satisfies
\textbf{lri} and the cyclic conditions \textbf{cc}.
\end{corollary}

  \begin{remark}
 \label{VIPLemma2}
 In assumption and notation as in Proposition \ref{VIPLemma}, the
following equality is in force for all $a, x, y \in X^{\star}$:
 \begin{equation}
\label{lri6}
 x^a+({}^xy)^a = {}^{a^{-1}}(xy).
 \end{equation}
  \end{remark}
  \begin{proof}
 Suppose  $a, x, y \in  X^{\star},$ then the following
   equalities hold in $G$:
   \[
\begin{array}{llll}
 x^a+({}^xy)^a &=&{}^{a^{-1}}{(x)} +{}^{a^{-1}}{({}^xy)}\quad &\text{by \textbf{lri$\star$}}\\
                 &=&  {}^{a^{-1}}{(x+{}^xy)}=  {}^{a^{-1}}{(xy)} \quad &\text{by \textbf{Laut}}.
 \end{array}
 \]
\end{proof}
\subsection{Symmetric groups  with conditions
\textbf{lri} and \textbf{Raut}} In this subsection, as usual,  $(G, r)$ is a symmetric group and $(G, +, \cdot)$ is the associated left
brace.
\label{subsec_actions}
\begin{definition}
\label{DefRaut}
\begin{enumerate}
\item The symmetric group $G$ \emph{acts upon itself} \emph{as
automorphisms} if condition \textbf{MLaut} given below is in
force:
\begin{equation}
  \label{MLaut}
\textbf{MLaut}: \quad   \quad {}^a(bc) = ({}^ab)({}^ac),  \quad \forall \;
a,b,c \in G.\\
  \end{equation}
\item
The brace $(G, +,\cdot)$ satisfies condition \textbf{Raut} if
  \begin{equation}
  \label{Raut}
\textbf{Raut}: \quad  \quad (a+ b)^c = a^c + b^c,  \quad \forall \; a,b,c
\in
G.
  \end{equation}
  In this case the group $(G, \cdot)$ \emph{acts upon} $(G,+)$ \emph{from the right}
\emph{as
automorphisms}.
The map $\Rcal: (G, \cdot) \longrightarrow Aut (G, +)$
is a homomorphism of groups.
  \end{enumerate}
\end{definition}
We have shown in Proposition \ref{Thm_braces1} that under the assumption that $(G, +, \cdot)$ is a nonempty set
with operations "$\cdot$" and "$+$" such that $(G,\cdot)$ is a group,
$(G,+)$ is an abelian group, and the map $G\times G \longrightarrow G$ is
defined as $(a, b)\mapsto {}^ab: = ab-a,$
condition \textbf{Laut}, see (\ref{Laut}), is equivalent to the property
"$(G, +, \cdot)$ is a left brace".
In contrast, condition  \textbf{Raut} which seems to be "a right analogue"  of
 \textbf{Laut}, does not imply that the left brace $(G, +,
 \cdot)$  is also a right brace.  The difference  is due to the 'asymmetric' definitions of the left and the
  right actions, see (\ref{braceq4a}).

Suppose $(G,r)$ is a symmetric group and ${}^aa=a, \forall a \in
G$, i.e., $(G,r)$ is a square-free solution. Then (i) $(G,r)$
satisfies conditions \textbf{lri}, \textbf{cc} and \textbf{Raut};
(ii) The corresponding left brace $(G, +, \cdot)$ is a two-sided
brace, see \cite{CJO14}, Theorem 5.

 \begin{theorem}
 \label{prop_lriRaut}
 Let $(G, r)$ be a symmetric group and let
 $(G,+, \cdot)$ be the associated left brace.
 Then the following conditions hold.
\begin{enumerate}
\item \label{prop_lriRaut1} If $(G, r)$  satisfies condition
\textbf{lri} then condition \textbf{Raut} holds in $(G,+, \cdot)$.
Moreover the following equality is in force
\begin{equation}
  \label{VIP_Raut_Eqa}
 ({}^{({}^va)}u)({}^av)  = ({}^au) ({}^{(a^u)}v),
 \quad   \forall \; a,u,v \in G.
  \end{equation}
  \item
\label{propRautlri}
  Suppose that $G=G(X,r_0)$ is the symmetric group of a solution $(X,r_0)$. If $(X,r_0)$ satisfies \textbf{lri} and
  condition \textbf{Raut} holds in $(G,+, \cdot)$, then $(G, r)$  satisfies condition \textbf{lri}.
\end{enumerate}
 \end{theorem}
\begin{proof}
\textbf{(1)}. Suppose $a,u,v \in G$.  We
apply successively \textbf{lri} for the first, \textbf{Laut},
for the second, and \textbf{lri} for the third
equality below and yield:
\[(u+v)^a= {}^{a^{-1}}{(u+v)}= {}^{a^{-1}}u+{}^{a^{-1}}v= u^a + v^a. \]
Hence $(u+v)^a= u^a + v^a$, for all $a,u,v \in G$, that is \textbf{Raut} is in force.

We shall verify (\ref{VIP_Raut_Eqa}). Suppose $a,u,v \in G$. We
apply successively \textbf{MR2},  Lemma \ref{lemma_lriSymSet} and
\textbf{lri} and obtain:
\begin{equation}
 \label{eqbracelri2a}
(uv)^{a^{-1}}= (u^{({}^v{(a^{-1})})})(v^{a^{-1}}) =
(u^{({}^va)^{-1}})(v^{a^{-1}}) =   ({}^{({}^va)}u)({}^av).
  \end{equation}
Applying \textbf{lri}  and \textbf{ML2} we get $(uv)^{a^{-1}}=
{}^a{(uv)}= ({}^au) ({}^{(a^u)}v)$, which together with
(\ref{eqbracelri2a}) implies (\ref{VIP_Raut_Eqa}).

\textbf{(2)} Assume now that $(X_0,r_0)$ is a solution with \textbf{lri},
 and its associated left brace $(G,+, \cdot)$ satisfies \textbf{Raut}. We have to show that the symmetric group
  $(G, r)$ satisfies \textbf{lri}.
First we use induction on the length $|a|$ of $a$ to prove that
\begin{equation}
 \label{eqlri1}
 a^{(x^{-1})}={}^xa, \quad {}^{(x^{-1})}a = a^x, \quad \text{for all}\; a \in G, \; x\in X^{\star}.
  \end{equation}
 By hypothesis $(X,r)$
satisfies \textbf{lri}, hence, by Proposition \ref{VIPLemma},
condition \textbf{lri$\star$} is also in force. Therefore
(\ref{eqlri1}) is satisfied for all $a \in G$, $|a|=1,$
 which gives the base for induction.
Assume  (\ref{eqlri1}) is true for all $a \in G, |a|\leq n$. Let $x \in X^{\star}$, and let $b \in G, |b|= n+1$.
Then $b = at= a + {}^at,$ where $a \in G, |a|= n, t \in X^{\star}.$ One has
 \[\begin{array}{llll}
 b^{(x^{-1})}= (at)^{(x^{-1})}&=&(a + {}^at)^{(x^{-1})} = a^{(x^{-1})} + ({}^at)^{(x^{-1})} & \text{by \textbf{Raut}} \\
                         &=&{}^xa+{}^x{({}^at)} = {}^x{(a + {}^at)}= {}^x{(at)} = {}^xb \quad& \text{by IH, and  by
                         \textbf{Laut}},
\end{array}
\]
 which proves the first equality in (\ref{eqlri1}). The second equality in (\ref{eqlri1}) follows straightforwardly.
  Next, using induction on the length $|u|$ of $u \in G$ we verify
 \begin{equation}
 \label{eqlri2}
 {({}^ua)}^u = a, \; \; {}^u{(a^u)}=a, \quad \forall a, u \in G.
  \end{equation}
  The base for induction follows from  (\ref{eqlri1}). Assume that (\ref{eqlri2}) is true for all $a \in G$ and all $u \in G, |u| \leq
  n.$ Let $a\in G,$ and $v \in G$, with $|v|= n+1$. Then $v = ux$, where
$u \in G, \;|u| =n, \; x \in X^{\star}$.
  In the following computation we use \textbf{ML1}, and \textbf{MR1} for the second equality, and the inductive hypothesis
  for the third and fourth equalities:
  \[({}^va)^v = ({}^{ux}a)^{ux}= (({}^{u}{({}^xa)})^u)^x =({}^xa)^x =a. \]
  This proves the first identity in (\ref{eqlri2}), the second one is verified analogously.
  Hence $(G,r)$ satisfies \textbf{lri}.
\end{proof}

\begin{corollary}
 \label{cor_lriRaut}
  Suppose $(X,r)$ is a solution with \textbf{lri}, $(G, r_G)$ and $(G,+, \cdot)$  are its symmetric group and its left brace,
  respectively. Then $(G, r_G)$ satisfies \textbf{lri} if and only if $(G,+, \cdot)$
  satisfies \textbf{Raut}.
 \end{corollary}

\subsection{A symmetric group acting upon itself  as automorphisms is a
multipermutation solution of level $\leq 2$}
\begin{lemma}
\label{Lemma1_symgr_mpl1}
Let $(G,r)$ be a symmetric group.
The following four conditions are equivalent.
\begin{equation}
  \label{MLauteq2}
  \begin{array}{lllll}
  &\text{\textbf{(i)}} \quad &\Lcal_{({}^ba)}=\Lcal_a, \; \forall\; a,b \in G;\quad
  &\text{\textbf{(ii)}} \quad     &\Lcal_{(a^b)}=\Lcal_a,   \; \forall\; a,b\in G;\\
 &\text{\textbf{(iii)}} \quad     &\Rcal_{({}^ba)}=\Rcal_a,\;  \forall\; a,b  \in G;\quad
 &\text{\textbf{(iv)}}       &\Rcal_{(a^b)}=\Rcal_a,\;     \forall\; a,b \in G.
  \end{array}
  \end{equation}
Moreover, each of these four condition implies conditions \textbf{lri} and \textbf{cc} on $(G, r)$.
\end{lemma}
\begin{proof}
The implications  \textbf{(i)} $\Longleftrightarrow$ \textbf{(ii)} $\Longrightarrow$ (condition \textbf{lri}) are verified by
Lemma \ref{lemma_Aut_lri_sec4}
for arbitrary symmetric sets.
An argument analogous to the proof of Lemma \ref{lemma_Aut_lri_sec4} verifies \textbf{(iii)} $\Longleftrightarrow$  \textbf{(iv)}
$\Longrightarrow$ (\text{condition \textbf{lri}}).

\textbf{(i)} $\Longrightarrow$  \textbf{(iv)}.
 Suppose   \textbf{(i)} holds, then \textbf{(ii)} is also in force. Let $a,b,c \in G.$ We use successively condition \textbf{lri}, Lemma \ref{lemma_lriSymSet},
and \textbf{(ii)} to obtain
\[
c^{a^b} = {}^{(a^b)^{-1}}c    =  {}^{({a^{-1}})^b}c =  {}^{a^{-1}}c = c^a,
\]
which implies \textbf{(iv)}.
 One verifies \textbf{(iv)} $\Longrightarrow$ \textbf{(i)} analogously.
\end{proof}

\begin{proposition}
  \label{prop_ActsAsAut&lri}
  Let $(G,r)$ be a symmetric group, and let $G= (G, +, \cdot)$ be the
  associated
   left brace. Suppose $(G,r)$ is a nontrivial solution.
   The following three conditions are equivalent.
  \begin{enumerate}
  \item
  \label{mpl2a}
 $(G,r)$ is a multipermutation solution of YBE of level
 $2$.
  \item
  \label{prop_ActsAsAutLeft}
  $G$ acts  upon itself  from the left as automorphisms, that is,
  \textbf{MLaut}, (\ref{MLaut}) holds.
  \item
 \label{four_eq_cond}
 At least one  of the four identities (\ref{MLauteq2}) \textbf{(i)}, \textbf{(ii)}, \textbf{(iii)}, \textbf{(iv)}
  is in force.
     \end{enumerate}
Moreover, each of the above conditions implies the following.
  \begin{enumerate}
  \item[(i)]
  The symmetric set $(G,r)$ satisfies  \textbf{lri} and the cyclic
  conditions \textbf{cc}.
  \item[(ii)]
 The brace $G= (G, +, \cdot)$ satisfies condition  \textbf{Raut}.
      \end{enumerate}
\end{proposition}
\begin{proof}
(\ref{mpl2a}) $\Longleftrightarrow$ (\ref{four_eq_cond}). The equivalence  (\ref{mpl2a}) $\Longleftrightarrow$ (\ref{MLauteq2})
\textbf{(ii)} follows from
Lemma \ref{Lemma_symgr_mpl1}.3.
Hence, by Lemma \ref{Lemma1_symgr_mpl1}, condition (\ref{mpl2a}) is equivalent to each of the four conditions listed in
(\ref{MLauteq2}).
 (\ref{prop_ActsAsAutLeft})$\Longleftrightarrow$ (\ref{four_eq_cond}). We shall verify  (\ref{MLaut}) $\Longleftrightarrow$
 (\ref{MLauteq2}) \textbf{(i)}.
 The symmetric group $G$ satisfies \textbf{ML2}, so  ${}^a(bc) =
({}^ab)({}^{(a^b)}c)$,   $\forall a,b,c \in G.$ Therefore
the following implications are in force:
\[
 \begin{array}{lll}
 (\ref{MLaut}) &\Longleftrightarrow & [({}^ab)({}^ac) = {}^a{(bc)}= ({}^ab)({}^{(a^b)}c),
\; \forall a,b,c \in G] \\
&\Longleftrightarrow & [{}^{(a^b)}c = {}^ac, \;  \forall a,b,c\in G].
\end{array}
\]
This proves the equivalence (\ref{MLaut}) $\Longleftrightarrow$ (\ref{MLauteq2}) \textbf{(i)}. Our argument above shows that
conditions (\ref{prop_ActsAsAutLeft}) and (\ref{four_eq_cond}) are equivalent.
It follows then that  (1), (2) and (3) are equivalent, and by Lemma \ref{Lemma1_symgr_mpl1} each
of them implies \textbf{lri}, hence,
  by Corollary \ref{cor_lri_cc}  the cyclic conditions \textbf{cc} are
also in force.
By Theorem \ref{prop_lriRaut} (1)
\textbf{lri} implies \textbf{Raut}.
  \end{proof}
\begin{remark} Let $(G,r)$ be a symmetric group which acts upon itself as automorphisms, i.e. \textbf{MLaut} holds.  Then the semidirect
product
$G\ltimes G$ is also a symmetric group, with a braiding operator canonically extending $r$.
\end{remark}
This agrees with analogous statement of
Rump about braces, see \cite{Ru14}, Proposition 6.2.
\subsection{The symmetric groups  $(G, r_G)$ and $(\Gcal, r_{\Gcal})$  of a nontrivial permutation solution}
We consider the class of involutive permutation solutions of Lyubashenko (or shortly, the permutation solutions) :  $(X, r)= (X, r,
\sigma)$, where
$X$ is a set, not necessarily finite, $|X|\geq 2$, $\sigma \in
\Sym(X)$ is a permutation (bijective map), and
$r$ is the map $r: X\times X\longrightarrow X\times X$ defined as
$r(x,y):= (\sigma(y), \sigma^{-1}(x)), \forall \; x,y \in X.$
$(X,r)$ satisfies \textbf{lri}, and is a nontrivial solution \emph{iff} $\sigma \neq
id_X.$
Recall that a symmetric set $(X,r)$ has $\mpl X = 1$ \emph{iff}
$(X,r)$ is a permutation solution, see Remark  \ref{mpl1}. Moreover,  $\mpl G(X,r) = \mpl X = 1$ \emph{iff}
$\sigma = id_X$,
and $(X,r)$ is the trivial solution.
The proof of the following proposition is given in our preprint  \cite{GI15}, p. 40.
\begin{proposition}
\label{propexample}
Let $(X,r)= (X, r, \sigma)$ be the permutation solution defined by
$\sigma \in Sym (X)$, $\sigma   \neq id_X$.  Let $(G, r_G)$ and $(\Gcal, r_{\Gcal})$ be
the associated
symmetric groups. Let $x \in X$. Then the following conditions hold.
\begin{enumerate}
\item $\Gcal \cong \langle \sigma\rangle$. The retraction $\Ret(G, r_G)= ([G], r_{[G]})$  is a trivial solution of order $|\sigma|$.
\item
\label{propexample1}
More precisely, either
(i) the permutation $\sigma$ has finite order $m$,
then $\Ret(G, r_G)$ is the trivial solution on the set
$[G] =\{[1], \;[x], \;[x^2], \;[x^3], \;\cdots \; [x^{m-1}]\}$;
or
(ii) the permutation $\sigma$ has infinite order,
then
$\Ret(G, r_G)$ is the trivial solution on the countably infinite  set:
$[G] =\{\cdots, \; [x^{-2}], \;[x^{-1}], \;[1], \;[x], \;[x^2], \; \cdots
\}.$
\item
\label{propexample2}
$G$ acts upon itself as automorphisms, that is conditions \textbf{MLaut}  holds. Moreover,
$G$ satisfies conditions \textbf{lri}  and  \textbf{Raut}.
\item
\label{propexample3}
There is a strict inequality
$ 1=\mpl X< \mpl G =2.$
 \end{enumerate}
\end{proposition}
Clearly if $(X,r)$ is a permutation solution of finite order, then  the
permutation $\sigma \in Sym(X)$ has finite order $m$, so
$Ret(G,r_G)$ is a finite (trivial) symmetric set of order $m$.
We give two concrete examples in which $X$ is a countably infinite set.
\begin{example}
\label{exlri_infty}
Let $X = \{ x_i \mid i \in
\mathbb{Z}\}.$
We consider the following two solutions.
\begin{enumerate}
\item $(X,r, \sigma_1)= (X, r_1)$, where  $\sigma_1\in \Sym(X)$ is the shift
    $\sigma_1 (x_i) = x_{i+1}, i \in \mathbb{Z}.$
Denote by $(G^{(1)}, r^{(1)})$ the associated symmetric group.
The permutation $\sigma_1$ has infinite order,
$\Ret(G^{(1)}, r^{(1)})$ is the trivial solution on the countably infinite set
 $[G^{(1)}] =\{\cdots, [x^{-2}],  [x^{-1}], [1], [x], [x^2], [x^3], \cdots \}$,
 where $x\in X$ is an arbitrary  fixed element of $X$.

 \item $(X,r, \sigma_2) = (X, r_2)$, where  $\sigma_2\in \Sym(X)$ is the
      infinite product of transpositions $\sigma_2 = \prod_{k\in
      \mathbb{Z} } (x_{2k} \; x_{2k+1})$.
In other words, $\sigma(x_{2k}) = x_{2k+1},  \sigma(x_{2k+1}) = x_{2k},$ for
all $k\in \mathbb{Z},$
so $\sigma^2 = id_X$.
Denote by $(G^{(2)}, r^{(2)})$ the associated symmetric group.
Then $\Ret(G^{(2)}, r^{(2)})$ is a trivial solution of order
$2$, one has:
 $\; [G^{(2)}] =\{[1], [x] \}$, where $x\in X$ is an arbitrary  fixed element of
 $X$.
 \end{enumerate}
 In both cases there are equalities $\mpl (X, r_i) = 1, \mpl (G^{(i)}, r^{(i)})
 =2,$ $i = 1,2.$
\end{example}

\section{Square-free solutions $(X,r)$ with condition lri on a derived symmetric group, or on a derived permutation group}
\label{sect_TheoremGcal} In this section $(X,r)$ denotes a
square-free solution of arbitrary cardinality $|X|\geq 2$. We
shall indicate explicitly the cases when $X$ is a finite set. As
usual $(G, r_G)$ and $(G, +,\cdot)$ are the associated symmetric
group and the corresponding left brace. Let $X_i, i \in I,$ be the
set of all $G$-orbits in $X$
 (we consider the left action of $G$ upon $X$), $(X_i, r_i), i \in
 I,$ denotes the induced solution on $X_i$.
\subsection{The symmetric group  $G(X,r)$  satisfies \textbf{lri} if and only if $\mpl X \leq 2$}
The main result of the subsection is Theorem  \ref{Thm_main}.
 We first study how condition \textbf{lri} on the associated symmetric group $(G, r_G)$ affects the properties of $(X,r)$.
\begin{proposition}
 \label{prop_Thm_main}
 If $(X,r)$  is a nontrivial square-free solution, whose
 associated symmetric group $(G, r_G)$ satisfies condition
 \textbf{lri},
 then the following conditions hold.
 \begin{enumerate}
\item For each $G$-orbit $X_i, i \in I,$ the induced solution
$(X_i, r_i), $ is a trivial, or one element
 solution.
 \item  $\mpl (X,r) = \mpl (G, r_G) =2$.
  \end{enumerate}
\end{proposition}
 \begin{proof}
\textbf{(1)}. We shall verify the equalities
 \begin{equation}
 \label{orbitsare trivialsolutions}
 {}^{({}^bx)}x = x, \quad \quad   x^{({}^bx)}=x,\quad \text{for all}\; b \in G, x
 \in X.
  \end{equation}
  By Theorem
\ref{prop_lriRaut} (1)
 condition  \textbf{lri} implies
 \begin{equation}
  \label{VIP_Raut_Eqaab}
 ({}^{({}^bv)}u)({}^vb)  = ({}^vu) ({}^{(v^u)}b),
 \quad\forall \; b,u,v \in G.
  \end{equation}
 We set $u =x, v=x$ in this equality to obtain $({}^{({}^bx)}x)({}^xb)  = ({}^xx)
 ({}^{(x^x)}b)= x({}^xb)$ which,
 after cancelling ${}^xb$ from both sides,
 implies the first equality in  (\ref{orbitsare trivialsolutions}). We use
 \textbf{lri} to deduce the second equality in (\ref{orbitsare
 trivialsolutions}).  Clearly, (\ref{orbitsare
 trivialsolutions}) implies
 $r(x,y) = ({}^xy, x^y) = (y,x)$, for every $x,y \in X_i$, so
  $(X_i, r_i), i \in I,$ is a trivial, or one element solution.
\textbf{(2)}. We shall prove $\mpl X = 2.$ By hypothesis $(X,r)$
is a nontrivial (square-free) solution.
 Let $a, x, y \in X$. Then ${}^{({}^ya)}x, {}^ay, {}^ax,  {}^{(a^x)}y\in
 X$,
 so writing
(\ref{VIP_Raut_Eqaab}) in terms of $a, x, y$ we obtain the following equality
of words of length two in $S=S(X,r)$:
   \begin{equation}
  \label{VIP_Raut_Eqaa}
 ({}^{({}^ya)}x)({}^ay)  = ({}^ax) ({}^{(a^x)}y).
\end{equation}
Two cases are possible. \textbf{Case i.} The equality
(\ref{VIP_Raut_Eqaa}) holds in the free monoid $\langle X\rangle$.
In this case  $({}^{({}^ya)}x,{}^ay)  = ({}^ax, {}^{(a^x)}y)$ in
$X \times X$, hence  ${}^{({}^ya)}x = {}^ax$. \textbf{Case ii.}
$({}^{({}^ya)}x)({}^ay)  \neq  ({}^ax) ({}^{(a^x)}y)$
 in  $\langle X\rangle$, so (\ref{VIP_Raut_Eqaa})
 is a non-trivial relation in $S$.
Then
\begin{equation}
\label{eq9} ({}^{({}^ya)}x,({}^ay))= r({}^ax, {}^{(a^x)}y)=
({}^{{}^ax}{({}^{(a^x)}y)},({}^ax)^{({}^{(a^x)}y)})
 \end{equation}
 is
 an equality in  $X\times X$. Comparing the first components we get
\begin{equation}
\label{eq9a} {}^{({}^ya)}x= {}^{{}^ax}{({}^{(a^x)}y)},
 \end{equation}
 so $x$ and $y$
are in the same orbit $X_i$, for some $i$. Now (\ref{eq9a})
together with (\ref{VIP_Raut_Eqaa}) and the fact that $(X_i, r_i),
i \in I,$ is a trivial solution imply
\[{}^ax= {}^{({}^ya)}x  = {}^{{}^ax}{({}^{(a^x)}y)}= {}^{(a^x)}y = {}^ay.
\]
Thus ${}^ax={}^ay$, which by
the non-degeneracy implies $x = y$, and therefore (ii) is
impossible. We have shown that  ${}^{({}^ya)}x = {}^ax$,  for  all
$a, x, y \in X$,
 or equivalently, $\Lcal_{({}^ya)}=  \Lcal_a, \quad \forall \; a, y  \in X$. This implies that $\mpl X = 2$. 
 By hypothesis $(X,r)$ is a
 square-free solution, so Theorem  \ref{Th_mplG_mplX} implies $\mpl G =\mpl X=2$.
 \end{proof}
\begin{theorem}
 \label{Thm_main}
 Let $(X,r)$ be a nontrivial square-free solution of arbitrary
 cardinality.
  \begin{enumerate}
 \item
 The following conditions are equivalent
 \begin{enumerate}
  \item
    \label{Thm_main1}
 $(X,r)$ is a multipermutation solution of level $2$.
 \item
    \label{Thm_main2}
 $(G,r_G)$ is a multipermutation solution of level $2$.
 \item
    \label{Thm_main4}
  $G$ acts  upon itself  as automorphisms that is
$\Lcal_{({}^ba)}= \Lcal_a, \quad  \forall\; a,b \in G.$
   \item
 \label{Thm_main5}
 $(G,r_G)$ satisfies condition \textbf{lri}.
 \item
   \label{Thm_main7}
   The left brace $(G, +, \cdot)$ satisfies condition \textbf{Raut}.
  \end{enumerate}
  \item
  More generally, $\mpl (X,r)= m\geq 2$ \emph{iff} the derived group $(G_{m-2}, r_{G_{m-2}})$ is nonabelian and satisfies
  \textbf{lri}.
    \end{enumerate}
\end{theorem}
\begin{proof}
\textbf{(1).} The implication (\ref{Thm_main1}) $\Longleftrightarrow$ (\ref{Thm_main2}) follows from Theorem \ref{Th_mplG_mplX} part (3).
Note that each of the solutions $(X,r)$ and $(G, r_G)$ satisfies condition (*), see Definition \ref{defcondstar}. Then Lemma
\ref{Lemma_symgr_mpl1}, part  (3) verifies
the implications
$(\ref{Thm_main2}) \Longleftrightarrow     (\ref{Thm_main4})
\Longrightarrow (\ref{Thm_main5})$.
Proposition  \ref{prop_Thm_main} verifies
(\ref{Thm_main5})   $\Longrightarrow$ (\ref{Thm_main1}).
Since every square-free solution $(X,r)$ satisfies \textbf{lri}, Corollary \ref{cor_lriRaut}
gives the implication (\ref{Thm_main5}) $\Longleftrightarrow$  (\ref{Thm_main7}).
\textbf{(2).} By definition $(G_{m-2}, r_{G_{m-2}})= G(\Ret^{m-2}(X, r))$, moreover there are equivalences:
$$
[\mpl (X, r) = m \geq 2]\Longleftrightarrow [\mpl\Ret{}^{ m-2}{(X, r)} = 2]\Longleftrightarrow [\mpl (G_{m-2}, r_{G_{m-2}}) = 2].
$$
Part (1) implies $\mpl\Ret{}^{ m-2}{(X, r)} = 2$ \emph{iff} $(G_{m-2}, r_{G_{m-2}})$ satisfies \textbf{lri} and is nonabelian, which verifies part (2).
\end{proof}

\subsection{A finite square-free symmetric set $(X,r)$ whose permutation group  $\Gcal(X,r)$ satisfies \textbf{lri} is a
multipermutation
solution}
\label{sect_TheoremGcal_a}

The conventions given in the beginning of this section are in
force. So $(X,r)$ denotes a nontrivial square-free solution of
arbitrary
 cardinality, $|X|\geq 2$,
 $(G, r_G), $    $(\Gcal, r_{\Gcal})$ are the associated
symmetric groups,  $S= S(X,r)$ is the associated monoid,
$(G_j,r_{G_j})$, and $(\Gcal_j,r_{\Gcal})$, $j \geq 0$ are the
corresponding derived groups. In particular, $G_1 = G([X],
r_{[X]})$ is the symmetric group of the retraction
$\Ret(X,r)=([X],r_{[X]})$. We denote by $X_i, i \in I,$ the set of
all $G$-orbits in $X$ (possibly infinite) and by
 $(X_i, r_i), i\in I$, the induced solutions. Clearly, a
 subset $Y\subset X$ is a $G$-orbit  \emph{iff} it is a $\Gcal$-orbit.
The main result of the subsection is Theorem  \ref{Thm_gcal(X,r)}.

Consider the group epimorphism
$\Lcal: G \longrightarrow \Gcal$, extending the assignment  $x\mapsto \Lcal_x, x  \in X$. For $a \in G$ we shall often write
 $\overline{a}:= \Lcal_a\in \Gcal$,  so every element $u \in \Gcal$
 can be written as $u = \overline{a},$ for some $a \in G.$
 If $a, b \in G$  we write $a\sim b$ to express  that $\Lcal_a =\Lcal_b$ on $X$ (by Remark \ref{factGamma=K} this  implies equality of
 the left actions $\Lcal_a =\Lcal_b$ of $G$ upon itself).
 We shall involve again computation with
long actions,
or, as we call these,  "towers of
actions".

  \begin{lemma}
  \label{Prop_gcal2}
Notation as above.  Let $(X,r)$ be a nontrivial solution, $|X|\geq 2$, such that: (i) for each $i \in
I$, the  $G$-orbit
  $(X_i, r_i)$ is a trivial solution,
 or one element solution; and (ii)
the symmetric group $(\Gcal, r_{\Gcal})$ satisfies
 condition
 \textbf{lri}.
     Then the  equality:
  \begin{equation}
\label{mplmeq1} ((\cdots((b\la y_{n})\la y_{n-1})\la \cdots \la
y_2)\la y_1)\la a = ((\cdots( y_{n}\la y_{n-1})\la \cdots \la y_2
)\la y_1)\la a
\end{equation}
holds for any integer $n\geq 1$, any choice of $y_1, \cdots, y_n\in X$,
     and any pair $a, b$, such that $a\in X_i$,  $ b \in X_i \bigcup
X_i^{-1},$ $ i \in I$.
 \end{lemma}
\begin{proof}
 By hypothesis  $(\Gcal, r_{\Gcal})$  satisfies condition \textbf{lri}, hence by Theorem  \ref{prop_lriRaut},
the identity  (\ref{VIP_Raut_Eqa}) hods in $\Gcal$. We write it
using the alternative notation,  " $\alpha\la x =
{}^{\alpha}{x}$":
\[
 ((v\la u)\la b)(u\la v)  =
 (u\la b).((u^b)\la v),
 \quad \forall \; u, v, b \in
 \Gcal.
  \]
 Equivalently, one has
\begin{equation}
  \label{VIP_Raut_Eqb07}
 \overline{((v\la u)\la b)(u\la v)}  =
 \overline{(u\la b).((u^b)\la v)},
 \quad \forall \; u, v, b \in
 G.
  \end{equation}
We shall use induction on $n$ to prove that
     (\ref{mplmeq1}) holds for any $n$, any choice of $y_1, \cdots, y_n\in X$,
     and any $a$ and $b$ such that $a\in X_i$,  $ b \in X_i \bigcup
X_i^{-1},$ $ i \in I$. The base for the induction follows
straightforwardly from Proposition \ref{Lemma_TrivialOrbits},
(\ref{Lemma_TrivialOrbitseq1}). Due to condition \textbf{lri} on
$(X,r)$, the equalities (\ref{Lemma_TrivialOrbitseq1}) are
satisfied for all $y \in X$, and all  $a,b$, such that $a\in X_i,
b \in X_i \bigcup X_i^{-1}$, $i \in I$.
 Assume (\ref{mplmeq1}) is in force for all
     $n, 1 \leq n \leq k$. Chose
     $y_1, \cdots , y_k, y_{k+1} \in X$, $a$ and $b$ as above.
  We write (\ref{VIP_Raut_Eqb07}) in terms of $ y_k, y_{k+1}, b$
   to obtain
   \begin{equation}
  \label{VIP_Raut_Eqc071}
  (\overline{(y_k \la y_{k+1})\la b}).(\overline{y_{k+1}\la y_k})=
  (\overline{y_{k+1}\la b}). (\overline{(y_{k+1}^b)\la y_k}).
  \end{equation}
This implies an equivalence of elements of $G$ with respect to
their left action upon $X$:
\begin{equation}
  \label{VIP_Raut_Eqc07}
   ((y_k \la y_{k+1})\la b). (y_{k+1}\la y_k)\sim
  (y_{k+1}\la b). ((y_{k+1}^b)\la y_k).
  \end{equation}
 Now we set
  \begin{equation}
  \label{VIP_Raut_Eq167}
\begin{array}{lll}
b^{\prime}&= (y_k \la y_{k+1})\la b, \quad &{y_k}^{\prime} =  y_{k+1}\la y_k,\\
b^{\prime\prime}&= y_{k+1}\la b, \quad &{y_k}^{\prime\prime}= (y_{k+1}^b)\la y_k.
\end{array}
 \end{equation}
 Substitute these in  (\ref{VIP_Raut_Eqc07}) and obtain
 $  b^{\prime}.{y_k}^{\prime}  \sim b^{\prime\prime}.{y_k}^{\prime\prime}$.
 Recall that $G$ acts on $X^{\ast}$, see Section  \ref{section_lri}.  By Proposition
 \ref{VIPLemma} one has $({}^cb)^{-1} = {}^c{(b^{-1})},
\forall c
 \in G, \; b \in X^{\ast}$, hence $b^{-1} \in X_i$ implies  $({}^c{b})^{-1}\in X_i$, $\forall c
 \in G$.  Therefore each of the pairs
 $a, b^{\prime}$ and $a, b^{\prime\prime}$
satisfies the inductive assumption.

Without loss of generality we may assume that $k \geq 2.$  (In the case when
$k =1$, replace $y_{k-1}$ by $a$ in the argument).
  Then  acting on the left upon $y_{k-1}$ we yield  
$(b^{\prime}.{y_k}^{\prime})\la y_{k-1}=
  (b^{\prime\prime}.{y_k}^{\prime\prime})\la y_{k-1}$,
  which by  condition \textbf{ML1} gives
\[
  \omega =b^{\prime}\la ({y_k}^{\prime}\la y_{k-1})=
  b^{\prime\prime}\la ({y_k}^{\prime\prime}\la y_{k-1}).
\]
   Using these two presentations of the element $\omega$ we have the
   following
   equalities of "towers of actions"  denoted by $\tau$ and $\tau_1$
  \begin{equation}
  \label{VIP_Raut_Eq15b}
  \begin{array}{lll}
  \tau &=&(\cdots((b^{\prime}\la ({y_k}^{\prime}\la y_{k-1}))\la y_{k-2})
  \la \cdots \la y_1)\la a\\
  &=&
  ((\cdots(b^{\prime\prime}\la ({y_k}^{\prime\prime}\la y_{k-1}))\la y_{k-2})
  \la \cdots \la
  y_1)\la a= \tau_1.
  \end{array}
  \end{equation}
Next we apply the inductive assumption to each of the elements $\tau$ and
$\tau_1$ occurring in
   (\ref{VIP_Raut_Eq15b}). For $\tau$ we have
   \begin{equation}
  \label{VIP_Raut_Eq17aa}
  \begin{array}{llll}
  \tau &=&(\cdots((b^{\prime}\la ({y_k}^{\prime}\la y_{k-1}))\la y_{k-2})
  \la \cdots \la y_1)\la a \quad &\\
 &=&(\cdots(( {y_k}^{\prime}\la y_{k-1})\la y_{k-2})\la \cdots \la y_1)\la
 a\quad &\text{by IH}\\
   &=&(\cdots( (y_{k+1}\la y_k)\la y_{k-1})\la \cdots \la y_1)\la a&\text{by
   (\ref{VIP_Raut_Eq167})},
  \end{array}
  \end{equation}
  where IH is the inductive hypothesis. Similarly, for   $\tau_1$ we have
    \begin{equation}
  \label{VIP_Raut_Eq18aa}
  \begin{array}{llll}
  \tau_1&=& ((\cdots (b^{\prime\prime}\la ({y_k}^{\prime\prime}\la
  y_{k-1}))\la y_{k-2})
  \cdots \la y_1)\la a&\\
   &=&(\cdots(({y_k}^{\prime\prime}\la y_{k-1})\la y_{k-2}) \cdots \la y_1)\la
   a \quad   & \text{by IH}\\
   &=& (\cdots ((((y_{k+1}^b)\la y_k) \la y_{k-1})\la y_{k-2})\la \cdots \la
   y_1)\la a \; & \text{by (\ref{VIP_Raut_Eq167}) }\\
   &=&(\cdots ((((b^{-1} \la y_{k+1})\la y_k) \la y_{k-1})\la y_{k-2})\la
   \cdots \la
   y_1)\la a \quad & \text{by
   \textbf{lri} }\\
   &=&(\cdots ((((c \;\la y_{k+1})\la y_k) \la y_{k-1})\la y_{k-2})\la \cdots
   \la
   y_1)\la a, \;  &\text{where $c=b^{-1}$}.
  \end{array}
  \end{equation}
The presentations of $\tau_1$ and $\tau$  given in (\ref{VIP_Raut_Eq18aa}) and
(\ref{VIP_Raut_Eq17aa}),  together with $\tau_1=\tau$ imply that the equality
\[
(\cdots (((c \; \la y_{k+1})\la y_k) \la y_{k-1})\la \cdots \la
   y_1)\la a=(\cdots( (y_{k+1}\la y_k)\la y_{k-1})\la \cdots \la y_1)\la a
\]
holds, whenever $a \in X_i$ and $c \in X_i \bigcup X_i^{-1}$, $i \in I$.
\end{proof}

  \begin{proposition}
  \label{Prop_gcal}
 Let $(X,r)$ be a nontrivial solution, whose symmetric group $(\Gcal, r_{\Gcal})$ satisfies
 condition
 \textbf{lri}. Suppose $X$ has a finite number of $\Gcal$-orbits $X_i,   1 \leq i \leq m,$ and that
  each $(X_i, r_i), 1 \leq i \leq m,$ is a trivial solution,
 or one element solution.
 Then  $(X,r)$ is a square-free multipermutation solution, with $|X|\geq 3$, $2
\leq \mpl (X,r) \leq m$  and presents as a strong twisted union $X =
X_1\stu X_2 \cdots \stu X_m$ of solutions with $0 \leq \mpl(X_i)
\leq 1$.
 \end{proposition}
 \begin{proof}
It is clear that $(X,r)$ is a square-free solution, since each orbit is a trivial solution and therefore a square-free symmetric set.
But
a square-free solution  of order $|X|=2$ is trivial, hence the hypothesis of the proposition implies
$|X|\geq 3$. It follows from Proposition \ref{Lemma_TrivialOrbits}
 that $X$
is a strong twisted union $X = X_1\stu X_2 \cdots \stu X_m$  of
solutions with $0 \leq \mpl(X_i) \leq 1.$ We shall prove that
$(X,r)$ is a multipermutation solution, with $2 \leq \mpl (X,r)
\leq m$. The square-free solution $(X,r)$ satisfies the hypothesis
of Proposition \ref{mplmtheorem1}, (2), therefore
 it will be enough to show that 
\begin{equation}
\label{mplmeqa}
\begin{array}{lll}
\omega &=& (\cdots((y_{m+1}\la y_{m})\la y_{m-1})\la
\cdots \la y_2)\la y_1\\
 &=& (\cdots( y_{m}\la y_{m-1})\la \cdots \la y_2
)\la y_1= \omega^{\prime},\\&&\;\;\text{for all} \;\; y_1, \cdots ,y_{m+1}\in X.
\end{array}
\end{equation}
Let  $y_1, \cdots ,y_{m+1}\in X.$ There are  exactly $m$ orbits, hence there will be some integers $\lambda,\mu$, $1
\leq \lambda< \lambda+\mu\leq m+1$, such that $y_{\lambda},
y_{\lambda+\mu}$ are in the same orbit, say $X_i$. Two cases are
possible.
\textbf{Case 1.} $\mu = 1.$ In this case $\lambda+\mu= \lambda+1$, so
$(\cdots(y_{m+1}\la y_{m})\la \cdots ) \la y_{\lambda+1} = u \in X_i$, and
since $X_i$ is a trivial solution, one has  $(\cdots (y_{m+1}\la
y_{m})\la \cdots \la y_{\lambda+1})\la y_{\lambda} =
{}^u{y_{\lambda}}=y_{\lambda}$.  Thus
\[
\omega=(\cdots((y_{m+1}\la y_{m})\la y_{m-1})\la \cdots \la y_2)\la
y_1 =(\cdots(y_{\lambda}\la y_{\lambda-1})\la \cdots \la y_2)\la y_1.
\]
Similarly,
$\omega^{\prime}=(\cdots( y_{m}\la y_{m-1})\la \cdots \la y_2)\la
y_1 =(\cdots(y_{\lambda}\la y_{\lambda-1})\la \cdots \la y_2)\la
y_1.$
This gives the desired equality $\omega=\omega^{\prime}$, which  proves
(\ref{mplmeqa}).
\textbf{Case 2.} $\mu > 1$.
We shall assume $\lambda+\mu < m+1$ (The proof in the case
$\lambda+\mu = m+1$ is analogous). We set
\begin{equation}
\label{eq_settings}
\begin{array}{lll}
z&=& y_{\lambda},\quad u=y_{\lambda+\mu},\quad
{u}^{\prime}=(\cdots(y_{m+1}\la y_{m})\la\cdots)\la y_{\lambda+\mu},\\
\omega_{\lambda}&=& ((\cdots(u^{\prime} \la y_{\lambda+\mu-1})\la
y_{\lambda+\mu-2})\cdots \la
y_{\lambda+1})\la z.
\end{array}
\end{equation}
In this notation $\omega$ presents as
\begin{equation}
\label{eq_settings1}
\omega= (\cdots(\omega_{\lambda}\la y_{\lambda - 1})\la\cdots)\la y_1.
\end{equation}
By assumption   $u= y_{\lambda+\mu}\in X_i$, hence ${u}^{\prime}$
is also in $X_i$. The tower $\omega_{\lambda}$ contains the
elements $u^{\prime}$ and $z$ which belong to the same orbit
$X_i$, therefore by Lemma \ref{Prop_gcal2} we can "cut"
$u^{\prime}$ to yield
\[
\begin{array}{rl}
\omega_{\lambda}= &((\cdots(u^{\prime} \la y_{\lambda+\mu-1})\la
y_{\lambda+\mu-2})\cdots \la
y_{\lambda+1})\la z=(\cdots  (y_{\lambda+\mu-1}\la y_{\lambda+\mu-2})\cdots \la
y_{\lambda+1})\la z\\
=&(\cdots  (y_{\lambda+\mu-1}\la y_{\lambda+\mu-2})\cdots \la
y_{\lambda+1})\la y_{\lambda}.
\end{array}
\]
We use this equality and (\ref{eq_settings1}) to yield a truncation of
$\omega$:
\[
\begin{array}{rl}
\omega=&(\cdots((y_{m+1}\la y_{m})\la y_{m-1})\la \cdots \la y_2)\la
y_1
=(\cdots(\omega_{\lambda}\la y_{\lambda - 1})\la\cdots)\la y_1\\
=& ((\cdots((\cdots  (y_{\lambda+\mu-1}\la y_{\lambda+\mu-2})\cdots \la
y_{\lambda+1})\la y_{\lambda})\la \cdots)\la y_2)\la
y_1.
\end{array}
\]
Analogous argument verifies that the tower  $\omega^{\prime}$ of length $m$
can also be truncated:
\[\begin{array}{rl}
\omega^{\prime}=&(\cdots( y_{m}\la y_{m-1})\la \cdots \la y_2)\la y_1 \\
=&((\cdots((\cdots  (y_{\lambda+\mu-1}\la y_{\lambda+\mu-2})\cdots \la
y_{\lambda+1})\la y_{\lambda})\la \cdots)\la y_2)\la
y_1 = \omega.
\end{array}
\]
Therefore (\ref{mplmeqa}) is in force, and Proposition
\ref{mplmtheorem1} implies $\mpl G \leq m$.
  \end{proof}

\begin{theorem}
 \label{Thm_gcal(X,r)}
 Let $(X,r)$ be a square-free solution, $|X|\geq 2$,  with associated permutation group $\Gcal=\Gcal(X,r)$.
 Suppose that the symmetric
 group $(\Gcal, r_{\Gcal})$ satisfies \textbf{lri}. Then the following is in force.
 \begin{enumerate}
 \item
 \label{Thm_gcal(X,r)1} At least one of the following two conditions holds:
(a) The solution $(X,r)$ is retractable; (b) Each induced solution
  $(X_i, r_i)$, $i \in I,$ is a trivial solution,
 or one element solution and
  $X$ is a strong twisted union $X = \natural_{i \in I}
 X_i$.
     \item
 \label{Thm_gcal(X,r)2}
If $X$ is
   a finite set,
 then $(X,r)$  is a multipermutation solution.
\end{enumerate}
 \end{theorem}
 \begin{proof}  
 \textbf{(\ref{Thm_gcal(X,r)1}).}
 By hypothesis $(\Gcal, r_{\Gcal})$ satisfies condition \textbf{lri}, hence
 the  equality   (\ref{VIP_Raut_Eqb07}) holds for all $u, v, b \in
 G.$  We choose $x \in X, v \in G$, set  $b = x, u= x$, and replace these
 in (\ref{VIP_Raut_Eqb07}).  Then the following equalities hold
 for all $x \in X, \; v \in G$:
  \[
  (\overline{(v\la x)\la x}).(\overline{x\la v}) = (\overline{(v\la x)\la
 x)(x\la v}) = \overline{(x\la x).((x^x)\la v)} = \overline{x.(x\la v)}=\overline{x}.\overline{(x\la v)}.
 \]
Now cancel $\overline{x\la v}$ from the first and from the last monomial
  to obtain $\overline{((v\la x)\la x)} = \overline{x}$, for all $x \in X$
  and all $v \in G$. This implies the equality of actions
  \[
  \Lcal_{((v\la x)\la x)} = \Lcal_x, \quad \forall  x \in X, v \in G.
  \]
Two cases are possible.
Case \textbf{ (i).}
There exist  $x\in X$ and $v\in G$ such that
  \[(v\la x)\la x \neq x, \quad\text{but}\;\;  \Lcal_{((v\la x)\la x)} =
  \Lcal_x.
  \]
  This implies that $(X,r)$ is retractable.
Case \textbf{(ii).}  $(v\la x)\la x = x$
  for all  $x \in X$, and all $v \in G$. This is equivalent to
 \[
  {}^yx = x,
 \quad\forall \; x, y \in X_i,\; i \in I.
 \]
  It follows that the induced solution  $(X_i, r_i), i \in I,$ is a trivial
 solution, or one element solution.
 The hypothesis of Proposition
  \ref{Prop_gcal} is satisfied and therefore
 $X$ is a strong
twisted union
$
X= \stu_{i \in I}  X_i
$,  in the sense of \cite{GIC}.
Part (\ref{Thm_gcal(X,r)1}) has been proved.
\textbf{(2).} Assume that $(X,r)$ is a finite square-free solution.
Then it has a finite set of $\Gcal$-orbits, say
$X_1, \cdots X_t$. Now consider each of the two cases \textbf{(i)} and \textbf{(ii)} occurring in the argument of part
(\ref{Thm_gcal(X,r)1}).
In the case when \textbf{(ii)} holds, each induced solution   $(X_i, r_i), 1 \leq i \leq t,$ is a trivial
 solution or one element solution. So Proposition  \ref{Prop_gcal}
implies that $(X,r)$ is a multipermutation solution with $\mpl
X \leq t$.
Suppose case \textbf{(i)} is in force, then $(X,r)$ is retractable.
Its retraction $Ret(X,r) = ([X],r_{[X]})$ is  a
 square-free symmetric set of order $< |X|$.  The
 corresponding
 symmetric  group $\Gcal_1 =\Gcal([X],r_{[X]})$ is a (braiding-preserving)
 homomorphic image
 of $(\Gcal, r_{\Gcal})$,
 so $(\Gcal_1, r_{\Gcal_1})$ inherits
 condition \textbf{lri} from $\Gcal$.
It follows by induction on the order $|X|$ of $X$ that $(X,r)$ is a
multipermutation solution with $\mpl X \leq |X| -1$.
\end{proof}
Suppose $(X,r)$ a (nontrivial) square-free solution, $(G, r_G)$, $(\Gcal, r_{\Gcal})$ in usual notation. We have shown that 
the conditions "$\mpl X = 2$" and  "$(G, r_G)$  satisfies \textbf{lri}" are equivalent (Theorem \ref{Thm_main}). Moreover,
condition \textbf{lri} on $(\Gcal, r_{\Gcal})$ is sufficient for "$\mpl X < \infty$" but it does not give
information about the precise value of $\mpl X$.
We shall show that "$(\Gcal, r_{\Gcal})$ satisfies \textbf{lri}" is a necessary condition for  "$\mpl X = 3$". 
\begin{lemma}
\label{Lemma_Gcal}
Assumption and notation as above.
Suppose $(X,r)$  is a square-free solution with $\mpl X = 3$.
Then $(\Gcal, r_{\Gcal})$ satisfies \textbf{lri} and therefore
\textbf{Raut}.
\end{lemma}
\begin{proof}
It follows from Theorem \ref{Thm_main}  that $\mpl X = 3$ if and
only if the derived symmetric group $(G_1, r_{G_1})$  satisfies
condition \textbf{lri} and is a nontrivial solution. The canonical
map  $f: (G_1,r_{G_1}) \longrightarrow (\Gcal, r_{\Gcal})$,  $[x]
\mapsto \Lcal_{x}$, see Lemma \ref{retlemma1},
 is an
 epimorphism of symmetric groups, so $(\Gcal, r_{\Gcal})$ also satisfies conditions  \textbf{lri} and
 \textbf{Raut}.
 \end{proof}
 \begin{corollary}
\label{Cor_main} Suppose $(X,r)$  is a finite nontrivial
square-free solution. The following conditions are equivalent.
\begin{enumerate}
 \item $2 \leq \mpl X = m < \infty$. \item There exists an integer $j \geq 0$, such that the derived permutation group
     $(\Gcal_j,r_{\Gcal_j})$
     satisfies \textbf{lri}, and $\Gcal_j \neq \{1\}$.
\end{enumerate}
  \end{corollary}
  \begin{proof}
If $\mpl X =2$, then $\mpl (\Gcal,r_{\Gcal}) =1$, so
$(\Gcal,r_{\Gcal})$ satisfies \textbf{lri}. Assume now $m=\mpl X
\geq 3$. Then $\mpl\Ret ^{m-3}(X,r) = 3$, and by Lemma
\ref{Lemma_Gcal} the derived permutation group $\Gcal_{m-3}$
satisfies \textbf{lri}. Moreover, $\mpl \Gcal_{m-3}= \mpl\Ret
^{m-3}(X,r) -1 =2$, so $\Gcal_{m-3}\neq \{1\}$. Conversely, if
$(\Gcal_j,r_{\Gcal_j})$ satisfies \textbf{lri} for some $j \geq 0$
and $\Gcal_j \neq \{1\}$ then the corresponding retraction $\Ret
^j(X,r)$ has finite multipermutation level, $2 \leq \mpl \Ret
^j(X,r) = s$, hence $2 \leq \mpl (X,r) = j + s<\infty$.
\end{proof}
\section{Epilogue}
\label{epilogue}
Suppose $(X,r)$ is an arbitrary  non-degenerate symmetric set, $|X|\geq 2$, no additional restrictions are imposed on $(X,r)$, unless
specified explicitly. Let
$(G, r_G)$, $(\Gcal, r_{\Gcal})$  be the associated symmetric groups, and let
$(G_j, r_{G_j})= G(\Ret^j(X, r))$, $j \geq 0,$  $\Gcal_j = \Gcal(\Ret^j(X, r)), j \geq 0,$ be the derived symmetric groups, and the
derived permutation groups related to $(X,r)$. $m\geq 1$ denotes an integer.
We have verified the following.
\begin{enumerate}
\item
\label{i1}
[$(G, +, \cdot)$ is a two sided brace]  $ \Longleftrightarrow$   [$(X,r)$ is a trivial solution].
\item
\label{i4}
   [$\mpl G  < \infty]\; \Longleftrightarrow  \;  [\mpl X < \infty] \; \Longleftrightarrow  \;  [\mpl \Gcal < \infty].$
In this case
\begin{equation}
\label{ineq1} 0 \leq \mpl \Gcal = m-1 \leq \mpl X \leq \mpl G = m
< \infty.\end{equation}
   \item
   \label{i5}
If  $(X,r)$ is square-free, then the following are equivalent: (i)
$\mpl X = m$; (ii) $\mpl G = m$; (iii) $(G_{m-2}, r_{G_{m-2}})$
satisfies \textbf{lri} and is not abelian; (iv) $(G_{m-2}, +,
\cdot)$ satisfies \textbf{Raut} and is not abelian; (v) the
derived chain of ideals for $G$ has the shape $\{1\}=
K_0\subsetneqq K_1 \subsetneqq K_2 \subsetneqq \cdots \subsetneqq
K_{m-1}\subsetneqq K_{m} = G$; (vi) The brace $G$ is right
nilpotent of nilpotency class $m+1$, i.e. $G^{(m+1)} = 0$, and
$G^{(m)}\neq 0.$

\item
\label{i6}
If $X$ is  finite then
\[[2 \leq \mpl X< \infty] \Longleftrightarrow [\exists j \geq 0, \text{such that}\; (\Gcal_j, +, \cdot)\; \text{is a two-sided
brace},\;
\Gcal_j\neq \{1\}].
\]

\item
\label{i6}
If $(X, r)$ is  finite and square-free then
\[[2 \leq \mpl X<\infty]\Longleftrightarrow  [\exists j \geq 0, \text{s.t.}\;
(\Gcal_j, r_{\Gcal_j})\; \text{satisfies \textbf{lri},}\; \Gcal_j \neq \{1\}].
\]
\end{enumerate}
\subsection*{Acknowledgements}  This work was completed in 2014-2016 while I
was visiting IH\'{E}S, Max-Planck Institute for Mathematics, Bonn,
and  ICTP, Trieste. I thank Shahn Majid  for our inspiring
collaboration through the years and for introducing me to the
fascinating theory of matched pairs of groups. It is my pleasant
duty to thank Maxim Kontsevich and Don Zagier for several
stimulating discussions. I give my cordial thanks to Agata
Smoktunowicz and Ferran Ced\'{o} for our fruitful collaboration.


\begin{thebibliography}{99}
\bibitem[AGV]{AGV}
Angiono, I, Galindo, C, Vendramin, L, {\em Hopf braces and Yang-Baxter operators}, arXiv preprint   arXiv:1604.02098 [math.QA], 2016.

\bibitem[BC]{BC}
Bachiller  D., Ced\'{o} F., \emph{A family of solutions of the
Yang--Baxter equation}
J. Algebra,  \textbf{412} (2014),  218--229.


\bibitem[BCJ]{BCJ15}
Bachiller D.,  Ced\'{o}  F., Jespers, E.,
\emph{ Solutions of the Yang-Baxter equation associated with a left brace},
 J. Algebra, \textbf{463} (2016), 80--102.

 \bibitem[BCJO]{BCJO}
Bachiller D.,  Ced\'{o}  F., Jespers, E., Okni\'{n}�ski, J. \emph{A family of irretractable square-free
solutions of the Yang-Baxter equation}
arXiv:1511.07769, [math.GR]  Oct. 2016.

\bibitem[CES]{Carter}
Carter, J.S.,  Elhamdadi, M.,  Saito, M., {\em{Homology theory for the set-theoretic Yang-Baxter
equation and knot invariants from generalizations of quandles}},  Fund. Math. \textbf{184}  (2004), 31--54.

\bibitem[CGIS16]{CGIS16} Ced\'{o}, F.,  Gateva--Ivanova, T, Smoktunowicz,  A, {\em{On the Yang-Baxter equation and left nilpotent left
    braces}},
Journal of Pure and Applied Algebra,
\textbf{ 221}
(2017), 751--756.
   doi:10.1016/j.jpaa.2016.07.014

\bibitem[CJO06]{CJO06}
Ced\'{o}, F., Jespers, E. , Okni\'{n}ski J.,
  {\em{The Gelfand-Kirillov dimension of quadratic algebras satisfying the cyclic condition }}, Proc. Amer. Math. Soc. \textbf{134}
  (2006), 653--663.

\bibitem[CJO10]{CJO10}
Ced\'{o}, F., Jespers, E. , Okni\'{n}ski J., {\em{Retractability of
set theoretic solutions of the Yang Baxter equation}},  Adv. in
Math. \textbf{224} (2010), 2472--2484.

\bibitem[CJO14]{CJO14}
Ced\'{o}, F., Jespers, E. , Okni\'{n}ski J., {\em{Braces and the
Yang--Baxter equation}}, Commun. Math. Phys.  \textbf{327}, (2014) 101--116.



\bibitem[De15]{Dehornoy15}
 Dehornoy, P.,
\emph{Set-theoretic solutions of the Yang-Baxter equation, RC-calculus, and
Garside germs},  Adv. in Math. \textbf{282}  (2015), 93--127.

\bibitem[Dri]{Dri}
Drinfeld, V.,
{\em{On some unsolved problems in quantum group
theory}},
Lecture Notes in Mathematics {\bf{1510}} (1992), 1--8.

\bibitem[ESS]{ESS}
Etingof, P., Schedler, T.\ and Soloviev, A.,
{\em{Set-theoretical solutions to the quantum
Yang--Baxter equation}},
Duke Math.\ J.\ {\bf{100}}  (1999), 169--209.


\bibitem[GI94]{GI94}
Gateva-Ivanova, T.,
{\em{Noetherian properties of skew polynomial rings
with binomial relations}},
Trans.\ Amer.\ Math.\ Soc.\ {\bf{343}} (1994), 203--219.

\bibitem[GI96]{GI96}
Gateva-Ivanova, T.,
{\em{Skew polynomial rings with binomial relations}},
J.\ Algebra {\bf{185}} (1996), 710--753.



\bibitem[GIJO]{GIJO}
Gateva-Ivanova, T., E. Jespers, and J. Okninski, {\em{Quadratic algebras of
skew type and the underlying semigroups}}, J. Algebra
\textbf{270} (2003), 635--659.

\bibitem[GI04s]{GI04s}
Gateva-Ivanova, T., {\em Quantum binomial algebras,
Artin-Schelter regular rings, and solutions of the Yang--Baxter
equations}, Serdica Math. J. {\bf 30} (2004), 431--470.

\bibitem[GI04]{GI04}
Gateva-Ivanova, T.,
{\em{A combinatorial approach to the set-theoretic
solutions of the Yang-Baxter equation}},
 J.Math.Phys. {\bf 45} (2004), 3828--3858.


\bibitem[GI11]{GI11}T. Gateva-Ivanova, \emph{Garside structures  on
monoids with quadratic square-free relations}, Algebr Represent
Theor \textbf{14} (2011),
779--802.

\bibitem[GI12]{GI12}
Gateva-Ivanova, T.,
\emph{Quadratic algebras, Yang-Baxter equation, and Artin-
Schelter regularity}, Adv. in Math. \textbf{230}  (2012),
2152--2175.

\bibitem[GI15]{GI15} Tatiana Gateva-Ivanova, {\em Set-theoretic solutions of the Yang-Baxter equation, Braces, and Symmetric groups},
    ArXiv:1507.02602 [math.QA] v2. Mon, 31 Aug 2015, 1--47.



\bibitem[GIB]{TM}
Gateva-Ivanova, T.\ and Van den Bergh, M.,
{\em{Semigroups of $I$-type}},
J.\ Algebra {\bf{206}} (1998), 97--112.

\bibitem[GIC12]{GIC}
Gateva-Ivanova, T.\ and Cameron, P.J., {\em{Multipermutation
solutions of the  Yang--Baxter equation}},
Comm. Math. Phys. \textbf{309} (2012), 583--621.

\bibitem[GIM07]{GIM07}
Gateva-Ivanova, T.\ and Majid, S., {\em{Set Theoretic
Solutions of the Yang--Baxter Equations, Graphs and Computations}},
 J. Symb. Comp. 42 (2007), 1079--1112.

\bibitem[GIM08]{GIM08}
Gateva-Ivanova, T.\ and Majid, S., {\em{Matched pairs
approach to set theoretic solutions of the Yang--Baxter equation}},
J. Algebra  {\bf{319}} (2008), 1462--1529.


\bibitem[GIM11]{GIM11}  {\sc{T. Gateva-Ivanova, S. Majid}}, \emph{Quantum
    spaces
associated to multipermutation solutions of
   level two},   Algebr. Represent. Theor. \textbf{14} (2011) 341--376




\bibitem[JO04]{JO04}
 Jespers, E. , Okni\'{n}ski J., {\em{Quadratic algebras of skew type satisfying the cyclic condition}}, Int. J. Algebra and Computation
 \textbf{14} (2004), 479--498.




\bibitem[KP66]{KP66}  Kac, G.I., Paljutkin, V.G., {\em Finite ring groups}, Trans. Amer. Math.
Soc. {\bf 15} (1966) 251--294.

\bibitem[Lam]{Lam}  Lam, T. Y. (2001), \emph{A first course in noncommutative rings}, Graduate Texts in Mathematics \textbf{131} (2
    ed.),
    New York: Springer-Verlag.

\bibitem[LV16]{LV16} Lebed, V., Vendramin, L.
{\em Cohomology and extensions of braces},  arXiv preprint  arXiv:1601.01633 [math.GR], 2016
\bibitem[LYZ]{LYZ}
Lu, J., Yan, M.,\ and Zhu, Y., {\em On the set-theoretical Yang--Baxter
  equation}, Duke Math. J. {\bf 104} (2000), 1--18.

\bibitem[Ma90]{Ma90}
Majid, S.,
{\em Matched pairs of {L}ie groups associated to solutions of the
  {Y}ang-{B}axter equations},
Pac. J. Math. {\bf 141} (1990), 311--332.

\bibitem[Ma90a]{Ma90a}
Majid, S.,
{\em Physics for algebraists: Non-commutative and non-cocommutative Hopf
algebras by a bicrosspoduct construction},
J. Algebra, {\bf 130} (1990) 17--64.

\bibitem[Ma90b]{Ma90b}
Majid, S.,
{\em More examples of bicrossproduct and double cross product Hopf algebras},
Israel J. Math, {\bf 72} (1990), 133--148.

\bibitem[Ma95]{Ma:book}
Majid, S.,
{\em Foundations of quantum group theory}, Cambridge Univ. Press (1995).

\bibitem[Man]{Man}
Manin, Yu., {\em Quantum Groups and Non Commutative
Geometry} Montreal University Report No. CRM- 1561, 1988.


\bibitem[Mat]{Matsumoto}
Matsumoto  D.K., {\em Dynamical braces and dynamical
Yang-Baxter maps},  J. Pure Appl. Algebra
\textbf{217} (2013), 195--206.

\bibitem[MS]{MS}
Matsumoto  D.K.,  and Shibukawa, Y,
{\em Quantum Yang-Baxter equation, braided semigroups, and dynamical Yang-Baxter maps,}
    Tokyo J. of Math.
    \textbf{38}, (2015), 227--237.


\bibitem[Ru05]{Ru05}
Rump, W., {\em A decomposition theorem for
square-free unitary
solutions of the quantum Yang--Baxter equation},
Adv.
Math. {\bf 193} (2005), 40--55.
\bibitem[Ru06]{Ru06}
Rump, W., {\em  Modules over braces},  Algebra Discr. Math. {\bf 2} (2006), 127--137
\bibitem[Ru07]{Ru07}
Rump,W., {\em Braces, radical rings, and the quantum Yang-Baxter equation}, J.
Algebra \textbf{307} (2007), 153--170
\bibitem[Ru14]{Ru14}
Rump, W. {\em The brace of a classical group}, Note Mat. \textbf{34} (2014),
115--144.

\bibitem[Smok]{Smok}
Smoktunowicz, A.,
{\em On Engel groups, nilpotent groups, rings, braces and the Yang-Baxter equation
},   arXiv:1509.00420, Sept. 2015

\bibitem[Ta81]{Tak1}
Takeuchi, M., {\em Matched pairs of groups and bismash products of
Hopf algebras}, Commun. Alg. {\bf 9} (1981),  841.


\bibitem[Ta]{Takeuchi}
Takeuchi, M., {\em Survey on matched pairs of groups. An elementary
approach to the ESS-LYZ theory}, Banach Center Publ. {\bf 61} (2003),
305--331.

\bibitem[Ve]{V15}
Vendramin, L., {\em Extensions of set-theoretic solutions of the Yang-Baxter equation and a conjecture of Gateva-Ivanova},
J. Pure Appl. Algebra, {\bf 220} (2016) 2064�-2076 .

\bibitem[WX]{W}
Weinstein, A.\ and Xu, P., {\em Classical solutions of
  the quantum Yang--Baxter equation}, Comm. Math. Phys. {\bf 148}
  (1992), ~309--343.


\end{thebibliography}
\end{document}